\documentclass[oneside,english]{amsbook}
\usepackage[T1]{fontenc}
\usepackage[latin9]{inputenc}
\usepackage{array}
\usepackage{longtable}
\usepackage{amsthm}
\usepackage{amstext}
\usepackage{amssymb}
\usepackage{esint}

\makeatletter

\providecommand{\tabularnewline}{\\}

\numberwithin{section}{chapter}
\numberwithin{equation}{section}
\numberwithin{figure}{section}
\theoremstyle{plain}
\ifx\thechapter\undefined
\newtheorem{thm}{Theorem}
\else
\newtheorem{thm}{Theorem}[chapter]
\fi
  \theoremstyle{definition}
  \newtheorem{defn}[thm]{Definition}
  \theoremstyle{remark}
  \newtheorem{rem}[thm]{Remark}
  \theoremstyle{plain}
  \newtheorem{lem}[thm]{Lemma}
  \theoremstyle{remark}
  \newtheorem*{rem*}{Remark}
  \theoremstyle{plain}
  \newtheorem{cor}[thm]{Corollary}
 \theoremstyle{definition}
  \newtheorem{example}[thm]{Example}
  \theoremstyle{definition}
  \newtheorem{xca}[thm]{Exercise}
  \theoremstyle{plain}
  \newtheorem{prop}[thm]{Proposition}
  \theoremstyle{definition}
  \newtheorem*{example*}{Example}
  \theoremstyle{remark}
  \newtheorem*{note*}{Note}
  \theoremstyle{remark}
  \newtheorem{note}[thm]{Note}
  \theoremstyle{definition}
  \newtheorem*{xca*}{Exercise}
  \theoremstyle{plain}
  \newtheorem*{cor*}{Corollary}
  \theoremstyle{plain}
  \newtheorem{conjecture}[thm]{Conjecture}
  \newcounter{casectr}
  \newenvironment{caseenv}
  {\begin{list}{{\itshape\ Case} \arabic{casectr}.}{%
   \setlength{\leftmargin}{\labelwidth}
   \addtolength{\leftmargin}{\parskip}
   \setlength{\itemindent}{\listparindent}
   \setlength{\itemsep}{\medskipamount}
   \setlength{\topsep}{\itemsep}}
   \setcounter{casectr}{0}
   \usecounter{casectr}}
  {\end{list}}

\usepackage{amsthm}
\usepackage[all]{xy}
\usepackage{xargs}

\providecommand{\norm}[1]{\lVert#1\rVert}
\providecommand{\abs}[1]{\lvert#1\rvert}

\providecommand{\fourier}[1]{\hat{#1}}
\providecommand{\iprod}[2]{\langle #1, #2\rangle}

\providecommand{\bra}[1]{\left\langle #1\right|}
\providecommand{\ket}[1]{\left| #1\right\rangle}
\providecommand{\braket}[2]{\langle #1, #2\rangle}
\providecommand{\ketbra}[2]{\left| #1 \rangle \langle #2 \right|}

\makeatother

\usepackage{babel}

\begin{document}
\begin{longtable}{c}
\textbf{\Huge Functional Analysis}\tabularnewline
\tabularnewline
\tabularnewline
{\Large Feng Tian, and Palle Jorgensen}\tabularnewline
\tabularnewline
\tabularnewline
\tabularnewline
\tabularnewline
Department of Mathematics\tabularnewline
14 MLH \tabularnewline
The University of Iowa\tabularnewline
Iowa City, IA 52242-1419\tabularnewline
 USA.\tabularnewline
\end{longtable}

\pagebreak{}

Notes from a course taught by Palle Jorgensen in
the fall semester of 2009. The course covered central themes in functional
analysis and operator theory, with an emphasis on topics of special
relevance to such applications as representation theory, harmonic
analysis, mathematical physics, and stochastic integration.

\pagebreak{}

These are the lecture notes I took from a topic course taught by Professor
Jorgensen during the fall semester of 2009. The course started with
elementary Hilbert space theory, and moved very fast to spectral theory,
completely positive maps, Kadison-Singer conjecture, induced representations,
self-adjoint extensions of operators, etc. It contains a lot of motivations
and illuminating examples. 

I would like to thank Professor Jorgensen for teaching such a wonderful
course. I hope students in other areas of mathematics would benefit
from these lecture notes as well. 

Unfortunately, I have not been able to fill in all the details. The
notes are undergoing editing. I take full responsibility for any errors
and missing parts.\\
\\

Feng Tian

03. 2010\pagebreak{}

\tableofcontents{}

\chapter{Elementary Facts}

\section{Transfinite induction}

Let $(X,\leq)$ be a paritially ordered set. A sebset $C$ of $X$
is said to be a chain, or totally ordered, if $x,y$ in $C$ implies
that either $x\leq y$ or $y\leq x$. Zorn's lemma says that if every
chain has a majorant then there exists a maximal element in $X$. 
\begin{thm}
(Zorn) Let $(X,\leq)$ be a paritially ordered set. If every chain
$C$ in $X$ has a majorant (or upper bound), then there exists an
element $m$ in $X$ so that $x\ge m$ implies $x=m$, for all $x$
in $X$.
\end{thm}
An illuminating example of a partially ordered set is the binary tree
model. Another example is when $X$ is a family of subsets of a given
set, partially ordered by inclusion. Zorn's lemma lies at the foundation
of set theory. It is in fact an axiom and is equivalent to the axiom
of choice and Hausdorff's maximality principle. 
\begin{thm}
(Hausdorff Maximality Principle) Let $(X,\leq)$ be a paritially ordered
set, then there exists a maximal totally ordered subset $L$ in $X$. 
\end{thm}
The axiom of choice is equivalent to the following statement on infinite
product, which itself is extensively used in functional analysis.
\begin{thm}
\label{thm:aoc}(axiom of choice) Let $A_{\alpha}$ be a family of
nonempty sets indexed by $\alpha\in I$. Then the infinite Cartesian
product $\Omega=\prod_{\alpha}A_{\alpha}$ is nonempty.
\end{thm}
$\Omega$ can be seen as the set of functions $\{(x_{\alpha}):x_{\alpha}\in A_{\alpha}\}$
from $I$ to $\cup A_{\alpha}$. The point of using the axiom of choice
is that if the index set is uncountable, there is no way to verify
whether $(x_{\alpha})$ is in $\Omega$ or not. It is just impossible
to check for each $\alpha$ that $x_{\alpha}$ in contained in $A_{\alpha}$,
for some coordinates will be unchecked. The power of transfinite induction
is that it applies to uncountable sets as well. In case the set is
countable, we simply apply the down to earth standard induction. The
standard mathematical induction is equivalent to the Peano's axiom
which states that every nonempty subset of of the set of natural number
has a unique smallest element.

The key idea in applications of the transfinite induction is to cook
up in a clear way a partially ordered set, so that the maximum element
turns out to be the object to be constructed. Examples include Hahn-Banach
extension theorem, Krein-Millman's theorem on compact convex set,
existance of orthonoral basis in Hilbert space, Tychnoff's theorem
on infinite Cartesian product of compact spaces, where the infinite
product space is nonempty follows imediately from the axiom of choice. 
\begin{thm}
(Tychonoff) Let $A_{\alpha}$ be a family of compact sets indexed
by $\alpha\in I$. Then the infinite Cartesian product $\prod_{\alpha}A_{\alpha}$
in compact with respect to the product topology.
\end{thm}
We will apply the transfinite induction to show that every infinite
dimensional Hilbert space has an orthonormal basis (ONB). 

Classical functional analysis roughly divides into two branches
\begin{itemize}
\item study of function spaces (Banach space, Hilbert space)
\item applications in physics and engineering
\end{itemize}
Within pure mathematics, it is manifested in 
\begin{itemize}
\item representation theory of groups and algebras
\item $C^{*}$-algebras, Von Neumann algebras
\item wavelets theory
\item harmonic analysis
\item analytic number theory\end{itemize}
\begin{defn}
Let $X$ be a vector space over $\mathbb{C}$. $\norm{\cdot}$ is
a norm on $X$ if for all $x,y$ in $X$ and $c$ in $\mathbb{C}$
\begin{itemize}
\item $\norm{cx}=c\norm{x}$
\item $\norm{x}\geq0$; $\norm{x}=0$ implies $x=0$
\item $\norm{x+y}\leq\norm{x}+\norm{y}$
\end{itemize}
$X$ is a Banach space if it is complete with respect to the metric
induced by $\norm{\cdot}$. 
\begin{defn}
Let $X$ be vector space over $\mathbb{C}$. An inner product is a
function $\iprod{\cdot}{\cdot}:X\times X\rightarrow\mathbb{C}$ so
that for all $x,y$ in $H$ and $c$ in $\mathbb{C}$,
\begin{itemize}
\item $\iprod{x}{\cdot}$ is linear (linearity)
\item $\iprod{x}{y}=\overline{\iprod{y}{x}}$ (conjugation) 
\item $\iprod{x}{x}\geq0$; and $\iprod{x}{x}=0$ implies $x=0$ (positivity)
\end{itemize}
\end{defn}
In that case $\sqrt{\iprod{x}{x}}$ defines a norm on $H$ and is
denote by $\norm{x}$. $X$ is said to be an inner product space is
an inner product is defined. A Hilbert space is a complete inner product
space. 

\end{defn}
\begin{rem}
The abstract formulation of Hilbert was invented by Von Neumann in
1925. It fits precisely with the axioms of quantum mechanics (spectral
lines, etc.) A few years before Von Neumann's formulation, Heisenberg
translated Max Born's quantum mechanics into mathematics.
\end{rem}
For any inner product space $H$, observe that the matrix \[
\left[\begin{array}{cc}
\iprod{x}{x} & \iprod{x}{y}\\
\iprod{y}{x} & \iprod{y}{y}\end{array}\right]\]
is positive definite by the positivity axiom of the definition of
an inner product. Hence the matrix has positive determinant, which
gives rises to the famous Cauchy-Schwartz inequality \[
\abs{\iprod{x}{y}}^{2}\leq\iprod{x}{x}\iprod{y}{y}.\]

An extremely useful way to construct a Hilbert space is the GNS construction,
which starts with a semi-positive definite funciton defined on a set
$X$. $\varphi:X\times X\rightarrow\mathbb{C}$ is said to be semi-positive
definite, if for finite collection of complex nubmers $\{c_{x}\}$,
\[
\sum\bar{c}_{x}c_{y}\varphi(x,y)\geq0.\]
Let $H_{0}$ be the span of $\delta_{x}$ where $x\in X$, and define
a sesiquilinear form $\iprod{\cdot}{\cdot}$ on $H_{0}$ as \[
\iprod{\sum c_{x}\delta_{x}}{\sum c_{y}\delta_{y}}:=\sum\bar{c}_{x}c_{y}\varphi(x,y).\]
However, the positivity condition may not be satisfied. Hence one
has to pass to a quotient space by letting $N=\{f\in H_{0},\iprod{f}{f}=0\}$,
and $\tilde{H}_{0}$ be the quotient space $H_{0}/N$. The fact that
$N$ is really a subspace follows from the Cauchy-Schwartz inequality
above. Therefore, $\iprod{\cdot}{\cdot}$ is an inner product on $\tilde{H}_{0}$.
Finally, let $H$ be the completion of $\tilde{H}_{0}$ under $\iprod{\cdot}{\cdot}$
and $H$ is a Hilbert space.
\begin{defn}
Let $H$ be a Hilbert space. A family of vectors $\{u_{\alpha}\}$
in $H$ is said to be an orthonormal basis of $H$ if 
\begin{enumerate}
\item $\iprod{u_{\alpha}}{u_{\beta}}=\delta_{\alpha\beta}$ and 
\item $\overline{span}\{u_{\alpha}\}=H$.
\end{enumerate}
We are ready to prove the existance of an orthonormal basis of a Hilbert
space, using transfinite induction. Again, the key idea is to cook
up a partially ordered set satisfying all the requirments in the transfinite
induction, so that the maximum elements turns out to be an orthonormal
basis. Notice that all we have at hands are the abstract axioms of
a Hilbert space, and nothing else. Everything will be developed out
of these axioms.\end{defn}
\begin{thm}
\label{thm:ONB}Every Hilbert space $H$ has an orthonormal basis.
\end{thm}
To start out, we need the following lemmas. 
\begin{lem}
\label{lem:dense}Let $H$ be a Hilbert space and $S\subset H$. Then
the following are equivalent:
\begin{enumerate}
\item $x\perp S$ implies $x=0$
\item $\overline{span}\{S\}=H$\end{enumerate}
\begin{lem}
\label{lem:Gram-Schmidt}(Gram-Schmidt) Let $\{u_{n}\}$ be a sequence
of linearly independent vectors in $H$ then there exists a sequence
$\{v_{n}\}$ of unit vectors so that $\iprod{v_{i}}{v_{j}}=\delta_{ij}$.
\begin{rem*}
The Gram-Schmidt orthogonalization process was developed a little
earlier than Von Neumann's formuation of abstract Hilbert space.
\end{rem*}
\end{lem}
\end{lem}
\begin{proof}
we now prove theorem (\ref{thm:ONB}). If $H$ is empty then we are
finished. Otherwise, let $u_{1}\in H$. If $\norm{u_{1}}\neq1$, we
may consider $u_{1}/\norm{u_{1}}$ which is a normalized vector. Hence
we may assume $\norm{u_{1}}=1$. If $span\{u_{1}\}=H$ we are finished
again, otherwise there exists $u_{2}\notin span\{u_{1}\}$. By lemma
(\ref{lem:Gram-Schmidt}), we may assume $\norm{u_{2}}=1$ and $u_{1}\perp u_{2}$.
By induction, we get a collection $S$ of orthonormal vectors in $H$. 

Consider $\mathbb{P}(S)$ partially order by set inclusion. Let $C\subset\mathbb{P}(S)$
be a chain and let $M=\cup_{E\in C}E$. $M$ is clearly a majorant
of $C$. We claim that $M$ is in the partially ordered system. In
fact, for all $x,y\in M$ there exist $E_{x}$ and $E_{y}$ in $C$
so that $x\in E_{x}$ and $y\in E_{y}$. Since $C$ is a chain, we
may assume $E_{x}\leq E_{y}$. Hence $x,y\in E_{2}$ and $x\perp y$,
which shows that $M$ is in the partially ordered system.

By Zorn's lemma, there exists a maximum element $m\in S$. It suffices
to show that the closed span of $m$ is $H$. Suppose this is false,
then by lemma (\ref{lem:dense}) there exists $x\in H$ so that $x\perp M$.
Since $m\cup\{x\}\geq m$ and $m$ is maximal, it follows that $x\in m$,
which implies $x\perp x$. By the positivity axiom of the definition
of Hilbert space, $x=0$.\end{proof}
\begin{cor}
Let $H$ be a Hilbert space, then $H$ is isomorphic to the $l^{2}$
space of the index set of an ONB of $H$.
\begin{rem}
There seems to be just one Hilbert space, which is true in terms of
the Hilbert space structure. But this is misleading, because numerous
interesting realizations of an abstract Hilbert space come in when
we make a choice of the ONB. The question as to which Hilbert space
to use is equivalent to a good choice of an ONB. This is simiar to
the argument that there is just one set for each given cardinality
in terms of set structure, but there are numerous choices of elements
in the sets making questions interesting.

Suppose $H$ is separable, for instance let $H=L^{2}(\mathbb{R})$.
Then $H\cong l^{2}(\mathbb{N})\cong l^{2}(\mathbb{N}\times\mathbb{N})$.
It follows that potentially we could choose a doublely indexed basis
$\{\psi_{jk}:j,k\in\mathbb{N}\}$ for $L^{2}$. It turns out that
this is precisely the setting of wavelet basis! What's even better
is that in $l^{2}$ space, there are all kinds of diagonalized operators,
which correspond to self-adjoint (or normal) operators in $L^{2}$.
Among these operators in $L^{2}$, we single out the scaling ($f(x)\mapsto f(2^{j}x)$)
and translation ($f(x)\mapsto f(x-k)$) operators, which are diagonalized,
NOT \emph{simultaneously} though.
\end{rem}
\end{cor}

\subsection{path space measures}

Let $\Omega=\prod_{k=1}^{\infty}\{1,-1\}$ be the infinite Cartesian
product of $\{1,-1\}$ with the product topology. $\Omega$ is compact
and Hausdorff by Tychnoff's theorem. 

For each $k\in\mathbb{N}$, let $X_{k}:\Omega\rightarrow\{1,-1\}$
be the $k^{th}$ coordinate projection, and assign probability measures
$\mu_{k}$ on $\Omega$ so that $\mu_{k}\circ X_{k}^{-1}\{1\}=a$
and $\mu_{k}\circ X_{k}^{-1}\{-1\}=1-a$, where $a\in(0,1)$. The
collection of measures $\{u_{k}\}$ satisfies the consistency conditiond,
i.e. $\mu_{k}$ is the restriction of $\mu_{k+1}$ onto the $k^{th}$
coordinate space. By Kolomogorov's extension theorem, there exists
a unique probability measure $P$ on $\Omega$ so that the restriction
of $P$ to the $k^{th}$ coordinate is equal to $\mu_{k}$. 

It follows that $\{X_{k}\}$ is a sequence of independent identically
distributed (i.i.d.) random variables in $L^{2}(\Omega,P)$ with $\mathbb{E}[X_{k}]=0$
and $Var[X_{k}]=1$; and $L^{2}(\Omega,P)=\overline{span}\{X_{k}\}$.

Let $H$ be a separable Hilbert space with an orthonormal basis $\{u_{k}\}$.
The map $\varphi:u_{k}\mapsto X_{k}$ extends linearly to an isometric
embedding of $H$ into $L^{2}(\Omega,P)$. Moreover, let $\mathcal{F}_{+}(H)$
be the symmetric Fock space. $\mathcal{F}_{+}(H)$ is the closed span
of the the algebraic tensors $u_{k_{1}}\otimes\cdots\otimes u_{k_{n}}$,
thus $\varphi$ extends to an isomorphism from $\mathcal{F}_{+}(H)$
to $L^{2}(\Omega,P)$.

\section{Dirac's notation}

P.A.M Dirac was every efficient with notations, and he introduced
the {}``bra-ket'' vectors. Let $H$ be a Hilbert space with inner
product $\iprod{\cdot}{\cdot}:H\times H\rightarrow\mathbb{C}$. We
denote by {}``bra'' for vectors $\bra{x}$ and {}``ket'' for vectors
$\ket{y}$ where $x,y\in H$.

With Dirac's notation, our first observation is the followsing lemma.
\begin{lem}
Let $v\in H$ be a unit vector. The operator $x\mapsto\iprod{v}{x}v$
can be written as $P_{v}=\ketbra{v}{v}$. $P_{v}$ is a rank-one self-adjoint
projection.\end{lem}
\begin{proof}
$P_{v}^{2}=(\ketbra{v}{v})(\ketbra{v}{v})=\ketbra{v}{v}=P_{v}$. Since
\[
\iprod{x}{P_{v}y}=\iprod{x}{v}\iprod{v}{y}=\iprod{\overline{\iprod{x}{v}}v}{y}=\iprod{\iprod{v}{x}v}{y}=\iprod{P_{v}x}{y}\]
so $P_{v}=P_{v}^{*}$.
\end{proof}
More generally, any rank-one operator can be wrritten as $\ketbra{u}{v}$
sending $x\in H$ to $<v,x>u$. With the bra-ket notation, it's easy
to verify that the set of rank-one operators forms an algebra, which
easily follows from the fact that $(\ketbra{v_{1}}{v_{2}})(\ketbra{v_{3}}{v_{4}})=\ketbra{v_{1}}{v_{4}}$.
$ $The moment that an orthonormal basis is selected, the algebra
of operators on $H$ will be translated to the algebra of matrices
(infinite). Every Hilbert space has an ONB, but it does not mean in
pratice it is easy to select one that works well for a particular
problem.

It's also easy to see that the operator \[
P_{F}=\sum_{v_{i}\in F}\ketbra{v_{i}}{v_{i}}\]
where $F$ is a finite set of orthonormal vectors in $H$, is a self-adjoint
projection. This follows, since\[
P_{F}^{2}=\sum_{v_{i},v_{j}\in F}(\ketbra{v_{i}}{v_{i}})(\ketbra{v_{j}}{v_{j}})=\sum_{v_{i}\in F}\ketbra{v_{i}}{v_{i}}\]
and $P_{F}^{*}=P_{F}$.

The Gram-Schmidt orthogonalization process may now be written in Dirac's
notation so that the induction step is really just \[
\frac{x-P_{F}x}{\norm{x-P_{F}x}}\]
which is a unit vector and orthogonal to $P_{F}H$. Notice that if
$H$ is non separable, the standard induction does not work, and the
transfinite induction is needed.

\subsection{connection to quamtum mechanics}

Quamtum mechanics was born during the years from 1900 to 1913. It
was created to explain phenomena in black body radiation, hydrogen
atom, where a discrete pattern occurs in the frequences of waves in
the radiation. The radiation energy $E=\nu\hbar$, with $\hbar$ being
the Plank's constant. Classical mechanics runs into trouble. 

During the years of 1925\textasciitilde{}1926, Heisenberg found a
way to represent the energy $E$ as a matrix, so that the matrix entries
$<v_{j},Ev_{i}>$ represents the transition probability from energy
$i$ to energy $j$. A foundamental relation in quantum mechenics
is the commutation relation satisfied by the momentum operator $P$
and the position operator $Q$, where \[
PQ-QP=\frac{1}{i}I.\]
Heisernberg represented the operators $P,Q$ by matrices, although
his solution is not real matrices. The reason is for matrices, there
is a trace operation where $trace(AB)=trace(BA)$. This implies the
trace on the left-hand-side is zero, while the trace on the righ-hand-side
is not. This suggests that there is no finite dimensional solution
to the commutation relation above, and one is forced to work with
infinite dimensional Hilbert space and operators on it. Notice also
that $P,Q$ do not commute, and the above commutation relation leads
to the uncertainty principle (Hilbert, Max Born, Von Neumann worked
out the mathematics), which says that the statistical variance $\triangle P$
and $\triangle Q$ satisfy $\triangle P\triangle Q\geq\hbar/2$ .
We will come back to this later.

However, Heisenberg found his {}``matrix'' solutions, where \[
P=\left[\begin{array}{ccccc}
0 & 1\\
1 & 0 & \sqrt{2}\\
 & \sqrt{2} & 0 & \sqrt{3}\\
 &  & \sqrt{3} & 0 & \ddots\\
 &  &  & \ddots & \ddots\end{array}\right]\]
 and \[
Q=\frac{1}{i}\left[\begin{array}{ccccc}
0 & 1\\
-1 & 0 & \sqrt{2}\\
 & -\sqrt{2} & 0 & \sqrt{3}\\
 &  & -\sqrt{3} & 0 & \ddots\\
 &  &  & \ddots & \ddots\end{array}\right]\]
the complex $i$ in front of $Q$ is to make it self-adjoint.

A selection of ONB makes a connection to the algebra operators acting
on $H$ and infinite matrices. We check that using Dirac's notation,
the algebra of operators really becomes the algebr of infinite matrices.

Pick an ONB $\{u_{i}\}$ in $H$, $A,B\in B(H)$. We denote by $M_{A}=A_{ij}:=\iprod{u_{i}}{Au_{j}}$
the matrix of $A$ under the ONB. We compute $\iprod{u_{i}}{ABu_{j}}$.\begin{eqnarray*}
(M_{A}M_{B})_{ij} & = & \sum_{k}A_{ik}B_{kj}\\
 & = & \sum_{k}\iprod{u_{i}}{Au_{k}}\iprod{u_{k}}{Bu_{j}}\\
 & = & \sum_{k}\iprod{A^{*}u_{i}}{u_{k}}\iprod{u_{k}}{Bu_{j}}\\
 & = & \iprod{A^{*}u_{i}}{Bu_{j}}\\
 & = & \iprod{u_{i}}{ABu_{j}}\end{eqnarray*}
where $I=\sum\ketbra{u_{i}}{u_{i}}$.

Let $w$ be a unit vector in $H$. $w$ represents a quantum state.
Since $\norm{w}^{2}=\sum\abs{\iprod{u_{i}}{w}}^{2}=1$, the numbers
$\abs{\iprod{u_{i}}{w}}^{2}$ represent a probability distribution
over the index set. If $w$ and $w'$ are two states, then \[
\iprod{w}{w'}=\sum\iprod{w}{u_{i}}\iprod{u_{i}}{w'}\]
which has the interpretation so that the transition from $w'$ to
$w$ may go through all possible intermediate states $u_{i}$. Two
states are uncorrelated if and only if they are orthogonal.

\section{Operators in Hilbert space}
\begin{defn}
Let $A$ be a linear operator on a Hilbert space $H$. 
\begin{enumerate}
\item $A$ is self-adjoint if $A^{*}=A$
\item $A$ is normal if $AA^{*}=A^{*}A$
\item $A$ is unitary if $AA^{*}=A^{*}A=I$
\item $A$ is a self-adjoint projection if $A=A^{*}=A^{2}$
\end{enumerate}
\end{defn}
Let $A$ be an operator, then we have $R=(A+A^{*})/2$, $S=(A-A^{*})/2i$
which are both self-adjoint, and $A=R+iS$. This is similar the to
decomposition of a complex nubmer into its real and imaginary parts.
Notice also that $A$ is normal if and only if $R$ and $S$ commute.
Thus the study of a family of normal operators is equivalent to the
study of a family of commuting self-adjoint operators. 
\begin{lem}
Let $z$ be a complex number, and $P$ be a self-adjoint projection.
Then $U(z)=zP+(I-P)$ is unitary if and only if $\abs{z}=1$.\end{lem}
\begin{proof}
Since $P$ is a self-adjoint projection, \[
U(z)U(z)^{*}=U(z)^{*}U(z)=\left(zP+(I-P)\right)\left(\bar{z}P+(I-P)\right)=\abs{z}^{2}P+(I-P).\]
If $\abs{z}=1$ then $U(z)U(z)^{*}=U(z)^{*}U(z)=I$ and $U(z)$ is
unitary. Conversely, if $\abs{z}^{2}P+(I-P)=I$ then it follows that
$\left(\abs{z}^{2}-1\right)P=0$. If we assume that $P$ is nondegenerate,
then $\abs{z}=1$.\end{proof}
\begin{defn}
Let $A$ be a linear operator on a Hilbert space $H$. The resolvent
$R(A)$ is defined as\[
R(A)=\{\lambda\in\mathbb{C}:(\lambda I-A)^{-1}\text{ exists}\}\]
and the spectrum of $A$ is the complement of $R(A)$, and it is denoted
by $sp(A)$ or $\sigma(A)$. 
\begin{defn}
Let $\mathfrak{B}(\mathbb{C})$ be the Borel $\sigma$-algebra of
$\mathbb{C}$. $H$ is a Hilbert space. $P:\mathfrak{B}(\mathbb{C})\rightarrow H$$ $
is a projection-valued measure, if 
\begin{enumerate}
\item $P(\phi)=0$, $P(\mathbb{C})=I$
\item $P(A\cap B)=P(A)P(B)$
\item $P(\sum E_{k})=\sum P(E_{k})$, $E_{k}\cap E_{j}=\phi$ if $k\neq j$.
The convergence is in terms of the strong operator topology.
\end{enumerate}
\end{defn}
\end{defn}
Von Neumann's spectral theorem states that an operator $A$ is normal
if and only if there exits a projection-valued measure on $\mathbb{C}$
so that $A=\int_{sp(A)}zP(dz)$, i.e. $A$ is represented as an integral
again the projection-valued measure $P$ over its spectrum. 

In quamtum mechanics, an observable is represented by a self-adjoint
operator. Functions of observables are again observables. This is
reflected in the spectral theorem as the functional calculus, where
we may define $f(A)=\int_{sp(A)}f(z)P(dz)$ using the spectral representation
of $A$. 

The stardard diagonalization of Hermitian matrix in linear algebra
is a special case of the spectral theorem. Recall that if $A$ is
a Hermitian matrix, then $A=\sum_{k}\lambda_{k}P_{k}$ where $\lambda_{k}'s$
are the eigenvalues of $A$ and $P_{k}'s$ are the self-adjoint projections
onto the eigenspace associated with $\lambda_{k}'s$. The projection-valued
measure in this case can be written as $P(E)=\sum_{\lambda_{k}\in E}P_{k}$,
i.e. the counting measure supported on $\lambda_{k}'s$.

Hersenberg's commutation relation $PQ-QP=-iI$ is an important example
of two non-commuting self-adjoint operators. When $P$ is a self-adjoint
projection acting on a Hilbert space $H$, $\braket{f}{Pf}$ is a
real number and it represents observation of the observable $P$ prepared
in the state $\ket{f}$. Quantum mechanics is stated using an abstract
Hilbert space as the state space. In practice, one has freedom to
choose exactly which Hilbert space to use for a particular problem.
The physics remains to same when choosing diffenent realizations of
a Hilbert space. The concept needed here is unitary equivalence.

Suppose $U:H_{1}\rightarrow H_{2}$ is a unitary operator, $P:H_{1}\rightarrow H_{1}$
is a self-adjoint projection. Then $UPU^{*}:H_{2}\rightarrow H_{2}$
is a self-adjoint projection on $H_{2}$. In fact, $ $$(UPU^{*})(UPU^{*})=UPU^{*}$
where we used $UU^{*}=U^{*}U=I$, as $U$ is unitary. Let $\ket{f_{1}}$
be a state in $H_{1}$ and $\ket{Uf_{1}}$ be the corresponding state
in $H_{2}$. Then\[
\braket{f_{2}}{UPU^{*}f_{2}}=\braket{U^{*}f_{2}}{PU^{*}f_{2}}=\braket{f_{1}}{Pf_{1}}\]
i.e. the observable $P$ has the same expectation value. Since every
self-adjoint operator is, by the spectral theorem, decomposed into
self-adjoint projections, it follows the expectation value of any
observable remains unchanged under unitary transformation.
\begin{defn}
Let $A:H_{1}\rightarrow H_{1}$ and $B:H_{2}\rightarrow H_{2}$ be
operators. $A$ is unitarily equivalent to $B$ is there exists a
unitary operator $U:H_{1}\rightarrow H_{2}$ so that $B=UAU^{*}$.\end{defn}
\begin{example}
Fourier transform $U:L^{2}(\mathbb{R})\rightarrow L^{2}(\mathbb{R})$,
\[
(Uf)(t)=\fourier{f}(t)=\frac{1}{\sqrt{2\pi}}\int e^{-itx}f(x)dx.\]
The operators $Q=M_{x}$ and $P=-id/dx$ are both densely defined
on the Schwartz space $\mathcal{S}\subset L^{2}(\mathbb{R})$. $P$
and $Q$ are unitary equivalent via the Fourier transform, \[
P=\mathcal{F}^{*}Q\mathcal{F}.\]
Apply Gram-Schmidt orthogonalization to polynomials against the measrue
$e^{-x^{2}/2}dx$, and get orthognoal polynomials. There are the Hermite
polynomials (Hermit functions). \[
h_{n}=e^{-x^{2}/2}P_{n}=e^{-x^{2}}\left(\frac{d}{dx}\right)^{n}e^{x^{2}/2}.\]
The Hermite functions form an orthonormal basis (normalize it) and
transform $P$ and $Q$ to Heisenberg's infinite matrices. Some related
operators: $H:=(Q^{2}+P^{2}-1)/2$. It can be shown that\[
Hh_{n}=nh_{n}\]
or equivalently,\[
(P^{2}+Q^{2})h_{n}=(2n+1)h_{n}\]
$n=0,1,2,\ldots$. $H$ is called the energy operator in quantum mechanics.
This explains mathematically why the energy levels are discrete, being
a multiple of $\hbar$.
\end{example}
A multiplication operator version is also available which works especially
well in physics. It says that $A$ is a normal operator in $H$ if
and only if $A$ is unitarily equivalent to the operator of multiplication
by a measurable function $f$ on $L^{2}(M,\mu)$ where $M$ is compact
and Hausdorff. We will see how the two versions of the spectral theorem
are related after first introduing the concept of transformation of
measure.

\subsection{Transformation of measure}

Let $(X,S)$ and $(Y,T)$ be two measurable spaces with $\sigma$-algebras
$S$ and $T$ respectively. Let $\varphi:X\rightarrow Y$ be a measurable
function. Suppose there is a measure $\mu$ on $(X,S)$. Then $\mu_{\varphi}(\cdot):=\mu\circ\varphi^{-1}(\cdot)$
defines a measure on $Y$. $\mu_{\varphi}$ is the transformation
measure of $\mu$ under $\varphi$.

Notice that if $E\in T$, then $\varphi(x)\in E$ if and only if $x\in\varphi^{-1}(E)\in S$.
Hence \[
\chi_{E}\circ\varphi(x)=\chi_{\varphi^{-1}(E)}.\]
It follows that for simple function $s(\cdot)=\sum c_{i}\chi_{E_{i}}(\cdot)=\sum c_{i}\chi_{\varphi^{-1}(E_{i})}$,
and \[
\int s\circ\varphi d\mu=\int\sum c_{i}\chi_{E_{i}}\circ\varphi(x)d\mu=\int\sum c_{i}\chi_{\varphi^{-1}(E_{i})}(\cdot)d\mu=\int sd(\mu\circ\varphi^{-1}).\]
With a standard approximation of measurable functions by simple functions,
we have for any measurable function $f:X\rightarrow Y$,\[
\int f(\varphi(\cdot))d\mu=\int f(\cdot)d(\mu\circ\varphi^{-1}).\]
The above equation is a generalization of the substitution formula
in calculus.

The multiplication version of the spectral theory states that every
normal operator $A$ is unitarily equivalent to the operator of multiplication
by a measurable function $M_{f}$ on $L^{2}(M,\mu)$ where $M$ is
compact and Hausdorff. With transformation of measure, we can go one
step further and get that $A$ is unitarily equivalent to the operator
of multiplication by the independent variable on some $L^{2}$ space.
Notice that if $f$ is nesty, even if $\mu$ is a nice measure (say
the Lebesgue measure), the transformation meaure $\mu\circ f^{-1}$
can still be nesty, it could even be singular.

Let's assume we have a normal operator $M_{\varphi}:L^{2}(M,\mu)\rightarrow L^{2}(M,\mu)$
given by multiplication by a measurable function $\varphi$. Define
an operator $U:L^{2}(\varphi(M),\mu\circ\varphi^{-1})\rightarrow L^{2}(M,\mu)$
by\[
(Uf)(\cdot)=f(\varphi(\cdot)).\]
$U$ is unitary, since \[
\int\abs{f}^{2}d(\mu\circ\varphi^{-1})=\int\abs{f(\varphi)}^{2}d\mu=\int\abs{Uf}^{2}d\mu.\]
Claim also that\[
M_{\varphi}U=UM_{t}.\]
To see this, let $f$ be a $\mu_{\varphi}$-measurable function. Then\begin{eqnarray*}
M_{\varphi}Uf & = & \varphi(t)f(\varphi(t))\\
UM_{t}f & = & U(tf(t))\\
 & = & \varphi(t)f(\varphi(t))\end{eqnarray*}

Recall we have stated two versions of the spectral theorem. (multiplication
operator and projection-valued measure) Consider the simplest case
for the projection-valued measure, where we work with $L^{2}(X,\mu)$.
Claim that $P(E):=M\chi_{E}$, i.e. the operator of multiplication
by $\chi_{E}$ on the Hilbert space $L^{2}(X,\mu)$, is a projection-valued
measure. 

Apply this idea to $P$ and $Q$, the momemtum and positon operators
in quantum mechanics. $P=-id/dx$, $Q=M_{x}$. As we discussed before,
\[
P=\mathcal{F}^{-1}Q\mathcal{F}\]
in other words, $Q$ and $P$ are unitarily equivalent via the Fourier
transform, which diagonalizs $P$. Now we get a projection-valued
measure (PVM) for $P$ by\[
E(\cdot):=\mathcal{F}^{-1}M_{\chi_{\{\cdot\}}}\mathcal{F}.\]
This can be seen as the convolution operator with respect to the inverse
Fourier transform of $\chi_{\{\cdot\}}$.

\section{Lattice structure of projections}

We first show some examples of using Gram-Schmidt orthoganoliztion
to obtain orthonormal bases for a Hilbert space.
\begin{example}
$H=L^{2}[0,1]$. The polynomials $\{1,x,x^{2},\ldots\}$ are linearly
independent in $H$, since if\[
\sum c_{k}x^{k}=0\]
then as an analytic function, the left-hand-side must be identically
zero. By Stone-Weierstrass theorem, $span\{1,x,x^{2},\ldots\}$ is
dense in $C([0,1])$ under the $\norm{\cdot}_{\infty}$ norm. Since
$\norm{\cdot}_{L^{2}}\leq\norm{\cdot}_{\infty}$, it follows that
$span\{1,x,x^{2},\ldots\}$ is also dense in $H$. By Gram-Schmidt,
we get a sequence $\{V_{n}\}$ of finite dimensional subspaces in
$H$, where $V_{n}$ has an orthonormal basis $\{h_{0},\ldots,h_{n-1}\}$,
so that $spanV_{n}=span\{1,x,\ldots,x^{n-1}\}$. Define \[
h_{n+1}=\frac{x^{n+1}-P_{n}x^{n+1}}{\norm{x^{n+1}-P_{n}x^{n+1}}}.\]
The set $\{h_{n}\}$ is dense in $H$, since $span\{h_{n}\}=span\{1,x,\ldots\}$
and the latter is dense in $H$. Therefore, $\{h_{n}\}$ forms an
orthonormal basis of $H$.
\begin{example}
$H=L^{2}[0,1]$. Consider the set of complex exponentials $\{e^{i2\pi nx}\}_{n=0}^{\infty}$.
This is already an ONB for $H$ and leads to Fourier series. Equivalently,
may also consider $\{\cos2\pi nx,\sin2\pi nx\}_{n=0}^{\infty}$.
\end{example}
\end{example}
The next example constructs the Haar wavelet.
\begin{example}
$H=L^{2}[0,1]$. Let $\varphi_{0}$ be the characteristic function
of $[0,1]$. Define $\varphi_{1}=\varphi_{0}(2x)-\varphi_{0}(2x-1)$
and $\psi_{jk}=2^{k/2}\varphi_{1}(2^{k}x-l)$. For fixed $k$ and
$j_{1}\neq j_{2}$, $\iprod{\psi_{jk_{1}}}{\psi_{jk_{2}}}=0$ since
they have disjoint support. \end{example}
\begin{xca}
Let $M_{t}:L^{2}[0,1]\rightarrow L^{2}[0,1]$ be the operator of multiplication
by $t$. Compute the matrix of $M_{t}$ under wavelet basis. (this
is taken from Joel Anderson, who showed the $A=A^{*}$ implies that
$A=D+K$ where $D$ is a diagonal operator and $K$ is a compact perturbation.d)\end{xca}
\begin{thm}
Let $H$ be a Hilbert space. There is a one-to-one correspondence
between self-adjoint projections and closed subspaces of $H$. \end{thm}
\begin{proof}
Let $P$ be a self-adjoint projection in $H$. i.e. $P^{2}=P=P^{*}$.
Then $PH=\{x\in H:Px=x\}$ is a closed subspace. Denote by $P^{\perp}$
the completement of $P$, i.e. $P^{\perp}=1-P$. Then $P^{\perp}H=\{x\in H:P^{\perp}x=x\}=\{x\in H:Px=0\}$.
Since $PP^{\perp}=P(1-P)=P-P^{2}=P-P=0$, therefore $PH\perp P^{\perp}H$. 

Conversely, let $W$ be a closed subspace in $H$. First notice that
the parallelogram law is satisfied in a Hilbert space, where for any
$x,y\in H$, $\norm{x+y}^{2}+\norm{x-y}^{2}=2(\norm{x}^{2}+\norm{y}^{2})$.
Let $x\in H\backslash W$, define $d=\inf_{w\in W}\norm{x-w}$. By
definition, there exists a sequence $w_{n}$ in $W$ so that $\norm{w_{n}-x}\rightarrow0$
as $n\rightarrow\infty$. Apply the parallelogram law to $x-w_{n}$
and $x-w_{m}$,\[
\norm{(x-w_{n})+(x-w_{m})}^{2}+\norm{(x-w_{n})-(x-w_{m})}^{2}=2(\norm{x-w_{n}}^{2}+\norm{x-w_{m}}^{2})\]
which simplies to\[
4\norm{x-\frac{w_{n}+w_{m}}{2}}^{2}+\norm{w_{n}-w_{m}}^{2}=2(\norm{x-w_{n}}^{2}+\norm{x-w_{m}}^{2}).\]
Notice here all we require is $(w_{n}+w_{m})/2$ lying in the subspace
$W$, hence it suffices to require simply that $W$ is a convex subset
in $H$. see Rudin or Nelson page 62 for more details.
\end{proof}
Von Neumann invented the abstract Hilbert space in 1928 as shown in
one of the earliest papers. He work was greatly motivated by quantum
mechanics. In order to express quantum mechanics logic operations,
he created lattices of projections, so that everything we do in set
theory with set operation has a counterpart in the operations of projections.

\begin{longtable}{|c|c|c|c|}
\hline 
SETS & CHAR & PROJECTIONS & DEFINITIONS\tabularnewline
\hline
\hline 
$A\cap B$ & $\chi_{A}\chi_{B}$ & $P\wedge Q$ & $PH\cap PQ$\tabularnewline
\hline 
$A\cup B$ & $\chi_{A\cup B}$ & $P\vee Q$ & $\overline{span}\{PH\cup QH\}$\tabularnewline
\hline 
$A\subset B$ & $\chi_{A}\chi_{B}=\chi_{A}$ & $P\leq Q$ & $PH\subset QH$ \tabularnewline
\hline 
$A_{1}\subset A_{2}\subset\cdots$ & $\chi_{A_{i}}\chi_{A_{i+1}}=\chi_{A_{i}}$ & $P_{1}\leq P_{2}\leq\cdots$ & $P_{i}H\subset P_{i+1}H$\tabularnewline
\hline 
$\bigcup_{k=1}^{\infty}A_{k}$ & $\chi_{\cup_{k}A}$ & $\vee_{k=1}^{\infty}P_{k}$ & $\overline{span}\{\bigcup_{k=1}^{\infty}P_{k}H\}$\tabularnewline
\hline 
$\bigcap_{k=1}^{\infty}A_{k}$ & $\chi_{\cap_{k}A_{k}}$ & $\wedge_{k=1}^{\infty}P_{k}$ & $\bigcap_{k=1}^{\infty}P_{K}H$\tabularnewline
\hline 
$A\times B$ & $\left(\chi_{A\times X}\right)\left(\chi_{X\times B}\right)$ & $P\otimes Q$ & $P\otimes Q\in proj(H\otimes K)$\tabularnewline
\hline
\end{longtable}

$PH\subset QH\Leftrightarrow P=PQ$. This is similar to set operation
where $A\cap B=A\Leftrightarrow A\subset B$. In general, product
and sum of projections are not projections. But if $PH\subset QH$
then the product is in fact a projection. Taking adjoint, one get
$P^{*}=(PQ)^{*}=Q^{*}P^{*}=QP$. It follows that $PQ=QP=P$. i.e.
\textit{containment implies the two projections commute}.

During the same time period as Von Neumann developed his Hilbert space
theory, Lesbegue developed his integration theory which extends the
classical Riemann integral. The motone sequence of sets $A_{1}\subset A_{2}\subset\cdots$
in Lebesgue's integration theory also has a counterpart in the theory
of Hilbert space. To see what happens here, let $P_{1}\leq P_{2}$
and we show that this implies $\norm{P_{1}x}\leq\norm{P_{2}x}$. 
\begin{lem}
$P_{1}\leq P_{2}\Rightarrow\norm{P_{1}x}\leq\norm{P_{2}x}$.\end{lem}
\begin{proof}
It follows from \[
\norm{P_{1}x}^{2}=\iprod{P_{1}x}{P_{1}x}=\iprod{x}{P_{1}x}=\iprod{x}{P_{2}P_{1}x}\leq\norm{P_{1}P_{2}x}^{2}\leq\norm{P_{2}x}^{2}.\]

\end{proof}
As a consequence, $P_{1}\leq P_{2}\leq\cdots$ implies $\norm{P_{k}x}$
forms a monotone increasing sequence in $\mathbb{R}$, and the sequence
is bounded by $\norm{x}$, since $\norm{P_{k}x}\leq\norm{x}$ for
all $k$. Therefore the sequence $P_{k}$ converges to $P$, in symbols\[
\vee P_{k}=\lim_{k}P_{k}=P\]
in the sense that (strongly convergent) for all $x\in H$, there exists
a vector, which we denote by $Px$ so that \[
\lim_{k}\norm{P_{k}x-Px}=0\]
and $P$ really defines a self-adjoint projection.

The examples using Gram-Schmidt can now be formulated in the lattice
of projections. We have $V_{n}$ the $n$-dimensional subspaces and
$P_{n}$ the orthogonal projection onto $V_{n}$, where\begin{eqnarray*}
V_{n}\subset V_{n+1}\rightarrow\cup V_{n} & \sim & P_{n}\leq P_{n+1}\rightarrow P\\
 &  & P_{n}^{\perp}\geq P_{n+1}^{\perp}\rightarrow P^{\perp}.\end{eqnarray*}
Since $\cup V_{n}$ is dense in $H$, it follows that $P=I$ and $P^{\perp}=0$.
We may express this in the lattice notations by\begin{eqnarray*}
\vee P_{n}=\sup P_{n} & = & I\\
\wedge P_{n}^{\perp}=\inf P_{n} & = & 0.\end{eqnarray*}

The tensor product construction fits with composite system in quamtum
mechanics.
\begin{lem}
$P,Q\in proj(H)$. Then $P+Q\in proj(H)$ if and only if $PQ=QP=0$.
i.e. $P\perp Q$.
\begin{proof}
Notice that \[
(P+Q)^{2}=P+Q+PQ+QP.\]
If $PQ=QP=0$ then $(P+Q)^{2}=P+Q=(P+Q)^{*}$, hence $P+Q$ is a projection.
Conversely, if $P+Q\in proj(H)$, then $(P+Q)^{2}=(P+Q)$ implies
that $PQ+QP=0$. Since $(PQ)^{*}=Q^{*}P^{*}=QP$ it follows that $2QP=0$
hence $PQ=QP=0$.
\end{proof}
\end{lem}
In terms of characteristic functions, \[
\chi_{A}+\chi_{B}=\chi_{A\cup B}-\chi_{A\cap B}\]
hence $\chi_{A}+\chi_{B}$ is a characteristic function if and only
if $A\cap B=\phi$.

The set of projections in a Hilbert space $H$ is partially ordered
according to the corresponding closed subspaces paritially ordered
by inclusion. Since containment implies commuting, the chain of projections
$P_{1}\leq P_{2}\leq\cdots$ is a family of commuting self-adjoint
operators. By the spectral theorem, $\{P_{i}\}$ may be simultaneously
diagonalized, so that $P_{i}$ is unitarily equivalent to the operator
of multiplication by $\chi_{E_{i}}$ on the Hilbert space $L^{2}(X,\mu)$,
where $X$ is a compact and Hausdorff space. Therefore the lattice
structure of prjections in $H$ is precisely the lattice structure
of $\chi_{E}$, or equivalently, the lattice structure of measurable
sets in $X$.
\begin{lem}
Consider $L^{2}(X,\mu)$. The followsing are equivalent.
\begin{enumerate}
\item $E\subset F$;
\item $\chi_{E}\chi_{F}=\chi_{F}\chi_{E}=\chi_{E}$;
\item $\norm{\chi_{E}f}\leq\norm{\chi_{F}f}$, for any $f\in L^{2}$;
\item $\chi_{E}\leq\chi_{F}$ in the sense that $\iprod{f}{\chi_{E}f}\leq\iprod{f}{\chi_{F}f}$,
for any $f\in L^{2}$.\end{enumerate}
\begin{proof}
The proof is trivial. Notice that\begin{eqnarray*}
\iprod{f}{\chi_{E}f} & = & \int\bar{f}\chi_{E}\bar{f}d\mu=\int\chi_{E}\abs{f}^{2}d\mu\\
\norm{\chi_{E}f}^{2} & = & \int\abs{\chi_{E}f}^{2}d\mu=\int\chi_{E}\abs{f}^{2}d\mu\end{eqnarray*}
where we used that fact that\[
\chi_{E}=\bar{\chi}_{E}=\chi_{E}^{2}.\]

\end{proof}
\end{lem}

\section{Ideas in the spectral theorem}

We show some main ideas in the spectral theorem. Since every normal
operator $N$ can be written as $N=A+iB$ where $A,B$ are commuting
self-adjoint operators, the presentation will be focused on self-adjoint
operators. 

Let $A$ be a self-adjoint operator acting on a Hilbert space $H$.
There are two versions of the spectral theorem. The projection-valued
measure (PVM), and the multiplication operator $M_{f}$.

\subsection{Multiplication by $M_{f}$}
\begin{enumerate}
\item In this version of the spectral theorem, $A=A^{*}$ implies that $A$
is unitarily equivalent to the operator $M_{f}$ of multiplication
by a measurable function $f$ on the Hilbert space $L^{2}(X,\mu)$,
where $X$ is a compact Hausdorff space, and $\mu$ is a regular Borel
measure. \[
\xymatrix{
&H\ar[r]^{A} & H &\\
&L^{2}(\mu)\ar[u]^{U}\ar[r]^{M_{f}} & L^{2}(\mu)\ar[u]_{U}
}
\]$f$ induces a Borel measure $\mu_{f}(\cdot)=\mu\circ f^{-1}(\cdot)$
on $\mathbb{R}$, supported on $f(X)$. Define the operator $W:L^{2}(\mu_{f})\rightarrow L^{2}(\mu)$
where \[
W:g\rightarrow g\circ f.\]
Then,\[
\norm{g}_{L^{2}(\mu_{f})}^{2}=\int\abs{g}^{2}d\mu_{f}=\int_{X}\abs{g\circ f}d\mu=\norm{Wg}_{L^{2}(\mu)}^{2}\]
hence $W$ is unitary. Moreover, \begin{eqnarray*}
M_{f}Wg & = & f(x)g(f(x))\\
WM_{t}g & = & W(tg(t))=f(x)g(f(x))\end{eqnarray*}
it follows that $M_{f}$ is unitarily equivalent to $M_{t}$ on $L^{2}(\mu_{f})$,
and $M_{f}=WM_{t}W^{-1}$. The combined transformation $F:=UW$ diagonalizes
$A$ as \[
A=FM_{t}F^{-1}.\]
It is seens as a vast extention of diagonalizing hermitian matrix
in linear algebra, or a generalization of Fourier transform.\[
\xymatrix{
&H\ar[r]^{A} & H &\\
&L^{2}(\mu)\ar[u]^{U}\ar[r]^{M_{f}} & L^{2}(\mu)\ar[u]_{U} &\\
&L^{2}(\mu_{f})\ar[u]^{W}\ar[r]^{M_{t}}\ar@/^2pc/@{-->}[luur]^{F}&L^{2}(\mu_{f})\ar[u]_{W}\ar@/_2pc/@{-->}[ruul]_{F} &
}
\]
 
\item What's involved are two algebras: the algebra of measurable functions
on $X$, treated as multiplication operators, and the algebra of operators
generated by $A$ (with identity). The two algebras are $*$-isomorphic.
The spectral theorem allows to represent the algebra of $A$ by the
algebra of functions (in this direction, it helps to understand $A$);
also represent the algebra of functions by algebra of operators generated
by $A$ (in this direction, it reveals properties of the function
algebra and the underlying space $X$. We will see this in a minute.) 
\item Let $\mathfrak{A}$ be the algebra of functions. $\pi\in Rep(\mathfrak{A},H)$
is a representation, where \[
\pi(\psi)=FM_{\psi}F^{-1}\]
and we may define operator \[
\psi(A):=\pi(\psi).\]
This is called the \emph{spectral representation}. In particular,
the spectral theorem of $A$ implies the following substitution rule\[
\sum c_{k}x^{k}\mapsto\sum c_{k}A^{k}\]
is well-defined, and it extends to all bounded measurable functions. \end{enumerate}
\begin{rem}
Notice that the map \[
\psi\mapsto\psi(A):=\pi(\psi)=FM_{\psi}F^{-1}\]
is an algebra isomorphism. To check this, \begin{eqnarray*}
(\psi_{1}\psi_{2})(A) & = & FM_{\psi_{1}\psi_{2}}F^{-1}\\
 & = & FM_{\psi_{1}}M_{\psi_{2}}F^{-1}\\
 & = & \left(FM_{\psi_{1}}F^{-1}\right)\left(FM_{\psi_{2}}F^{-1}\right)\\
 & = & \psi_{1}(A)\psi_{2}(A)\end{eqnarray*}
where we used the fact that \[
M_{\psi_{1}\psi_{2}}=M_{\psi_{1}}M_{\psi_{2}}\]
i.e. multiplication operators always commute.
\end{rem}

\subsection{Projection-valued measure}

Alternatively, we have the PVM version of the spectral theorem. \[
A=\int xP(dx)\]
where $P$ is a projection-valued measure defined on the Borel $\sigma$-algebra
of $\mathbb{R}$. Notice that the support of $P$ might be a proper
subset of $\mathbb{R}$. Recall that $P$ is a PVM if $P(\phi)=0$,
$P(\cup E_{k})=\sum P(E_{k})$ and $E_{k}\cap E_{l}=\phi$ for $k\neq l$.
By assumption, the sequence of projections $\sum_{k=1}^{N}P(E_{k})$
($E_{k}'s$ mutually disjoint) is monotone increasing, hence it has
a limit, $\lim_{N\rightarrow}\sum_{k=1}^{N}P(E_{k})=P(\cup E_{k})$.
Convergence is in terms of strong operator topology.

The standard Lebesgue integration extends to PVM. \[
\iprod{\varphi}{P(E)\varphi}=\iprod{P(E)\varphi}{P(E)\varphi}=\norm{P(E)\varphi}^{2}\geq0\]
since $P$ is countablely addative, the map $E\mapsto\norm{P(E)\varphi}^{2}$
is also countablely addative. Therefore, each $\varphi\in H$ induces
a regular Borel measure $\mu_{\varphi}$ on the Borel $\sigma$-algebra
of $\mathbb{R}$.

For a measurable function $\psi$, \begin{eqnarray*}
\int\psi d\mu_{\varphi} & = & \int\psi(x)\iprod{\varphi}{P(dx)\varphi}\\
 & = & \iprod{\varphi}{\left(\int\psi P(dx)\right)\varphi}\end{eqnarray*}
hence we may define \[
\int\psi P(dx)\]
as the operator so that for all $\varphi\in H$, \[
\iprod{\varphi}{\left(\int\psi P(dx)\right)\varphi}.\]

\begin{rem}
$P(E)=F\chi_{E}F^{-1}$ defines a PVM. In fact all PVMs come from
this way. In this sense, the $M_{t}$ version of the spectral theorem
is better, since it implies the PVM version. However, the PVM version
facilites some formuations in quantum mechanics, so physicists usually
prefer this version.
\begin{rem}
Suppose we start with the PVM version of the spectral theorem. How
to prove $(\psi_{1}\psi_{2})(A)=\psi_{1}(A)\psi_{2}(A)$? i.e. how
to chech we do have an algebra isomorphism? Recall in the PVM version,
$\psi(A)$ is defined as the operator so that for all $\varphi\in H$,
we have\[
\int\psi d\mu_{\varphi}=\iprod{\varphi}{\psi(A)\varphi}.\]
As a stardard approximation technique, once starts with simple or
even step functions. Once it is worked out for simple functions, the
extention to any measurable functions is straightforward. Hence let's
suppose (WLOG) f\begin{eqnarray*}
\psi_{1} & = & \sum\psi_{1}(t_{i})\chi_{E_{i}}\\
\psi_{2} & = & \sum\psi_{2}(t_{j})\chi_{E_{j}}\end{eqnarray*}
then\begin{eqnarray*}
\int\psi_{1}P(dx)\int\psi_{2}P(dx) & = & \sum_{i,j}\psi_{1}(t_{i})\psi_{2}(t_{j})P(E_{i})P(E_{j})\\
 & = & \sum_{i}\psi_{1}(t_{i})\psi_{2}(t_{i})P(E_{i})\\
 & = & \int\psi_{1}\psi_{2}P(dx)\end{eqnarray*}
where we used the fact that $P(A)P(B)=0$ if $A\cap B=\phi$.
\end{rem}
As we delve into Nelson's lecture notes, we notice that on page 69,
there is another unitary operator. By pieceing these operators together
is precisely how we get the spectral theorem. This {}``pieceing''
is a vast generalization of Fourier series. \end{rem}
\begin{lem}
pick $\varphi\in H$, get the measure $\mu_{\varphi}$ where \[
\mu_{\varphi}(\cdot)=\iprod{\varphi}{P(\cdot)\varphi}=\norm{P(\cdot)\varphi}^{2}\]
and we have the Hilbert space $L^{2}(\mu_{\varphi})$. Take $H_{\varphi}:=\overline{span}\{\psi(A)\varphi:\psi\in L^{\infty}(\mu_{\varphi})\}$.
Then the map\[
\psi(A)\varphi\mapsto\psi\]
is an isometry, and it extends uniquely to a unitary operator from
$H_{\varphi}$ to $L^{2}(\mu_{\varphi})$.
\end{lem}
To see this, \begin{eqnarray*}
\norm{\psi(A)\varphi}^{2} & = & \iprod{\psi(A)\varphi}{\psi(A)\varphi}\\
 & = & \iprod{\varphi}{\bar{\psi}(A)\psi(A)\varphi}\\
 & = & \iprod{\varphi}{\abs{\psi}^{2}(A)\varphi}\\
 & = & \int\abs{\psi}^{2}d\mu_{\varphi}.\end{eqnarray*}

\begin{rem}
$H_{\varphi}$ is called the cyclic space generated by $\varphi$.
Before we can construct $H_{\varphi}$, we must make sense of $\psi(A)\varphi$. \end{rem}
\begin{lem}
(Nelson p67) Let $p=a_{0}+a_{1}x+\cdots+a_{n}x^{n}$ be a polynomial.
Then $\norm{p(A)u}\leq\max\abs{p(t)}$, where $\norm{u}=1$ i.e. $u$
is a state.\end{lem}
\begin{proof}
$M:=span\{u,Au,\ldots,A^{n}u\}$ is a finite dimensional subspace
in $H$ (automatically closed). Let $E$ be the orthogonal projection
onto $M$. Then\[
p(A)u=Ep(A)Eu=p(EAE)u.\]
Since $EAE$ is a Hermitian matrix on $M$, we may apply the spectral
theorem for finite dimensional space and get\[
EAE=\sum\lambda_{k}P_{\lambda_{k}}\]
where $\lambda_{k}'s$ are eigenvalues associated with the projections
$P_{\lambda_{k}}$. It follows that \begin{eqnarray*}
p(A)u & = & p(\sum\lambda_{k}P_{\lambda_{k}})u=\left(\sum p(\lambda_{k})P_{\lambda_{k}}\right)u\end{eqnarray*}
and \begin{eqnarray*}
\norm{p(A)u}^{2} & = & \sum\left|p(\lambda_{k})\right|^{2}\norm{P_{\lambda_{k}}u}^{2}\\
 & \leq & \max\abs{p(t)}\sum\norm{P_{\lambda_{k}}u}^{2}\\
 & = & \max\abs{p(t)}\end{eqnarray*}
since \[
\sum\norm{P_{\lambda_{k}}u}^{2}=\norm{u}^{2}=1.\]
Notice that $I=\sum P_{\lambda_{k}}$.\end{proof}
\begin{rem}
How to extend this? polynomials - continuous functions - measurable
functions. $[-\norm{A},\norm{A}]\subset\mathbb{R}$, \[
\norm{EAE}\leq\norm{A}\]
is a uniform estimate for all truncations. Apply Stone-Weierstrass
theorem to the interval $[-\norm{A},\norm{A}]$ we get that any continuous
function $\psi$ is uniformly approximated by polynomials. i.e. $\psi\sim p_{n}$.
Thus\[
\norm{p_{n}(A)u-p_{m}(A)u}\leq\max\abs{p_{n}-p_{m}}\norm{u}=\norm{p_{n}-p_{m}}_{\infty}\rightarrow0\]
and $p_{n}(A)u$ is a Cauchy sequence, hence \[
\lim_{n}p_{n}(A)u=:\psi(A)u\]
where we may define the operator $\psi(A)$ so that $\psi(A)u$ is
the limit of $p_{n}(A)u$.
\end{rem}

\subsection{Convert $M_{f}$ to PVM }
\begin{thm}
Let $A:H\rightarrow H$ be a self-adjoint operator. Suppose $A$ is
unitarily equivalent to the operator $M_{t}$ of multiplication by
the independent variable on the Hilbert space $L^{2}(\mu)$. Then
there exists a unique projection-valued measure $P$ so that \[
A=\int tP(dt)\]
i.e. for all $h,k\in H$, \[
\iprod{k}{Ah}=\int t\iprod{k}{P(dt)k}.\]

\begin{proof}
The uniquess part follows from a standard argument. We will only prove
the existance of $P$. Let $F:L^{2}(\mu)\rightarrow H$ be the unitary
operator so that $A=FM_{t}F^{-1}$. Define \[
P(E)=F\chi_{E}F^{-1}\]
for all $E$ in the Borel $\sigma$-algebra $\mathfrak{B}$ of $\mathbb{R}$.
Then $P(\phi)=0$, $P(\mathbb{R})=I$; for all $E_{1},E_{2}\in\mathfrak{B}$,
\begin{eqnarray*}
P(E_{1}\cap E_{2}) & = & F\chi_{E_{1}\cap E_{2}}F^{-1}\\
 & = & F\chi_{E_{1}}\chi_{E_{2}}F^{-1}\\
 & = & F\chi_{E_{1}}F^{-1}F\chi_{E_{2}}F^{-1}\\
 & = & =P(E_{1})P(E_{2}).\end{eqnarray*}
Suppose $\{E_{k}\}$ is a sequence of mutually disjoint elements in
$\mathfrak{B}$. Let $h\in H$ and write $h=F\hat{h}$ for some $\hat{h}\in L^{2}(\mu)$.
Then\begin{eqnarray*}
\iprod{h}{P(\cup E_{k})h}_{H} & = & \iprod{F\hat{h}}{P(\cup E_{k})F\hat{h}}_{H}\\
 & = & \iprod{\hat{h}}{F^{-1}P(\cup E_{k})F\hat{h}}_{L^{2}}\\
 & = & \iprod{\hat{h}}{\chi_{\cup E_{k}}\hat{h}}_{L^{2}}\\
 & = & \int_{\cup E_{k}}\abs{\hat{h}}^{2}d\mu\\
 & = & \sum_{k}\int_{E_{k}}\abs{\hat{h}}^{2}d\mu\\
 & = & \sum_{k}\iprod{h}{P(E_{k})h}_{H}.\end{eqnarray*}
Therefore $P$ is a projection-valued measure. 

For any $h,k\in H$, write $h=F\hat{h}$ and $k=F\hat{k}$. Then \begin{eqnarray*}
\iprod{k}{Ah}_{H} & = & \iprod{F\hat{k}}{AF\hat{h}}_{H}\\
 & = & \iprod{\hat{k}}{M_{t}\hat{h}}_{L^{2}}\\
 & = & \int t\overline{\hat{k}(t)}\hat{h}(t)d\mu(t)\\
 & = & \int t\iprod{k}{P(dt)h}_{H}\end{eqnarray*}
Thus $A=\int tP(dt)$.
\end{proof}
\end{thm}
\begin{rem}
In fact, $A$ is in the closed (under norm or storng topology) span
of $\{P(E):E\in\mathfrak{B}\}$. This is equivalent to say that $t=F^{-1}AF$
is in the closed span of the set of characteristic functions, the
latter is again a standard approximation in measure theory. It suffices
to approximate $t\chi_{[0,\infty]}$. 
\end{rem}
The wonderful idea of Lebesgue is not to partition the domain, as
was the case in Riemann integral over $\mathbb{R}^{n}$, but instead
the range. Therefore integration over an arbitrary set is made possible.
Important exmaples include analysis on groups.
\begin{prop}
\label{pro:approx_step}Let $f:[0,\infty]\rightarrow\mathbb{R}$,
$f(x)=x$, i.e. $f=x\chi_{[0,\infty]}$. Then there exists a sequence
of step functions $s_{1}\leq s_{2}\leq\cdots\leq f(x)$ such that
$\lim_{n\rightarrow\infty}s_{n}(x)=f(x)$.\end{prop}
\begin{proof}
For $n\in\mathbb{N}$, define\[
s_{n}(x)=\begin{cases}
i2^{-n} & x\in[i2^{-n},(i+1)2^{-n})\\
n & x\in[n,\infty]\end{cases}\]
where $0\leq i\leq n2^{-n}-1$. Equivalently, $s_{n}$ can be written
using characteristic functions as \[
s_{n}=\sum_{i=0}^{n2^{n}-1}i2^{-n}\chi_{[i2^{-n},(i+1)2^{-n})}+n\chi_{[n,\infty]}.\]
Notice that on each interval $[i2^{-n},(i+1)2^{-n})$, \begin{eqnarray*}
s_{n}(x) & \equiv & i2^{-n}\leq x\\
s_{n}(x)+2^{-n} & \equiv & (i+1)2^{-n}>x\\
s_{n}(x) & \leq & s_{n+1}(x).\end{eqnarray*}
Therefore, for all $n\in\mathbb{N}$ and $x\in[0,\infty]$, \begin{equation}
x-2^{-n}<s_{n}(x)\leq x\label{eq:conv}\end{equation}
and $s_{n}(x)\leq s_{n+1}(x)$. 

It follows from (\ref{eq:conv}) that \[
\lim_{n\rightarrow\infty}s_{n}(x)=f(x)\]
for all $x\in[0,\infty]$.\end{proof}
\begin{cor}
Let $f(x)=x\chi_{[0,M]}(x)$. Then there exists a sequence of step
functions $s_{n}$ such that $0\leq s_{1}\leq s_{2}\leq\cdots\leq f(x)$
and $s_{n}\rightarrow f$ uniformly, as $n\rightarrow\infty$.\end{cor}
\begin{proof}
Define $s_{n}$ as in proposition (\ref{pro:approx_step}). Let $n>M$,
then by construction\[
f(x)-2^{-n}<s_{n}(x)\leq f(x)\]
for all $s\in[0,M]$. Hence $s_{n}\rightarrow f$ uniformly as $n\rightarrow\infty$.
\end{proof}
Proposition (\ref{pro:approx_step}) and its corollary immediate imply
the following.
\begin{cor}
Let $(X,S,\mu)$ be a measure space. A function (real-valued or complex-valued)
is measurable if and only if it is the pointwise limit of a sequence
of simple function. A function is bounded measurable if and only if
it is the uniform limit of a sequence of simple functions. Let $\{s_{n}\}$
be an approximation sequence of simple functions. Then $s_{n}$ can
be chosen such that $\abs{s_{n}(x)}\leq\abs{f(x)}$ for all $n=1,2,3\ldots$. \end{cor}
\begin{thm}
Let $M_{f}:L^{2}(X,S,\mu)\rightarrow L^{2}(X,S,\mu)$ be the operator
of multiplication by $f$. Then,
\begin{enumerate}
\item if $f\in L^{\infty}$, $M_{f}$ is a bounded operator, and $M_{f}$
is in the closed span of the set of self-adjoint projections under
norm topology. 
\item if $f$ is unbounded, $M_{f}$ is an unbounded operator. $M_{f}$
is in the closed span of the set of self-adjoint projections under
the strong operator topology. \end{enumerate}
\begin{proof}
If $f\in L^{\infty}$, then there exists a sequence of simple functions
$s_{n}$ so that $s_{n}\rightarrow f$ uniformly. Hence $\norm{f-s_{n}}_{\infty}\rightarrow0$,
as $n\rightarrow\infty$. 

Suppose $f$ is unbounded. By proposition () and its corollaries,
there exists a sequence of simple functions $s_{n}$ such that $\abs{s_{n}(x)}\leq\abs{f(x)}$
and $s_{n}\rightarrow f$ pointwisely, as $n\rightarrow\infty$. Let
$h$ be any element in the domain of $M_{f}$, i.e. \[
\int\abs{h}+\abs{fh}^{2}d\mu<\infty.\]
Then\[
\lim_{n\rightarrow\infty}\left|(f(x)-s_{n}(x))h(x)\right|^{2}=0\]
and \[
\abs{(f(x)-s_{n}(x))h(x)}^{2}\leq const\abs{h(x)}^{2}.\]
Hence by the dominiated convergence theorem, \[
\lim_{n\rightarrow\infty}\int\left|(f(x)-s_{n}(x))h(x)\right|^{2}d\mu=0\]
or equivalently,\[
\norm{(f-s_{n})h}^{2}\rightarrow0\]
as $n\rightarrow\infty$. i.e. $M_{s_{n}}$ converges to $M_{f}$
in the strong operator topology. 
\end{proof}
\end{thm}

\section{Spectral theorem for compact operators}

\chapter{GNS, Representations}

\section{Representations, GNS, primer of multiplicity}

\subsection{Decomposition of Brownian motion}

The integral kernel $K:[0,1]\times[0,1]\rightarrow\mathbb{R}$ \[
K(s,t)=s\wedge t\]
is a compact operator on $L^{2}[0,1]$, where\[
Kf(x)=\int(x\wedge y)f(y)dy.\]
$Kf$ is a solution to the differential equation \[
-\frac{d^{2}}{dx^{2}}u=f\]
with zero boundary conditions. 

$K$ is also seen as the covariance functions of Brownian motion process.
A stochastic process is a family of measurable functions $\{X_{t}\}$
defined on some sample probability space $(\Omega,\mathfrak{B},P)$,
where the parameter $t$ usually represents time. $\{X_{t}\}$ is
a Brownian motion process if it is a mean zero Gaussian process such
that \[
E[X_{s}X_{t}]=\int_{\Omega}X_{s}X_{t}dP=s\wedge t.\]
It follows that the corresponding increament process $\{X_{t}-X_{s}\}\sim N(0,t-s)$.
$P$ is called the Wiener measure.

Building $(\Omega,\mathfrak{B},P)$ is a fancy version of Riesz's
representation theorem as in Theorem 2.14 of Rudin's book. It turns
out that \[
\Omega=\prod_{t}\bar{\mathbb{R}}\]
which is a compact Hausdorff space. \[
X_{t}:\Omega\rightarrow\mathbb{R}\]
is defined as\[
X_{t}(\omega)=\omega(t)\]
i.e. $X_{t}$ is the continuous linear functional of evluation at
$t$ on $\Omega$.

For Brownian motion, the increament of the process $\triangle X_{t}$,
in some statistical sense, is proportional to $\sqrt{\triangle t}$.
i.e. \[
\triangle X_{t}\sim\sqrt{\triangle t}.\]
It it this property that makes the set of differentiable functions
have measure zero. In this sense, the trajectory of Brownian motion
is nowhere differentiable. 

An very important application of the spectral theorem of compact operators
in to decompose the Brownian motion process. \[
B_{t}(\omega)=\sum\lambda_{n}\sin(n\pi t)Z_{n}(\omega)\]
where \[
s\wedge t=\sum\lambda_{n}\sin(n\pi t)\]
and $Z_{n}\sim N(0,1)$.

\subsection{Idea of multiplicity}

Recall that $A=A^{*}$ if and only if \[
A=\int_{sp(A)}tP(dt)\]
where $P$ is a projection-valued measure (PVM). The simplest example
of a PVM is $P(E)=\chi_{E}$. The spectral theorem states that all
PVMs come this way.

Let $A$ be a positive compact operator on $H$. $A$ is positive
means $\iprod{x}{Ax}\geq0$ for all $x\in H$. The spectral theorem
of $A$ states that \[
A=\sum\lambda_{n}P_{n}\]
where $\lambda_{n}'s$ are the eigenvalues of $A$, such that $\lambda_{1}\geq\lambda_{2}\geq\cdots\lambda_{n}\rightarrow0$.
and $P_{n}'s$ are self-adjoin projections onto the corresponding
finite dimensional eigenspace of $\lambda_{n}$. $P$ is a PVM supported
on $\{1,2,3,\ldots\}$, so that $P_{n}=P(\{n\})$.

\emph{Question}: what does $A$ look like if it is represented as
the operator of multiplication by the independent variable? 

We may arrange the eigenvalues of $A$ such that\[
\overset{s_{1}}{\overbrace{\lambda_{1}=\cdots=\lambda_{1}}}>\overset{s_{2}}{\overbrace{\lambda_{2}=\cdots=\lambda_{2}}}>\cdots>\overset{s_{n}}{\overbrace{\lambda_{n}=\cdots=\lambda_{n}}}>\cdots\rightarrow0.\]
We say that $\lambda_{i}$ has multiplicity $s_{i}$. The dimension
of the eigen space of $\lambda_{i}$ is $s_{i}$, and \[
dimH=\sum s_{i}.\]

\begin{example}
We represent $A$ as the operator $M_{f}$ of multiplication by $f$
on $L^{2}(X,\mu)$. Let $E_{k}=\{x_{k,1},\ldots,x_{k,s_{k}}\}\subset X$,
and let $H_{k}=span\{\chi_{\{x_{k,j}\}}:j\in\{1,2,\ldots,s_{k}\}\}$.
Let \[
f=\sum_{k=1}^{\infty}\lambda_{k}\chi_{E_{k}},\quad(\lambda_{1}>\lambda_{2}>\cdots>\lambda_{n}\rightarrow0)\]
Notice that $\chi_{E_{k}}$ is a rank $s_{1}$ projection. $M_{f}$
is compact if and only if it is of the given form.
\begin{example}
Follow the previous example, we represent $A$ as the operator $M_{t}$
of multiplication by the independent variable on some Hilbert space
$L^{2}(\mu_{f})$. For simplicity, let $\lambda>0$ and \[
f=\lambda\chi_{\{x_{1},x_{2}\}}=\lambda\chi_{\{x_{1}\}}+\lambda\chi_{\{x_{2}\}}\]
i.e. $f$ is compact since it is $\lambda$ times a rank-2 projection;
$f$ is positive since $\lambda>0$. The eigenspace of $\lambda$
has two dimension, \[
M_{f}\chi_{\{x_{i}\}}=\lambda\chi_{\{x_{i}\}},\quad i=1,2.\]
Define $\mu_{f}(\cdot)=\mu\circ f^{-1}(\cdot)$, then \[
\mu_{f}=\mu(\{x_{1}\})\delta_{\lambda}\oplus\mu(\{x_{2}\})\delta_{\lambda}\oplus\text{cont. sp }\delta_{0}\]
and \[
L^{2}(\mu_{f})=L^{2}(\mu(\{x_{1}\})\delta_{\lambda})\oplus L^{2}(\mu(\{x_{2}\})\delta_{\lambda})\oplus L^{2}(\text{cont. sp }\delta_{0}).\]
Define $U:L^{2}(\mu)\rightarrow L^{2}(\mu_{f})$ by \[
(Ug)=g\circ f^{-1}.\]
$U$ is unitary, and the following diagram commute. \[
\xymatrix{
&L^{2}(X,\mu)\ar[d]^{U}\ar[r]^{M_f} & L^{2}(X,\mu)\ar[d]^{U} &\\
&L^{2}(\mathbb{R},\mu_{f})\ar[r]^{M_{t}} & L^{2}(\mathbb{R},\mu_{f})
}
\]\\
\\
To check $U$ preserves the $L^{2}$-norm,\begin{eqnarray*}
\norm{Ug}^{2} & = & \int\norm{g\circ f^{-1}(\{x\})}^{2}d\mu_{f}\\
 & = & \norm{g\circ f^{-1}(\{\lambda\})}^{2}+\norm{g\circ f^{-1}(\{0\})}^{2}\\
 & = & \abs{g(x_{1})}^{2}\mu(\{x_{1}\})+\abs{g(x_{2})}^{2}\mu(\{x_{2}\})+\int_{X\backslash\{x_{1},x_{2}\}}\abs{g(x)}^{2}d\mu\\
 & = & \int_{X}\abs{g(x)}^{2}d\mu\end{eqnarray*}
To see $U$ diagonalizes $M_{f}$, \begin{eqnarray*}
M_{t}Ug & = & \lambda g(x_{1})\oplus\lambda g(x_{2})\oplus0g(t)\chi_{X\backslash\{x_{1},x_{2}\}}\\
 & = & \lambda g(x_{1})\oplus\lambda g(x_{2})\oplus0\\
UM_{f}g & = & U(\lambda g(x)\chi_{\{x_{1},x_{2}\}})\\
 & = & \lambda g(x_{1})\oplus\lambda g(x_{2})\oplus0\end{eqnarray*}
Thus \[
M_{t}U=UM_{f}.\]

\end{example}
\end{example}
\begin{rem}
Notice that $f$ should really be written as\[
f=\lambda\chi_{\{x_{1},x_{2}\}}=\lambda\chi_{\{x_{1}\}}+\lambda\chi_{\{x_{2}\}}+0\chi_{X\backslash\{x_{1},x_{2}\}}\]
since $0$ is also an eigenvalue of $M_{f}$, and the corresponding
eigenspace is the kernel of $M_{f}$. \end{rem}
\begin{example}
diagonalize $M_{f}$ on $L^{2}(\mu)$ where $f=\mathbf{\chi_{[0,1]}}$
and $\mu$ is the Lebesgue measure on $\mathbb{R}$.
\begin{example}
diagonalize $M_{f}$ on $L^{2}(\mu)$ where \[
f(x)=\begin{cases}
2x & x\in[0,1/2]\\
2-2x & x\in[1/2,1]\end{cases}\]
and $\mu$ is the Lebesgue measure on $[0,1]$.
\end{example}
\end{example}
\begin{rem}
see direc integral and disintegration of measures.
\end{rem}
In general, let $A$ be a self-adjoint operator acting on $H$. Then
there exists a second Hilbert space $K$, a measure $\nu$ on $\mathbb{R}$,
and unitary transformation $F:H\rightarrow L_{K}^{2}(\mathbb{R},\nu)$
such that\[
M_{t}F=FA\]
for measurable function $\varphi:\mathbb{R}\rightarrow K$, \[
\norm{\varphi}_{L_{K}^{2}(\nu)}=\int\norm{\varphi(t)}_{K}^{2}d\nu(t)<\infty.\]

\subsection{GNS}

The multiplication version of the spectral theorm is an exmaple of
representation of the algebra of $L^{\infty}$ functions (or $C(X)$)
as operators acting on a Hilbert space. \[
\pi:L^{\infty}\rightarrow g(A)\in B(H)\]
where $\pi(fg)=\pi(f)\pi(g)$ and $\pi(\bar{f})=\pi(f)^{*}$.

The general question is representation of algebras $B(H)$. The GNS
construction was developed about 60 years ago for getting representations
from data in typical applications, especially in quantum mechanics.
It was developed independently by I. Gelfand, M. Naimark, and I. Segal.

Let $\mathfrak{A}$ be a $*$-algebra with identity. A representation
of $\mathfrak{A}$ is a map $\pi$ on $\mathfrak{A}$, a Hilbert space
$\mathcal{H}_{\pi}$ \[
\pi:\mathfrak{A}\rightarrow B(\mathcal{H}_{\pi})\]
so that \begin{eqnarray*}
\pi(AB) & = & \pi(A)\pi(B)\\
\pi(A^{*}) & = & \pi(A)^{*}.\end{eqnarray*}
The $*$ operation is given on $\mathfrak{A}$ so that $A^{**}=A$,
$(AB)^{*}=B^{*}A^{*}$.

The question is given any $*$-algebra, where to get such a representation?
The answer is given by states. A state $\omega$ on $\mathfrak{A}$
is a linear functional $\omega:\mathfrak{A}\rightarrow\mathbb{C}$
such that \[
\omega(A^{*}A)\geq0,\;\omega(1_{\mathfrak{A}})=1.\]
For example, if $\mathfrak{A}=C(X)$ where $X$ is a compact Hausdorff
space, then \[
\omega(f)=\int fd\mu_{\omega}\]
is a state, where $\mu_{\omega}$ is a Borel probability measure.
\emph{In fact, in the abelian case, states are Borel probability measures.}
\begin{thm}
(GNS) There is a bijection between states $\omega$ and representation
$(\pi,\mathcal{H},\Omega)$ where $\norm{\Omega}=1$, and \[
\omega(A)=\iprod{\Omega}{\pi(A)\Omega}.\]

\end{thm}

\section{States, dual and pre-dual}

Banach space
\begin{itemize}
\item vector space over $\mathbb{C}$
\item norm $\norm{\cdot}$
\item complete
\end{itemize}
Let $V$ be a vector space. 

Let $V$ be a Banach space. $l\in V^{*}$ if $l:V\rightarrow\mathbb{C}$
such that\[
\norm{l}:=\sup_{\norm{v}=1}\abs{l(v)}<\infty.\]
Hahn-Banach theorem implies that for all $v\in V$, $\norm{v}\neq0$,
there exists $l_{v}\in V^{*}$ such that $l(v)=\norm{v}$. The construction
is to define $l_{v}$ on one vector, then use transfinite induction
to extend to all vectors in $V$. Notice that $V^{*}$ is always complete,
even $V$ is an incomplete normed space. i.e. $V^{*}$ is always a
Banach space.

$V$ is embedded in to $V^{**}$ (we always do this). The embedding
is given by\[
v\mapsto\psi(v)\in V^{*}\]
where \[
\psi(v)(l):=l(v).\]

\begin{example}
Let $X$ be a compact Hausdorff space. $C(X)$ with the sup norm is
a Banach space. $(l^{p})^{*}=l^{q}$, $(L^{p})^{*}=L^{q}$, for $1/p+1/q=1$
and $p<\infty$. If $1<p<\infty$, then $(l^{p})^{**}=l^{p}$, i.e.
these spaces are reflexive. $(l^{1})^{*}=l^{\infty}$, but $(l^{\infty})^{*}$
is much bigger than $l^{1}$. Also note that $(l^{p})^{*}\neq l^{q}$
except for $p=q=2$ where $l^{2}$ is a Hilbert space.
\end{example}
Hilbert space
\begin{itemize}
\item vector space over $\mathbb{C}$
\item norm forms an inner product $\iprod{\cdot}{\cdot}$
\item complete with respect to $\norm{\cdot}=\iprod{\cdot}{\cdot}^{1/2}$
\item $H^{*}=H$ 
\item every Hilbert space has a basis (proved by Zorn's lemma)
\end{itemize}
The identification $H=H^{*}$ is due to Riesz, and the corresponding
map is given by\[
h\mapsto\iprod{h}{\cdot}\in H^{*}\]
This can also be seen by noting that $H$ is unitarily equivalent
to $l^{2}(A)$ and the latter is reflecxive.

Let $H$ be a Hilbert space. The set of all bounded operators $B(H)$
on $H$ is a Banach space. We ask two questions: What is $B(H)^{*}$?
Is $B(H)$ the dual space of some Banach space?

The first question extremely difficult and we will discuss that later.
We now show that \[
B(H)=T_{1}(H)^{*}\]
where we denote by $T_{1}(H)$ the trace class operators.

Let $\rho:H\rightarrow H$ be a compact self-adjoint operator. Assume
$\rho$ is positive, i.e. $\iprod{x}{\rho x}\geq0$ for all $x\in H$.
By the spectral theorem of compact operators,\[
\rho=\sum\lambda_{k}P_{k}\]
where $\lambda_{1}\geq\lambda_{2}\geq\cdots\rightarrow0$, and $P_{k}$
is the finite dimensional eigenspace of $\lambda_{k}$. $\rho$ is
a trace class operator, if \[
\sum\lambda_{k}<\infty\]
in which case we may assume $\sum\lambda_{k}=1$.

In general, we want to get rid of the assumption that $\rho\geq0$.
Hence we work, instead, with $\psi:H\rightarrow H$ so that $\psi$
is a trace class operator if \[
\sqrt{\psi^{*}\psi}\]
is a trace class operator.
\begin{lem}
$T_{1}(H)$ is a two-sided ideal in $B(H)$. 
\begin{enumerate}
\item Let $(e_{n})$ be an ONB in $H$. Define the trace of $A\in T_{1}(H)$
as is\[
trace(A)=\sum\iprod{e_{n}}{Ae_{n}}\]
then $trace(A)$ is independent of the choice of ONB in $H$.
\item $trace(AB)=trace(BA)$
\item $T_{1}(H)$ is a Banach space with the trace norm $\norm{\rho}=trace(\rho)$.\end{enumerate}
\begin{lem}
Let $\rho\in T_{1}(H)$. Then \[
A\mapsto trace(A\rho)\]
is a state on $B(H)$. 
\end{lem}
\end{lem}
\begin{rem}
Notice that $A\rho\in T_{1}(H)$ for all $A\in B(H)$. The map \[
A\mapsto trace(A\rho)\]
is in $B(H)^{*}$ means that the dual pairing \[
\iprod{A}{\rho}:=trace(A\rho)\]
satisfies\[
\abs{\iprod{A}{\rho}}\leq\norm{A}_{operator}\norm{\rho}_{trace}.\]
\end{rem}
\begin{thm}
$T_{1}^{*}(H)=B(H)$.
\end{thm}
Let $l\in T_{1}^{*}$. How to get an operator $A$? It is supposed
to be the case such that\[
l(\rho)=trace(\rho A)\]
for all $\rho\in T_{1}$. How to pull an operator $A$ out of the
hat? The idea also goes back to Paul Dirac. It is in fact not difficult
to find $A$. Since $A$ is determined by its matrix, it suffices
to fine\[
\iprod{f}{Af}\]
which are the entries in the matrix of $A$. For any $f_{1},f_{2}\in H$,
the rank-one operator\[
\ketbra{f_{1}}{f_{2}}\]
is in $T_{1}$, hence we know what $l$ does to it, i.e. we know these
numbers $l(\ketbra{f_{1}}{f_{2}})$. Since \[
l(\ketbra{f_{1}}{f_{2}})\]
is linear in $f_{1}$ , and conjugate linear in $f_{2}$, by the Riesz
theorem for Hilbert space, there exists a unique operator $A$ such
that\[
l(\ketbra{f_{1}}{f_{2}})=\iprod{f_{2}}{Af_{1}}.\]
This defines $A$.

Now we check that $l(\rho)=trace(\rho A)$ . Take $\rho=\ketbra{f_{1}}{f_{2}}$.
Then\begin{eqnarray*}
trace(\rho A) & = & trace(\ketbra{f_{1}}{f_{2}}A)\\
 & = & trace(A\ketbra{f_{1}}{f_{2}})\\
 & = & \sum_{n}\iprod{e_{n}}{Af_{1}}\iprod{f_{2}}{e_{n}}\\
 & = & \iprod{f_{2}}{Af_{1}}\end{eqnarray*}
where the last equality follows from Parseval's indentity.
\begin{rem}
If $B$ is the dual of a Banach space, then we say that $B$ has a
predual. For example $l^{\infty}=(l^{1})^{*}$, hence $l^{1}$ is
the predual of $l^{\infty}$. Another example in Rudin's book, $H^{1}$,
hardy space of analytic functions on the disk. $(H^{1})^{*}=BMO$,
where $BMO$ referes to bounded mean oscillation. It was developed
by Charles Fefferman in 1974 who won the fields medal for this theory.
Getting hands on a specific dual space is often a big thing.
\end{rem}
Let $B$ be a Banach space and denote by $B^{*}$ its dual space.
$B^{*}$ is a Banach space as well, where the norm is defined by\[
\norm{f}_{B^{*}}=\sup_{\norm{x}=1}{\abs{f(x)}}.\]
Let $B_{1}^{*}=\{f\in B^{*}:\norm{f}\leq1\}$ be the unit ball in
$B^{*}$. 
\begin{thm}
$B_{1}^{*}$ is weak $*$ compact in $B^{*}$.
\begin{proof}
This is proved by showing $B_{1}^{*}$ is a closed subspace in $\prod_{\norm{x}=1}\mathbb{C}$,
where the latter is given its product topology, and is compact and
Hausdorff.\end{proof}
\begin{cor}
Every bounded sequence in $B^{*}$ has a convergent subsequence in
the $W^{*}$-topology.
\begin{cor}
Every bounded sequence in $B(H)$ contains a convergence subsequence
in the $W^{*}$-topology.
\end{cor}
\end{cor}
\end{thm}

\section{Examples of representations, proof of GNS}
\begin{example}
Fourier algebra. \end{example}
\begin{itemize}
\item discrete case: \[
(a*b)_{n}=\sum_{k}a_{k}b_{n-k}\]
involution\[
(a^{*})_{n}=\bar{a}_{-n}\]
Most abelian algebras can be thought of function algebras. \[
(a_{n})\mapsto\sum a_{n}z^{n}\]
may specilize to $z=e^{it}$, $t\in\mathbb{R}\text{ mod }2\pi$. $\{F(z)\}$
is an abelian algebra of functions.\[
F(z)G(z)=\sum(a*b)_{n}z^{n}\]
Homomorphism:\begin{eqnarray*}
(l^{1},*) & \rightarrow & C(T^{1})\\
(a_{n}) & \mapsto & F(z).\end{eqnarray*}
If we want to write $F(z)$ as power series, then we need to drop
$a_{n}$ for $n<0$. Then $F(z)$ extends to an analytic function
over the unit disk. The sequence space\[
\{a_{0},a_{1},\ldots\}\]
was suggested by Hardy. The Hardy space $H_{2}$ is a Hilbert space.
Rudin has two beautiful chapters on $H^{2}$. (see chapter 16)
\item continuous case:\[
(f*g)(x)=\int f(s)g(t-s)ds.\]
\end{itemize}
\begin{rem}
$C(T^{1})$ is called the $C^{*}$-algebra completion of $l^{1}$.
$L^{\infty}(X,B,\mu)=L^{1}(\mu)^{*}$ is also a $C^{*}$-algebra.
It is called a $W$ {*} algebra, or Von Neumann algebra. The $W$
{*} refers to the fact that its topology comes from the weak {*} topology.
$B(H)$, for any Hilbert space, is a Von Neumann algebra. \end{rem}
\begin{example}
$uf=e^{i\theta}f(\theta)$, $vf=f(\theta-\varphi)$, restrict to $[0,2\pi]$,
i.e. $2\pi$ periodic functions.\begin{eqnarray*}
vuv^{-1} & = & e^{i\varphi}u\\
vu & = & e^{i\varphi}uv\end{eqnarray*}
$u,v$ generate $C^{*}$-algebra, noncummutative.
\begin{example}
(from quantum mechanics)\[
[p,q]=-iI\]
$p,q$ generate an algebra. But they can not be represented by bounded
operators. May apply bounded functions to them and get a $C^{*}$-algebra.
\begin{example}
Let $H$ be an infinite dimensional Hilbert space. $H$ is isometrically
isomorphic to a subspace of itself. For example, let $\{e_{n}\}$
be an ONB. $H_{1}=\overline{span}\{e_{2n}\}$, $H_{2}=\overline{span}\{e_{2n+1}\}$.
Let \begin{eqnarray*}
V_{1}(e_{n}) & = & e_{2n}\\
V_{2}(e_{n}) & = & e_{2n+1}\end{eqnarray*}
then we get two isometries.\begin{eqnarray*}
V_{1}V_{1}^{*}+V_{2}V_{2}^{*} & = & I\\
V_{i}^{*}V_{i} & = & I\\
V_{i}V_{i}^{*} & = & P_{i}\end{eqnarray*}
where $P_{i}$ is a self-adjoint projection, $i=1,2$. This is the
Cuntz algebra $\mathcal{O}_{2}$. More general $\mathcal{O}_{n}$.
Cuntz showed that this is a simple algebra in 1977.
\end{example}
\end{example}
\end{example}
From algebras, get representations. For abelian algebras, we get measures.
For non-abelian algebras, we get representations, and the measures
come out as a corallory of representations. 
\begin{proof}
(Sketch of the proof of $GNS$) Let $w$ be a state on $\mathfrak{A}$.
Need to construct $(\pi,H,\Omega)$.

$\mathfrak{A}$ is an algebra, and it is also a complex vector space.
We pretend that $\mathfrak{A}$ is a Hilbert space, see what is needed
for this.

We do get a homomorphism $\mathfrak{A}\rightarrow H$ which follows
from the associative law of $\mathfrak{A}$ being an algebra, i.e.
$(AB)C=A(BC)$.

For Hilbert space $H$, need an inner product. Try\[
\iprod{A_{1}}{A_{2}}=w(A_{1}^{*}A_{2})\]
Then $\iprod{A_{1}}{A_{2}}$ is linear in $A_{2}$, conjugate linear
in $A_{1}$. It also satisfies \[
\iprod{A_{1}}{A_{1}}\geq0\]
which is built into the definition of a state. But it may not be positive
definite. Therefore we take\[
H:=[H/\{A:w(A^{*}A)=0\}]^{cl}.\]

\begin{lem}
\textup{$\{A:w(A^{*}A)=0\}$ is a closed subspace of $\mathfrak{A}$.}\end{lem}
\begin{proof}
this follows from the Schwartz inequality.
\end{proof}
Let $\pi:\mathfrak{A}\rightarrow H$ such that $\pi(A)=A/ker$. Let
$\Omega=\pi(I)$. Therefore \[
H=\overline{span}\{\pi(A)\Omega:A\in\mathfrak{A}\}.\]
To see we do get a representation, take two typical vectors $\pi(B)\Omega$
and $\pi(A)\Omega$ in $H$, then\begin{eqnarray*}
\iprod{\pi(B)\Omega}{\pi(C)\pi(A)\Omega} & = & \iprod{\Omega}{\pi(B^{*}CA)\Omega}\\
 & = & w(B^{*}CA)\\
 & = & w((C^{*}B)^{*}A)\\
 & = & \iprod{\pi(C^{*}B)\Omega}{\pi(A)\Omega}\\
 & = & \iprod{\pi(C^{*})\pi(B)\Omega}{\pi(A)\Omega}\end{eqnarray*}
it follows that \[
\iprod{v}{\pi(C)u}=\iprod{\pi(C^{*})v}{u}\]
for any $u,v\in H$. Therefore,\[
\pi(C)^{*}=\pi(C^{*}).\]

\end{proof}
\begin{example}
$C[0,1]$, $a\mapsto a(0)$, $a^{*}a=\abs{a}^{2}$, hence $w(a^{*}a)=\abs{a}^{2}(0)\geq0$.
\[
ker=\{a:a(0)=0\}\]
and $C/ker$ is one dimensional. The reason if $\forall f\in C[0,1]$
such that $f(0)\neq0$, we have\[
f(x)\sim f(0)\]
because $f(x)-f(0)\in ker$, where $f(0)$ represent the constant
function $f(0)$ over $[0,1]$. This shows that $w$ is a pure state,
since the representation has to be irreducible.
\end{example}

\section{GNS, spectral thoery}
\begin{rem}
some dover books
\begin{itemize}
\item Stefan Banach, Theory of linear operators
\item Howard Georgi, weak interactions and modern particle theory
\item P.M. Prenter, splines and variational methods
\end{itemize}
\end{rem}

\subsection{GNS for $C^{*}$-algebras}
\begin{defn}
A representation $\pi\in Rep(\mathfrak{A},H)$ is cyclic if it has
a cyclic vector.\end{defn}
\begin{thm}
Give a representation $\pi\in Rep(\mathfrak{A},H)$, there exists
an index set $J$, closed subspaces $H_{j}\subset H$ such that
\begin{itemize}
\item (orthogonal) $H_{i}\perp H_{j}$, $\forall i\neq j$
\item (total) $\sum_{j\in J}^{\oplus}H_{j}=H$, $v_{j}\in H_{j}$ such that
the restriction of $\pi$ to $H_{j}$ is cyclic with cyclic vector
$v_{j}$. Hence on each $H_{j}$, $\pi$ is abelian.\end{itemize}
\begin{rem*}
This looks like the construction of orthonormal basis. But it's a
family of mutually orthogonal subspaces. Of course, if $H_{j}$ is
one-dimensional for all $j$, then it is a decomposition into an ONB.
Not every representation is irreducible, but every representation
can be decompossed into direct sum of cyclic representations.

We use Zorn's lemma to show total, exactly the same argument for the
existance of an ONB of any Hilbert space.\end{rem*}
\begin{thm}
(Gelfand-Naimark) Every $C^{*}$-algebra is isometrically isomorphic
to a norm-closed sub-algebra of $B(H)$, for some Hilbert space $H$.
\begin{proof}
Let $\mathfrak{A}$ be any $C^{*}$-algebra, no Hilbert space $H$
is given from outside. Let $S(\mathfrak{A})$ be the states on $\mathfrak{A}$,
which is a compact convex subset of $\mathfrak{A}^{*}$. Compactness
refers to the weak {*} topology.

We use Hahn-Banach theorem to show that there are plenty of states.
Specifically, $\forall A\in\mathfrak{A}$, $\exists w$ such that
$w(A)>0$. This is done first on 1-dimensional subspace,\[
tA\mapsto t\in\mathbb{R}\]
then extend to $\mathfrak{A}$. This is also a consequence of Krein-Millman,
i.e. $S(\mathfrak{A})=cl(PureStates)$. We will come back to this
point later in the course.

For each state $w$, get a cyclic representation $(\pi_{w},H_{w},\Omega_{w})$,
such that $\pi=\oplus\pi_{w}$ is a representation on the Hilbert
space $H=\oplus H_{w}$.\end{proof}
\begin{thm}
$\mathfrak{A}$ abelian $C^{*}$-algebra. Then $\mathfrak{A}\cong C(X)$
where $X$ is a compact Hausdorff space.
\end{thm}
\end{thm}
\begin{rem*}
Richard Kadison in 1950's reduced the axioms of $C^{*}$-algebra from
about 6 down to just one on the $C^{*}$-norm, \[
\norm{A^{*}A}=\norm{A}^{2}.\]

\end{rem*}
\end{thm}

\subsection{Finishing spectral theorem on Nelson}

We get a family of states $w_{j}\in S$, corresponding measures $\mu_{j}$,
and Hilbert spaces $H_{j}=L^{2}(\mu_{j})$. Note that all the $L^{2}$
spaces are on $K=sp(A)$. So it's the same underlying set, but with
possibly different measures.

To get a single measure space with $\mu$, Nelson suggests taking
the disjoint union \[
K^{DS}=\bigcup_{j}K\times\{j\}\]
and $\mu^{DS}$ is the disjoint union of $\mu_{j}'s$. The existance
of $\mu$ follows from Riesz. Then we get\[
H=\oplus H_{j}\mapsto L^{2}(K^{DS},\mu^{DS}).\]
Notice that this map is into but not onto.

The multiplicity theory starts with breaking up each $H_{j}$ into
irreducible components.

\subsection{Examples on disintegration}
\begin{example}
$L^{2}(I)$ with Lebesgue measure. Let \[
F_{x}(t)=\begin{cases}
1 & t\geq x\\
0 & t<x\end{cases}\]
$F_{x}$ is a monotone increasing function on $\mathbb{R}$, hence
by Riesz, we get the corresponding Riemann-Stieljes measure $dF_{x}$.
\[
d\mu=\int^{\oplus}dF_{x}(t)dx.\]
i.e.\[
\int fd\mu=\int dF_{x}(f)dx=\int f(x)dx.\]
Equivalently, \[
d\mu=\int\delta_{x}dx\]
i.e.\[
\int fd\mu=\int\delta_{x}(f)dx=\int f(x)dx.\]
$\mu$ is a state, $\delta_{x}=dF_{x}(t)$ is a pure state, $\forall x\in I$.
This is a decomposition of state into direct integral of pure states.
\begin{example}
$\Omega=\prod_{t\geq0}\bar{\mathbb{R}}$, $\Omega_{x}=\{w\in\Omega:w(0)=x\}$.
Kolmogorov gives rise to $P_{x}$ by conditioning $P$ with respect
to {}``starting at $x$''.\[
P=\int^{\oplus}P_{x}dx\]
i.e.\[
P()=\int P(\cdot|\text{start at }x)dx.\]

\begin{example}
Harmonic function on $D$ \[
h\mapsto h(z)=\int()d\mu_{z}\]
Poisson integration.
\end{example}
\end{example}
\end{example}

\subsection{Noncommutative Radon-Nicodym derivative}

Let $w$ be a state, $K$ is an operator.\[
w_{k}(A)=\frac{w(\sqrt{K}A\sqrt{K})}{w(K)}\]
is a state, and $w_{k}\ll w$, i.e. $w(A)=0\Rightarrow w_{K}(A)=0$.
$K=dw/dw_{K}$. 

Check:\begin{eqnarray*}
w_{K}(1) & = & 1\\
w_{K}(A^{*}A) & = & \frac{w(\sqrt{K}A^{*}A\sqrt{K})}{w(K)}\\
 & = & \frac{w((A\sqrt{K})^{*}(A\sqrt{K}))}{w(K)}\geq0\end{eqnarray*}
see Sakai - $C^{*}$ and $W^{*}$ algebras.

\section{Choquet, Krein-Milman, decomposition of states}

The main question here is how to break up a representation into smaller
ones. The smallest representations are the irreducible ones. The next
would be multiplicity free representation.

Let $\mathfrak{A}$ be an algebra. 
\begin{itemize}
\item commutative: e.g. function algebra
\item non-commutative: matrix algebra, algebras generated by representation
of non-abelian groups
\end{itemize}
Smallest representation:
\begin{itemize}
\item \emph{irreducible -} $\pi\in Rep_{irr}(\mathfrak{A},H)$ where $H$
is 1-dimensional. This is the starting point of further analysis
\item \emph{multiplicity free - $\pi\in Rep(\mathfrak{A},H)$ assuming cyclic,
since otherwise }$\pi=\oplus\pi_{cyc}$. $\pi$ is multiplicity free
if and only if $\mathfrak{A}'$ is abelian. In general $\mathfrak{A}'$
may or may not be abelian.
\end{itemize}
Let $\mathfrak{C}\subset B(H)$ be a $*$-algebra, $\mathfrak{C}'=\{X:H\rightarrow H|XC=CX,\forall C\in\mathfrak{C}\}$.
$\mathfrak{C}'$ is also a $*$-algebra. It is obvious that $\mathfrak{C}\subset\mathfrak{C}'$
if and only if $\mathfrak{C}$ is abelian.
\begin{defn}
Let $\pi\in Rep(\mathfrak{A},H)$. $multi(\pi)=n\Leftrightarrow\pi(\mathfrak{A})'\simeq M_{n}(\mathbb{C})$.\end{defn}
\begin{example*}
Let \[
A=\left[\begin{array}{ccc}
1\\
 & 1\\
 &  & 2\end{array}\right]=\left[\begin{array}{cc}
I_{2} & 0\\
0 & 2\end{array}\right].\]
Then $AC=CA$ if and only if \[
C=\left[\begin{array}{ccc}
a & b\\
c & d\\
 &  & 1\end{array}\right]=\left[\begin{array}{cc}
B & 0\\
0 & 1\end{array}\right]\]
where $B\in M_{2}(\mathbb{C})$.
\end{example*}

\subsection{Decomposition of states}

Let $\mathfrak{A}$ be a commutative $C^{*}$-algebra containing identity.
Gelfand-Naimark's theorem for commutative $C^{*}$-algebra says that
$\mathfrak{A}\simeq C(X)$, where $X$ is a compact Hausdorff space.
$X$ is called the Gelfand space or the spectrum of $\mathfrak{A}$.
In general, for Banach $*$-algebra, also get $X$.

What is $X$?

Let $\mathfrak{A}$ be a commutative Banach $*$-algebra. The states
$S(\mathfrak{A})$ is a compact convec non empty set in the dual $\mathfrak{A}^{*}$,
which is embedded into a closed subset of $\prod_{\mathfrak{A}}\bar{\mathbb{R}}$.
$w$ is a state if $w(1)=1$; $w(A^{*}A)\geq0$; and $w$ is linear.
Note that convex linear combination of states are also states. i.e.
if $w_{1},w_{2}$ are states, then $w=tw_{1}+(1-t)w_{2}$ is also
a state.
\begin{note*}
dual of a normed vector space has its unit ball being weak $*$-compact.\end{note*}
\begin{thm}
(Krein-Milman) Let $K$ be a compact convec set in a locally compac
topological space. Then $K$ is equal to the closed convex hull of
its extreme points. i.e.\[
K=\overline{conv}(E(K)).\]
\end{thm}
\begin{note*}
The dual of a normed vector space is always a Banach space, so the
theorem applies. The convex hull in an infinite dimensional space
is not always closed, so close it. A good reference to locally convex
topological space is TVS by F. Treves.
\end{note*}
The decompsition of states into pure states was developed by R.Phelps
for representation theory. The idea goes back to Choquet.
\begin{thm}
(Choquet) Let $w\in K=S(\mathfrak{A})$, there exists a measure $\mu_{w}$
{}``concentrated'' on $E(K)$, such that for affine function $f$
\[
f(w)=\int_{"E(K)"}fd\mu_{w}.\]
\end{thm}
\begin{note*}
$E(K)$ may not be Borel. In this case replace $E(K)$ be Borel set
$V\supset E(K)$ s.t.\[
\mu_{w}(V\backslash E(K))=0.\]
This is a theorem of Glimm. Examples for this case include the Cuntz
algebra, free group with 2 generators, $uvu^{-1}=u^{2}$ for wavelets.
\begin{note*}
$\mu_{w}$ may not be unique. If it is unique, $K$ is called a simplex.
The unit disk has its boundary as extreme points. But representation
of points in the interior using points on the boundary is not unique.
Therefore the unit disk is not a simplex. A triangle is. 
\end{note*}
\end{note*}
\begin{proof}
(Krein-Milman) If $K\supsetneqq\overline{conv}(E(K))$, get a linear
functional $w$, such that $w$ is zero on $\overline{conv}(E(K))$
and not zero on $w\in K\backslash\overline{conv}(E(K))$. Extend by
Hahn-Banach theorem to a linear functional to the whole space, and
get a contradiction.
\end{proof}

\subsection{The Gelfand space $X$}

Back to the question of what the Gelfand space $X$ is.

Let $\mathfrak{A}$ be a commutative Banach $*$-algebra. Consider
the closed ideals in $\mathfrak{A}$ (since $\mathfrak{A}$ is normed,
so consider closed ideals) ordered by inclusion. By zorn's lemma,
there exists maximal ideals $M$. $\mathfrak{A}/M$ is 1-dimensional,
hence $\mathfrak{A}/M=\{tv\}$ for some $v\in\mathfrak{A}$ and $t\in\mathbb{R}$.
Therefore\[
\mathfrak{A}\rightarrow\mathfrak{A}/M\rightarrow\mathbb{C}\]
the combined map $\varphi:a\mapsto a/M\mapsto t_{a}$ is a homomorphism.
$\mathfrak{A}\ni1\mapsto v:=1/M\in\mathfrak{A}/M$ then $\varphi(1)=1$. 

Conversely, the kernel of a homomorphism $\varphi$ is a maximal ideal
in $\mathfrak{A}$. Therefore there is a bijection between maximal
ideas and homomorphisms. Note that if $\varphi:\mathfrak{A}\rightarrow\mathbb{C}$
is a homomorphism then it has to be a contraction.

Let $X$ be the set of all maximal ideals $\sim$ all homomorphisms
in $\mathfrak{A}_{1}^{*}$, where $\mathfrak{A}_{1}^{*}$ is the unit
ball in $\mathfrak{A}^{*}$. Since $\mathfrak{A}_{1}^{*}$ is compact,
$X$ is closed in it, $X$ is also compact.

The Gelfand transform $\mathcal{F}:\mathfrak{A}\rightarrow C(X)$
is defined by\[
\mathcal{F}(a)(\varphi)=\varphi(a)\]
then \[
\mathfrak{A}/ker\mathcal{F}\simeq C(X)\]
(mod the kernel for general Banach algebras)
\begin{example}
$l^{1}(\mathbb{Z})$, the convolution algebra. \begin{eqnarray*}
(ab)_{n} & = & \sum_{k}a_{k}b_{n-k}\\
a_{n}^{*} & = & \bar{a}_{-n}\\
\norm{a} & = & \sum_{n}\abs{a_{n}}\end{eqnarray*}
To identity $X$ in practice, always start with a guess, and usually
it turns out to be correct. Since Fourier transform converts convolution
to multiplication, \[
\varphi_{z}:a\mapsto\sum a_{n}z^{n}\]
is a complext homormorphism. To see $\varphi_{z}$ is multiplicative,
\begin{eqnarray*}
\varphi_{z}(ab) & = & \sum(ab)_{n}z^{n}\\
 & = & \sum_{n,k}a_{k}b_{n-k}z^{n}\\
 & = & \sum_{k}a_{k}z^{k}\sum_{n}b_{n-k}z^{n-k}\\
 & = & \left(\sum_{k}a_{k}z^{k}\right)\left(\sum_{k}b_{k}z^{k}\right).\end{eqnarray*}
Thus $\{z:\abs{z}=1\}$ is a subspace in the Gelfand space $X$. Note
that we cannot use $\abs{z}<1$ since we are dealing with two-sided
$l^{1}$ sequence. If the sequences were trancated, so that $a_{n}=0$
for $n<0$ then we allow $\abs{z}<1$. 

$\varphi_{z}$ is contractive: $\abs{\varphi_{z}(a)}=\abs{\sum a_{n}z^{n}}\leq\sum_{n}\abs{a_{n}}=\norm{a}$. 

It turns out that every homorphism is obtained as $\varphi_{z}$ for
some $\abs{z}=1$, hence $X=\{z:\abs{z}=1\}$.
\begin{example}
$l^{\infty}(\mathbb{Z})$, with $\norm{a}=\sup_{n}\abs{a_{n}}$. The
Gelfand space in this case is $X=\beta\mathbb{Z}$, the Stone-Cech
compactification of $\mathbb{Z}$, which are the ultrafilters on $\mathbb{Z}$.
Pure states on diagonal operators correspond to $\beta\mathbb{Z}$.
$\beta\mathbb{Z}$ is much bigger then $p-$adic numbers. 
\end{example}
\end{example}

\section{Beginning of multiplicity}

We study Nelson's notes. Try to get the best generalization from finite
dimensional linear algebra. Nelson's idea is to get from self-adjoint
operators $\rightarrow$ cyclic representation of function algebra
$\rightarrow$ measure $\mu$ $\rightarrow$ $L^{2}(\mu)$.

Start with a single self-adjoint operator $A$ acting on an abstract
Hilbert space $H$. $H$ can be broken up into a direct sum of mutually
disjoint cyclic spaces.

Let $u\in H$. $\{f(A)u\}$, $f$ runs through some function algebra,
generates a subspace $H_{u}\subset H$. The funciton algebra might
be taken as the algebra of polynomials, then later it is extended
to a much bigger algebra contaning polynimials as a dense sub-algebra. 

The map $f\mapsto w_{u}(f):=\iprod{u}{f(A)u}$ is a state on polynomials,
and it extends to a state on $C_{c}(\mathbb{R})$. By Riesz, there
exists a unique mesure $\mu_{u}$ such that\[
w_{u}(f)=\iprod{u}{f(A)u}=\int_{\mathbb{R}}fd\mu\]
It turns out that $\mu_{w}$ is supported on $[0,\norm{A}]$, assuming
$A\geq0$. If $A$ is not positive, it can be written as the positive
part and the negative part. Therefore we get $L^{2}(\mu)$, a Hilbert
space containing polyomials as a dense subspace. Let $H_{u}=\overline{span}\{f(A)u\}$
for all polynomials $f$. Define $W:H_{u}\rightarrow L^{2}(\mu)$,
such that $Wf(A)u=f$.
\begin{lem}
(1) $W$ is well-defined, isometric; (2) $WA=M_{t}W$, i.e. $W$ intertwines
$A$ and $M_{t}$, and $W$ diagnolizes $A$. 
\begin{note*}
$WA=M_{t}W\Leftrightarrow WAW^{*}=M_{t}$. In finite dimension, it
is less emphasized that the adjoint $W^{*}$ equals the inverse $W^{-1}$.
For finite dimensional case, $M_{t}=diag(\lambda_{1},\lambda_{2},\ldots\lambda_{n})$
where the measure $\mu=\sum\delta_{\lambda_{i}}$.\end{note*}
\begin{proof}
For (1),\begin{eqnarray*}
\int\abs{f}^{2}d\mu & = & \iprod{u}{\abs{f}^{2}(A)u}\\
 & = & \iprod{u}{\bar{f}(A)f(A)u}\\
 & = & \iprod{u}{f(A)^{*}f(A)u}\\
 & = & \iprod{f(A)u}{f(A)u}\\
 & = & \norm{f(A)u}^{2}.\end{eqnarray*}
Notice that strictly speaking, $f(A^{*})^{*}=\bar{f}(A)$. Since $A^{*}=A$,
therefore $f(A^{*})^{*}=\bar{f}(A)$. $\pi(f)=f(A)$ is a representation.
i.e. $\pi(\bar{f})=f(A)^{*}$.

For (2), let $f$ be a polynomial, where $f(t)=a_{0}+a_{1}t+\cdots a_{n}t^{n}$.
Then \begin{eqnarray*}
WAf(A) & = & WA(a_{0}+a_{1}A+a_{2}A^{2}+\cdots+a_{n}A^{n})\\
 & = & W(a_{0}A+a_{1}A^{2}+a_{2}A^{3}+\cdots+a_{n}A^{n+1})\\
 & = & a_{0}t+a_{1}t^{2}+a_{2}t^{3}+\cdots+a_{n}t^{n+1}\\
 & = & tf(t)\\
 & = & M_{t}Wf(A)\end{eqnarray*}
thus $WA=M_{t}W$.
\end{proof}
\end{lem}
The whole Hilbert space $H$ then decomposes into a direct sum of
mutually orthogonal cyclic subspaces, each one is unitarily equivalent
to $L^{2}(\mu_{u})$ for some cyclic vector $u$. This representation
of $L^{\infty}$ onto $H=\sum^{\oplus}H_{u}$ is highly non unique.
There we enter into the multiplicity theory.

Each cyclic representation is multiplicity free. i.e. $\pi\in Rep(L^{\infty}(\mu),H_{u})$
is multiplicity if and only if $(\pi(L^{\infty}))'$ is abelian. In
general, $(\pi(L^{\infty}))'$ is not abelian, and we say $\pi$ has
multiplicity equal to $n$ if and only if $(\pi(L^{\infty}))'\simeq M_{n}(\mathbb{C})$.
This notation of multiplicity free generalizes the one in finite dimensional
linear algebra.
\begin{example}
$M_{t}:L^{2}[0,1]\rightarrow L^{2}[0,1]$, $M_{t}$ has no eigenvalues.
\begin{example}
$M_{\varphi}:L^{2}(\mu)\rightarrow L^{2}(\mu)$, with $\varphi\in L^{\infty}(\mu)$.
Claim: $\pi\in Rep(L^{\infty}(\mu),L^{2}(\mu))$ is multiplicity free,
i.e. it is maximal abelian. Suppose $B$ commutes with all $M_{\varphi}$,
define $g=B1$. Then\[
B\psi=B\psi1=BM_{\psi}1=M_{\psi}B1=M_{\psi}g=\psi g=g\psi=M_{g}\psi\]
thus $B=M_{g}$.
\end{example}
\end{example}
Examples that do have multiplicties in finite dimensional linear algebra:
\begin{example}
2-d, $\lambda I$, $\{\lambda I\}'=M_{2}(\mathbb{C})$ which is not
abelian. Hence $mult(\lambda)=2$.
\begin{example}
3-d, \[
\left[\begin{array}{ccc}
\lambda_{1}\\
 & \lambda_{1}\\
 &  & \lambda_{2}\end{array}\right]=\left[\begin{array}{cc}
\lambda_{1}I\\
 & \lambda_{2}\end{array}\right]\]
where $\lambda_{1}\neq\lambda_{2}$. The commutatant is\[
\left[\begin{array}{cc}
B\\
 & b\end{array}\right]\]
 where $B\in M_{2}(\mathbb{C})$, and $b\in\mathbb{C}$. Therefore
the commutant is isomorphic to $M_{2}(\mathbb{C})$, and multiplicity
is equal to 2.
\begin{example}
The example of $M_{\varphi}$ with repeteation. \[
M_{\varphi}\oplus M_{\varphi}:L^{2}(\mu)\oplus L^{2}(\mu)\rightarrow L^{2}(\mu)\oplus L^{2}(\mu)\]
\[
\left[\begin{array}{cc}
M_{\varphi}\\
 & M_{\varphi}\end{array}\right]\left[\begin{array}{c}
f_{1}\\
f_{2}\end{array}\right]=\left[\begin{array}{c}
\varphi f_{1}\\
\varphi f_{2}\end{array}\right]\]
the commutant is this case is isomorphic to $M_{2}(\mathbb{C})$.
If we introduces tensor product, then representation space is also
written as$L^{2}(\mu)\otimes V_{2}$, the multiplication operator
is amplified to $M_{\varphi}\otimes I$, whose commutant is represented
as $I\otimes V_{2}$. Hence it's clear that the commutatant is isomorphic
to $M_{2}(\mathbb{C})$. To check\begin{eqnarray*}
(\varphi\otimes I)(I\otimes B) & = & \varphi\otimes B\\
(I\otimes B)(\varphi\otimes I) & = & \varphi\otimes B.\end{eqnarray*}

\end{example}
\end{example}
\end{example}

\section{Completely positive maps}

The GNS construction gives a bijection between states and cyclic representations.
An extention to the GNS construction is Stinespring's completely positive
map. It appeared in an early paper by Stinspring in 1956 (PAMS). Arveson
in 1970's (Arveson 1972) reformuated Stingspring's result using tensor
product. He realized that complemetly positive maps are the key in
multivariable operator theory, and in noncommutative dynamics.

\subsection{Motivations}

Let $\mathfrak{A}$ be a $*$-algebra with identity. $w:\mathfrak{A}\rightarrow\mathbb{C}$
is a a state if $w(1)=1$, $w(A^{*}A)\geq0$. If $\mathfrak{A}$ was
a $C^{*}$-algebra, $A\geq0\Leftrightarrow sp(A)\geq0$, hence may
take $B=\sqrt{A}$ and $A=B^{*}B$. Given a state $w$, the GNS construction
gives a Hilbert space $K$, a cyclic vector $\Omega\in K$, and a
representation $\pi:\mathfrak{A}\rightarrow B(K)$, such that \[
w(A)=\iprod{\Omega}{\pi(A)\Omega}\]
\[
K=\overline{span}\{\pi(A)\Omega:A\in\mathfrak{A}\}\]
Moreover, the Hilbert space is unique up to unitary equivalence.

Stinespring modified a single axiom in the GNS construction. In stead
of a state $w:\mathfrak{A}\rightarrow\mathbb{C}$, Stinespring considered
a positive map $\varphi:\mathfrak{A}\rightarrow B(H)$, i.e. $\varphi$
maps positive elements in $\mathfrak{A}$ to positive operators in
$B(H)$. $\varphi$ is a natural extention of $w$, since $\mathbb{C}$
can be seen as a 1-dimensional Hilbert space, and $w$ is a positive
map $w:\mathfrak{A}\rightarrow B(\mathbb{C})$. He soon realized that
the $\varphi$ being a positive map is not enough to produce a Hilbert
space and a representation. It turns out that the condition to put
on $\varphi$ is complete positivity, in the sense that, for all $n\in\mathbb{N}$
\begin{equation}
\varphi\otimes I_{M_{n}}:\mathfrak{A}\otimes M_{n}\rightarrow B(H\otimes\mathbb{C}^{n})\label{eq:CP}\end{equation}
maps positive elements in $\mathfrak{A}\otimes M_{n}$ to positive
operators in $B(H\otimes\mathbb{C}^{n})$. $\varphi$ is called a
completely positive map, or a CP map. CP maps are developed primarily
for nonabelian algebras.

The algebra $M_{n}$ of $n\times n$ matrices can be seen as an $n^{2}$-dimensional
Hilbert space with an ONB given by the matrix units $\{e_{ij}\}_{i,j=1}^{n}$.
It is also a $*$-algebra generated by $\{e_{ij}\}_{i,j=1}^{n}$ such
that \[
e_{ij}e_{kl}=\begin{cases}
e_{il} & j=k\\
0 & j\neq k\end{cases}\]
Members of $\mathfrak{A}\otimes M_{n}$ are of the form\[
\sum_{i,j}A_{ij}\otimes e_{ij}.\]
In other words, $\mathfrak{A}\otimes M_{n}$ consists of precisely
the $\mathfrak{A}$-valued $n\times n$ matrices. Similarly, members
of $H\otimes\mathbb{C}^{n}$ are the $n$-tuple column vectors with
$H$-valued entries.

Let $I_{M_{n}}:M_{n}\rightarrow B(\mathbb{C}^{n})$ be the identity
representation of $M_{n}$ onto $B(\mathbb{C}^{n})$. Then, 

\[
\varphi\otimes I_{M_{n}}:\mathfrak{A}\otimes M_{n}\rightarrow B(H)\otimes B(\mathbb{C}^{n})=B(H\otimes\mathbb{C}^{n})\]
\[
\varphi\otimes I_{M_{n}}(\sum_{i,j}A_{ij}\otimes e_{ij})=\sum_{i,j}\varphi(A_{ij})\otimes e_{ij}\]
where the right-hand-side is an $n\times n$ $B(H)$-valued matrix. 
\begin{note}
The algebra $B(\mathbb{C}^{n})$ of bounded operators on $\mathbb{C}^{n}$
is generated by the rank-one operators $I_{M_{n}}(e_{ij})=\ketbra{e_{i}}{e_{j}}$.
Hence the $e_{ij}$ on the left-hand-side is seen as an element in
the algebra $M_{n}$ of $n\times n$ matrices, while on the right-hand-side,
$e_{ij}$ is seen as the rank one operator $\ketbra{e_{i}}{e_{j}}$
$B(\mathbb{C}^{n})$. With Dirac's notation, when we look at $e_{ij}$
as operators, \[
e_{i,j}(e_{k})=\ketbra{e_{i}}{e_{j}}\:\ket{e_{k}}=\begin{cases}
\ket{e_{i}} & j=k\\
0 & j\neq k\end{cases}\]
\[
e_{i,j}e_{kl}=\ketbra{e_{i}}{e_{j}}\;\ketbra{e_{k}}{e_{l}}=\begin{cases}
\ketbra{e_{i}}{e_{l}} & j=k\\
0 & j\neq k\end{cases}\]
which also proves that $I_{M_{n}}$ is in fact an algebra isomorphism.
\end{note}
The complete positivity condition in (\ref{eq:CP}) is saying that
if $\sum_{i,j}A_{ij}\otimes e_{ij}$ is a positive element in the
algebra $\mathfrak{A}\otimes M_{n}$, then the $n\times n$ $B(H)$-valued
matrix $\sum_{i,j}\varphi(A_{ij})\otimes e_{ij}$ is a positive operator
acting on the Hilbert space $H\otimes\mathbb{C}^{n}$, i.e. take any
$v=\sum_{k}v_{k}\otimes e_{k}$ in $H\otimes\mathbb{C}^{n}$,\begin{eqnarray*}
 &  & \iprod{\sum_{l}v_{l}\otimes e_{l}}{(\sum_{i,j}\varphi(A_{ij})\otimes e_{ij})(\sum_{k}v_{k}\otimes e_{k})}\\
 & = & \iprod{\sum_{l}v_{l}\otimes e_{l}}{\sum_{i,j,k}\varphi(A_{ij})v_{k}\otimes e_{ij}(e_{k})}\\
 & = & \iprod{\sum_{l}v_{l}\otimes e_{l}}{\sum_{i,j}\varphi(A_{ij})v_{j}\otimes e_{i}}\\
 & = & \sum_{i,j,l}\iprod{v_{l}}{\varphi(A_{ij})v_{j}}\iprod{e_{l}}{e_{i}}\\
 & = & \sum_{i,j}\iprod{v_{i}}{\varphi(A_{ij})v_{j}}\geq0.\end{eqnarray*}
The CP condition is illustrated in the following diagram. \[
\otimes\begin{cases}
\mathfrak{A}\rightarrow B(H): & A\mapsto\varphi(A)\\
M_{n}\rightarrow M_{n}: & x\mapsto I_{M_{n}}(X)=X\mbox{\,\ (identity representation of }M_{n})\end{cases}\]
The CP condition can be formulated more conveniently using matrix
notation: if the operator matrix $(A_{ij})\in\mathfrak{A}\otimes M_{n}$
is positive, the corresponding operator matrix $(\varphi(A_{ij}))$
is a positive operator in $B(H\otimes\mathbb{C}^{n})$; i.e. for all
$v\in H\otimes\mathbb{C}^{n}$,\begin{eqnarray*}
 &  & \iprod{\sum_{l}v_{l}\otimes e_{l}}{(\sum_{i,j}\varphi(A_{ij})\otimes e_{ij})(\sum_{k}v_{k}\otimes e_{k})}\\
 & = & \left[\begin{array}{cccc}
v_{1} & v_{2} & \cdots & v_{n}\end{array}\right]\left[\begin{array}{cccc}
\varphi(A_{11}) & \varphi(A_{12}) & \cdots & \varphi(A_{1n})\\
\varphi(A_{21}) & \varphi(A_{22}) & \cdots & \varphi(A_{2n})\\
\vdots & \vdots & \ddots & \vdots\\
\varphi(A_{n1}) & \varphi(A_{n2}) & \cdots & \varphi(A_{nn})\end{array}\right]\left[\begin{array}{c}
v_{1}\\
v_{2}\\
\vdots\\
v_{n}\end{array}\right]\geq0.\end{eqnarray*}

\subsection{CP v.s. GNS}

The GNS construction can be reformulated as a special case of the
Stinespring's theorem. Let $\mathfrak{A}$ be a $*$-algebra, given
a state $w:\mathfrak{A}\rightarrow\mathbb{C}$, there exists a triple
$(K_{w},\Omega_{w},\pi_{w})$, such that \[
w(A)=\iprod{\Omega_{w}}{\pi(A)\Omega_{w}}_{w}\]
\[
K_{w}=\overline{span}\{\pi(A)\Omega_{w}:A\in\mathfrak{A}\}.\]
The unit cyclic vector $\Omega\in K_{w}$ generates a one-dimensional
subspace $\mathbb{C}\Omega_{w}\subset K_{w}$. The 1-dimensional Hilbert
space $\mathbb{C}$ is thought of being embedded into $K_{w}$ (possibly
infinite dimensional) via the map $V:\mathbb{C}\mapsto\mathbb{C}\Omega_{w}$. 
\begin{lem}
$V$ is an isometry. $V^{*}V=I_{\mathbb{C}}:\mathbb{C}\rightarrow\mathbb{C}$,
$VV^{*}:K_{w}\rightarrow\mathbb{C}\Omega_{w}$ is the projection from
$K_{w}$ onto the 1-d subspace $\mathbb{C}\Omega$ in $K_{w}$. 
\begin{proof}
Let $t\in\mathbb{C}$, then $\norm{Vt}_{w}=\norm{t\Omega_{w}}_{w}=\abs{t}$.
Hence $V$ is an isometry. For all $\xi\in K_{w}$, \[
\iprod{\xi}{Vt}_{w}=\iprod{V^{*}\xi}{t}_{\mathbb{C}}=t\overline{V^{*}\xi}\Longleftrightarrow tV^{*}\xi=\overline{\iprod{\xi}{Vt}_{w}}\]
by setting $t=1$, we get\[
V^{*}\xi=\overline{\iprod{\xi}{V1}_{w}}=\overline{\iprod{\xi}{\Omega_{w}}_{w}}=\iprod{\Omega_{w}}{\xi}_{w}\Longleftrightarrow V^{*}=\iprod{\Omega_{w}}{\cdot}\]
Therefore, \[
V^{*}Vt=V^{*}(t\Omega_{w})=\iprod{\Omega_{w}}{t\Omega_{w}}=t,\;\forall t\in\mathbb{C}\Longleftrightarrow V^{*}V=I_{\mathbb{C}}\]
\[
VV^{*}\xi=V(\iprod{\Omega_{w}}{\xi}_{w})=\iprod{\Omega_{w}}{\xi}_{w}\Omega_{w},\;\forall\xi\in K_{w}\Longleftrightarrow VV^{*}=\ketbra{\Omega_{w}}{\Omega_{w}}.\]

\end{proof}
\end{lem}
It follows that \begin{eqnarray*}
w(A) & = & \iprod{\Omega_{w}}{\pi(A)\Omega_{w}}_{w}\\
 & = & \iprod{V1}{\pi(A)V1}_{\mathbb{C}}\\
 & = & \iprod{1}{V^{*}\pi(A)V1}_{\mathbb{C}}\\
 & = & V^{*}\pi(A)V.\end{eqnarray*}
In other words, $\pi(A):\Omega_{w}\mapsto\pi(A)\Omega_{w}$ sends
the unit vector $\Omega_{w}$ from the 1-dimensional subspace $\mathbb{C}\Omega_{w}$
to the vector $\pi(A)\Omega_{w}\in K_{w}$, and $\iprod{\Omega_{w}}{\pi(A)\Omega_{w}}_{w}$
cuts off the resulting vector $\pi(A)\Omega_{w}$ and only preserves
component corresponding to the 1-d subspace $\mathbb{C}\Omega_{w}$.
Notice that the unit vector $\Omega_{w}$ is obtained from embedding
the constant $1\in\mathbb{C}$ via the map $V$, $\Omega_{w}=V1$.
In matrix notation, if we identify $\mathbb{C}$ with its image $\mathbb{C}\Omega_{w}$
in $K_{w}$, then $w(A)$ is put into a matrix corner.\[
\pi(A)=\left[\begin{array}{cc}
w(A) & *\\*
 & *\end{array}\right]\]
so that when acting on vectors,\[
w(A)=\left[\begin{array}{cc}
\Omega_{w} & 0\end{array}\right]\left[\begin{array}{cc}
w(A) & *\\*
 & *\end{array}\right]\left[\begin{array}{c}
\Omega_{w}\\
0\end{array}\right]=\left[\begin{array}{cc}
\Omega_{w} & 0\end{array}\right]\pi(A)\left[\begin{array}{c}
\Omega_{w}\\
0\end{array}\right].\]
Equivalently\[
w(A)=P_{1}\pi(A):P_{1}K_{w}\rightarrow\mathbb{C}\]
where $P_{1}=VV^{*}$ is the rank-1 projection onto $\mathbb{C}\Omega_{w}$.

Stinespring's construction is a generalization of the above formulation.
Let $\mathfrak{A}$ be a $*$-algebra, given a CP map $\varphi:\mathfrak{A}\rightarrow B(H)$,
there exists a Hilbert space $K_{\varphi}$, an isometry $V:H\rightarrow K_{\varphi}$,
and a representation $\pi_{w}:\mathfrak{A}\rightarrow K_{\varphi}$,
such that \[
\varphi(A)=V^{*}\pi(A)V,\;\forall A\in\mathfrak{A}.\]
Notice that this construction start with a possibly infinite dimensional
Hilbert space $H$ (instead of the 1-dimensional Hilbert space $\mathbb{C}$),
the map $V$ embeds $H$ into a bigger Hilbert space $K_{\varphi}$.
If $H$ is identified with its image in $K$, then $\varphi(A)$ is
put into a matrix corner,\[
\left[\begin{array}{cc}
\pi(A) & *\\*
 & *\end{array}\right]\qquad\mbox{should this be}\left[\begin{array}{cc}
\varphi(A) & *\\*
 & *\end{array}\right]?\]
so that when acting on vectors,\[
\varphi(A)\xi=\left[\begin{array}{cc}
V\xi & 0\end{array}\right]\left[\begin{array}{cc}
\pi(A) & *\\*
 & *\end{array}\right]\left[\begin{array}{c}
V\xi\\
0\end{array}\right].\]
Stinespring's theorem can then be formulated alternatively as: for
every CP map $\varphi:\mathfrak{A}\rightarrow B(H)$, there is a dilated
Hilbert space $K_{\varphi}\supset H$, a representation $\pi_{\varphi}:\mathfrak{A}\rightarrow B(K_{\varphi})$,
such that \[
\varphi(A)=P_{H}\pi(A)\]
i.e. $\pi(A)$ can be put into a matrix corner. $K_{\varphi}$ is
be chosen as minimal in the sense that \[
K_{\varphi}=\overline{span}\{\pi(A)(Vh):A\in\mathfrak{A},h\in H\}.\]
\[
\xymatrix{
H\ar[r]^{V}\ar[d]_\varphi(A) & K\ar[d]^{\pi(A)} \\
H\ar[r]_{V} & K
}
\]
\begin{note}
$K_{\varphi}\supset H$ comes after the identification of $H$ with
its image in $K_{\varphi}$ under an isometric embedding $V$. We
use $\varphi(A)=P_{H}\pi(A)$ instead of  {}``$\varphi(A)=P_{H}\pi(A)P_{H}$'',
since $\varphi(A)$ only acts on the subspace $H\subset K_{\varphi}$.
\end{note}

\subsection{Stinespring's theorem}
\begin{thm}
(Stinespring) Let $\mathfrak{A}$ be a $*$-algebra. The following
are equivalent:
\begin{enumerate}
\item $\varphi:\mathfrak{A}\rightarrow B(H)$ is a completely positive map,
and $\varphi(1_{\mathfrak{A}})=I_{H}$. 
\item There exists a Hilbert space $K$, an isometry $V:H\rightarrow K$,
and a representation $\pi:\mathfrak{A}\rightarrow B(K)$ such that\[
\varphi(A)=V^{*}\pi(A)V\]
for all $A\in\mathfrak{A}$.
\item If the dilated Hilbert space is taken to be minimum, then it is unique
up to unitary equivalence. i.e. if there are $V_{i},K_{i},\pi_{i}$
such that \begin{eqnarray*}
\varphi(A) & = & V_{i}^{*}\pi_{i}(A)V_{i}\\
K_{i} & = & \overline{span}\{\pi_{i}(A)Vh:A\in\mathfrak{A},h\in H\}\end{eqnarray*}
then there exists a unitary operator $W:K_{1}\rightarrow K_{2}$ so
that\[
W\pi_{1}=\pi_{2}W\]

\end{enumerate}
\end{thm}
\begin{proof}
For the uniqueness: define\[
W\pi_{1}(A)Vh=\pi_{2}(A)Vh\]
then $W$ is an isometry, which follows from the fact that \begin{eqnarray*}
\norm{\pi_{i}(A)Vh}_{K}^{2} & = & \iprod{\pi_{i}(A)Vh}{\pi_{i}(A)Vh}_{K}\\
 & = & \iprod{h}{V^{*}\pi_{i}(A^{*}A)Vh}_{H}\\
 & = & \iprod{h}{\varphi(A^{*}A)h}_{H}.\end{eqnarray*}
Hence $W$ extends uniquely to a unitary operator from $K_{1}$ to
$K_{2}$. To show $W$ intertwines $\pi_{i}$, notice that a typical
vector in $K_{1}$ is $\pi_{1}(A)Vh$, hence\begin{eqnarray*}
W\pi_{1}(B)\pi_{1}(A)Vh & = & W\pi_{1}(BA)Vh\\
 & = & \pi_{2}(BA)Vh\\
 & = & \pi_{2}(B)\pi_{2}(A)Vh\\
 & = & \pi_{2}(B)W\pi_{1}(A)Vh\end{eqnarray*}
therefore $W\pi_{1}=\pi_{2}W$.
\begin{note}
The norm of $\pi_{1}(B)\pi_{1}(A)Vh$ is given by\begin{eqnarray*}
\norm{\pi_{1}(B)\pi_{1}(A)Vh}^{2} & = & \iprod{\pi_{1}(B)\pi_{1}(A)Vh}{\pi_{1}(B)\pi_{1}(A)Vh}\\
 & = & \iprod{h}{V^{*}\pi_{1}(A^{*}B^{*}BA)Vh}\end{eqnarray*}
where \[
\mbox{operator}\mapsto A^{*}(\mbox{operator})A\]
is an automorphism on $B(H)$.
\end{note}
Suppose $\varphi(A)=V^{*}\pi(A)V$. Since positive elements in $\mathfrak{A}\otimes M_{n}$
are sums of the operator matrix\[
\sum_{i,j}A_{i}^{*}A_{j}\otimes e_{ij}=\left[\begin{array}{c}
A_{1}^{*}\\
A_{2}^{*}\\
\vdots\\
A_{n}^{*}\end{array}\right]\left[\begin{array}{cccc}
A_{1} & A_{2} & \cdots & A_{n}\end{array}\right]\]
it suffices to show that\[
\varphi\otimes I_{M_{n}}(\sum_{i,j}A_{i}^{*}A_{j}\otimes e_{ij})=\sum_{i,j}\varphi(A_{i}^{*}A_{j})\otimes I_{M_{n}}(e_{ij})=\sum_{i,j}\varphi(A_{i}^{*}A_{j})\otimes e_{ij}\]
is a positive operator in $B(H\otimes\mathbb{C}^{n})$, i.e. to show
that for all $v\in H\otimes\mathbb{C}^{n}$ \[
\left[\begin{array}{cccc}
v_{1} & v_{2} & \cdots & v_{n}\end{array}\right]\left[\begin{array}{cccc}
\varphi(A_{1}^{*}A_{1}) & \varphi(A_{1}^{*}A_{2}) & \cdots & \varphi(A_{1}^{*}A_{n})\\
\varphi(A_{2}^{*}A_{1}) & \varphi(A_{2}^{*}A_{2}) & \cdots & \varphi(A_{2}^{*}A_{n})\\
\vdots & \vdots & \ddots & \vdots\\
\varphi(A_{n}^{*}A_{1}) & \varphi(A_{n}^{*}A_{2}) & \cdots & \varphi(A_{n}^{*}A_{n})\end{array}\right]\left[\begin{array}{c}
v_{1}\\
v_{2}\\
\vdots\\
v_{n}\end{array}\right]\geq0.\]
This is true, since\begin{eqnarray*}
 &  & \left[\begin{array}{cccc}
v_{1} & v_{2} & \cdots & v_{n}\end{array}\right]\left[\begin{array}{cccc}
\varphi(A_{1}^{*}A_{1}) & \varphi(A_{1}^{*}A_{2}) & \cdots & \varphi(A_{1}^{*}A_{n})\\
\varphi(A_{2}^{*}A_{1}) & \varphi(A_{2}^{*}A_{2}) & \cdots & \varphi(A_{2}^{*}A_{n})\\
\vdots & \vdots & \ddots & \vdots\\
\varphi(A_{n}^{*}A_{1}) & \varphi(A_{n}^{*}A_{2}) & \cdots & \varphi(A_{n}^{*}A_{n})\end{array}\right]\left[\begin{array}{c}
v_{1}\\
v_{2}\\
\vdots\\
v_{n}\end{array}\right]\\
 & = & \sum_{i,j}\iprod{v_{i}}{\varphi(A_{i}^{*}A_{j})v_{j}}_{H}\\
 & = & \sum_{i,j}\iprod{v_{i}}{V^{*}\pi(A_{i}^{*}A_{j})Vv_{j}}_{H}\\
 & = & \sum_{i,j}\iprod{\pi(A_{i})Vv_{i}}{\pi(A_{j})Vv_{j}}_{\varphi}\\
 & = & \norm{\sum_{i}\pi(A_{i})Vv_{i}}_{\varphi}\geq0.\end{eqnarray*}

Conversely, given a completely positive map $\varphi$, we construct
$K_{\varphi}$, $V_{\varphi}$ and $\pi_{\varphi}$ where the subscript
indicates dependence on $\varphi$. Recall that $\varphi:\mathfrak{A}\rightarrow B(H)$
is a CP map if for all $n\in\mathbb{N}$, \[
\varphi\otimes I_{M_{n}}:\mathfrak{A}\otimes M_{n}\rightarrow B(H\otimes M_{n})\]
is positive, and \[
\varphi\otimes I_{M_{n}}(1_{\mathfrak{A}}\otimes I_{M_{n}})=I_{H}\otimes I_{M_{n}}.\]
The condition on the identity element can be stated using matrix notation
as\[
\left[\begin{array}{cccc}
\varphi & 0 & \cdots & 0\\
0 & \varphi & \cdots & 0\\
\vdots & \vdots & \ddots & \vdots\\
0 & 0 & \cdots & \varphi\end{array}\right]\left[\begin{array}{cccc}
1_{\mathfrak{A}} & 0 & \cdots & 0\\
0 & 1_{\mathfrak{A}} & \cdots & 0\\
\vdots & \vdots & \ddots & \vdots\\
0 & 0 & \cdots & 1_{\mathfrak{A}}\end{array}\right]=\left[\begin{array}{cccc}
I_{H} & 0 & \cdots & 0\\
0 & I_{H} & \cdots & 0\\
\vdots & \vdots & \ddots & \vdots\\
0 & 0 & \cdots & I_{H}\end{array}\right].\]

Let $K_{0}$ be the algebraic tensor product $\mathfrak{A}\otimes H$,
i.e. \[
K_{0}=span\{\sum_{finite}A_{i}\otimes\xi_{i}:A\in\mathfrak{A},\xi\in H\}.\]
Define a sesquilinear form $\iprod{\cdot}{\cdot}_{\varphi}:K_{0}\times K_{0}\rightarrow\mathbb{C}$,
\begin{equation}
\iprod{\sum_{i=1}^{n}A_{i}\otimes\xi_{i}}{\sum_{j=1}^{n}B_{j}\otimes\eta_{j}}_{\varphi}:=\sum_{i,j}\iprod{\xi_{i}}{\varphi(A_{i}^{*}B_{j})\eta_{j}}_{H}.\label{eq:sesquilinear}\end{equation}
By the completely positivity condition (\ref{eq:CP}),\[
\iprod{\sum_{i=1}^{n}A_{i}\xi_{i}}{\sum_{j=1}^{n}A_{j}\xi_{j}}_{\varphi}=\sum_{i,j}\iprod{\xi_{i}}{\varphi(A_{i}^{*}A_{j})\xi_{j}}_{H}\geq0\]
hence $\iprod{\cdot}{\cdot}_{\varphi}$ is positive semi-definite.
Let $ker:=\{v\in K_{0}:\iprod{v}{v}_{\varphi}=0\}$. Since the Schwartz
inequality holds for any sesquilinear form, it follows that $ker=\{v\in K_{0}:\iprod{s}{v}_{\varphi}=0,\mbox{ for all }s\in K_{0}\}$,
thus $ker$ is a closed subspace in $K_{0}$. Let $K_{\varphi}$ be
the Hilbert space by completing $K_{0}$ under the norm $\norm{\cdot}_{\varphi}:=\iprod{\cdot}{\cdot}_{\varphi}^{1/2}$. 

Define $V:H\rightarrow K_{0}$, where $V\xi:=1_{\mathfrak{A}}\otimes\xi$.
Then \begin{eqnarray*}
\norm{V\xi}_{\varphi}^{2} & = & \iprod{1_{\mathfrak{A}}\otimes\xi}{1_{\mathfrak{A}}\otimes\xi}_{\varphi}\\
 & = & \iprod{\xi}{\varphi(1_{\mathfrak{A}}^{*}1_{\mathfrak{A}})\xi}_{H}\\
 & = & \iprod{\xi}{\xi}_{H}\\
 & = & \norm{\xi}_{H}^{2}\end{eqnarray*}
which implies that $V$ is an isometry, and $H$ is isometrically
embeded into $K_{0}$. Claim: $V^{*}V=I_{H}$; $VV^{*}$ is a self-adjoint
projection from $K_{0}$ onto the subspace $1_{\mathfrak{A}}\otimes H$.
In fact, for any $A\otimes\eta\in K_{0}$, \begin{eqnarray*}
\iprod{A\otimes\eta}{V\xi}_{\varphi} & = & \iprod{A\otimes\eta}{1_{\mathfrak{A}}\otimes\xi}_{\varphi}\\
 & = & \iprod{\eta}{\varphi(A^{*})\xi}_{H}\\
 & = & \iprod{\varphi(A^{*})^{*}\eta}{\xi}_{H}\\
 & = & \iprod{V^{*}(A\otimes\eta)}{\xi}_{\varphi}\end{eqnarray*}
therefore,\[
V^{*}(A\otimes\eta)=\varphi(A^{*})^{*}\eta.\]
It follows that \[
V^{*}V\xi=V^{*}(1_{\mathfrak{A}}\otimes\xi)=\varphi(1_{\mathfrak{A}}^{*})^{*}\xi=\xi,\;\forall\xi\in H\Leftrightarrow V^{*}V=I_{H}.\]
Moreover, for any $A\otimes\eta\in K_{0}$,\[
VV^{*}(A\otimes\eta)=V(\varphi(A^{*})^{*}\eta)=1_{\mathfrak{A}}\otimes\varphi(A^{*})^{*}\eta.\]

For any $A\in\mathfrak{A}$, let $\pi_{\varphi}(A)(\sum_{j}B_{j}\otimes\eta_{j}):=\sum_{j}AB_{j}\otimes\eta_{j}$
and extend to $K_{\varphi}$. For all $\xi,\eta\in H$,\begin{eqnarray*}
\iprod{\xi}{V^{*}\pi(A)V\eta}_{H} & = & \iprod{V\xi}{\pi(A)V\eta}_{\varphi}\\
 & = & \iprod{1_{\mathfrak{A}}\otimes\xi}{\pi(A)1_{\mathfrak{A}}\otimes\eta}_{\varphi}\\
 & = & \iprod{1_{\mathfrak{A}}\otimes\xi}{A\otimes\eta}_{\varphi}\\
 & = & \iprod{\xi}{\varphi(1_{\mathfrak{A}}^{*}A)\eta}_{H}\\
 & = & \iprod{\xi}{\varphi(A)\eta}_{H}\end{eqnarray*}
hence $\varphi(A)=V^{*}\pi(A)V$ for all $A\in\mathfrak{A}$.
\end{proof}

\section{Comments on Stinespring's theorem}

We study objects, usually the interest is not on the object itself
but function algebras on it. These functions reveal properties of
the object.
\begin{example}
In linear algebra, there is a bijection between inner product structions
on $\mathbb{C}^{n}$ and positive-definite $n\times n$ matrices.
Specifically, $\iprod{\cdot}{\cdot}:\mathbb{C}^{n}\times\mathbb{C}^{n}\rightarrow\mathbb{C}$
is an inner product if and only if there exists a positive definite
matrix $A$ such that\begin{eqnarray*}
\iprod{v}{w} & = & v^{t}Aw\\
 & = & \left[\begin{array}{cccc}
\overline{v}_{1} & \overline{v}_{2} & \cdots & \overline{v}_{n}\end{array}\right]\left[\begin{array}{cccc}
a_{11} & a_{12} & \cdots & a_{1n}\\
a_{21} & a_{22} & \cdots & a_{2n}\\
\vdots & \vdots & \vdots & \vdots\\
a_{n1} & a_{n2} & \cdots & a_{nn}\end{array}\right]\left[\begin{array}{c}
v_{1}\\
v_{2}\\
\vdots\\
v_{n}\end{array}\right]\end{eqnarray*}
for all $v,w\in\mathbb{C}^{n}$. We think of $\mathbb{C}^{n}$ as
$\mathbb{C}$-valuded functions defined on $\{1,2,\ldots,n\}$, then
$\iprod{\cdot}{\cdot}_{A}$ is an inner product built on the function
space. 
\end{example}
This is then extended to infinite dimensional space. 
\begin{example}
If $F$ is a positive definite function on $\mathbb{R}$, then on
$K_{0}=span\{\delta_{x}:x\in\mathbb{R}\}$, $F$ defines a sesquilinear
form $\iprod{\cdot}{\cdot}_{F}:\mathbb{R}\times\mathbb{R}\rightarrow\mathbb{C}$,
where\[
\iprod{\sum_{i}c_{i}\delta_{x_{i}}}{\sum_{j}d_{j}\delta_{x_{j}}}_{F}=\sum_{i,j}\bar{c}_{i}d_{j}F(x_{i},x_{j}).\]
Let $ker(F)=\{v\in K_{0}:\iprod{v}{v}=0\}$, and $ker(F)$ is a closed
subspace in $K_{0}$. We get a Hilbert space $K$ as the completion
of $K_{0}/ker$ under the norm $\norm{\cdot}_{F}:=\iprod{\cdot}{\cdot}_{F}^{1/2}$. 
\end{example}
What if the index set is not $\{1,2,\ldots,n\}$ or $\mathbb{R}$,
but a $*$-algebra?
\begin{example}
Let $X$ be a locally compact Hausdorff space and $\mathfrak{M}$
be a $\sigma$-algebra in $X$. The space of continuous functions
with compact support $C_{c}(X)$ is an abelian $C^{*}$-algebra, where
the $C^{*}$-norm is given by $\norm{f}=\max\{{\abs{f(x)}:x\in X}\}$.
By Riesz's theorem, there is a bijection between positive linear functionals
on $C_{c}(X)$ and Borel measures. A Borel measure $\mu$ is a state
if and only it is a probability measure. 

$\mathfrak{M}$ is an abelian algebra. The associative multiplication
is defined as $AB:=A\cap B$. The identity element is $X$, since
$A\cap X=X\cap A=A$, for all $A\in\mathfrak{M}$. Let $\mu$ be a
probability measure. Then $\mu(A\cap B)\geq0$, for all $A,B\in\mathfrak{M}$,
and $\mu(X)=1$, therefore $\mu$ is a state. As before, we consider
functions on $\mathfrak{M}$, i.e. the index set is the $\sigma$-algebra
$\mathfrak{M}$. Then expressions such as $\sum_{i}c_{i}\delta_{A_{i}}$
are precisely the simple functions $\sum_{i}c_{i}\chi_{A_{i}}$. The
GNS construction starts with \[
K_{0}=span\{\delta_{A}:A\in\mathfrak{M}\}=span\{\chi_{A}:A\in\mathfrak{M}\}\]
and the sesquilinear form $\iprod{\cdot}{\cdot}_{\mu}:K_{0}\times K_{0}\rightarrow\mathbb{C}$,
where\[
\iprod{\sum_{i}c_{i}\chi_{A_{i}}}{\sum_{j}d_{j}\chi_{B_{j}}}:=\sum_{i,j}\overline{c}_{i}d_{j}\mu(A_{i}\cap B_{j}).\]
$\iprod{\cdot}{\cdot}_{\mu}$ is positive semi-definite, since \[
\iprod{\sum_{i}c_{i}\chi_{A_{i}}}{\sum_{i}c_{i}\chi_{A_{i}}}=\sum_{i,j}\overline{c}_{i}c_{j}\mu(A_{i}\cap A_{j})=\sum_{i}\abs{c_{i}}^{2}\mu(A_{i})\geq0.\]
Let $ker=\{v\in K_{0}:\iprod{v}{v}_{\mu}=0\}$, complete $K_{0}/ker$
with the corresponding norm, we actually get the Hilbert space $L^{2}(\mu)$.
\end{example}
Let $\mathfrak{A}$ be a $*$-algebra. The set of $\mathbb{C}$-valued
functions on $\mathfrak{A}$ is precisely $\mathfrak{A}\otimes\mathbb{C}$.
We may think of putting $\mathfrak{A}$ on the horizontal axis, and
at each point $A\in\mathfrak{A}$ attach a complex number to it i.e.
building functions indexed by $\mathfrak{A}$. Then members of $\mathfrak{A}\otimes\mathbb{C}$
are of the form\[
\sum_{i}A_{i}\otimes c_{i}1_{\mathbb{C}}=\sum_{i}c_{i}A_{i}=\sum_{i}c_{i}\delta_{A_{i}}\]
with finite summation over $i$. Note that $\mathbb{C}$ is natually
embedded into $\mathfrak{A}\otimes\mathbb{C}$ as $1_{\mathfrak{A}}\otimes\mathbb{C}$,
i.e. $c\mapsto c\delta_{1_{\mathfrak{A}}}$, and the latter is a 1-dimensional
subspace. In order to build a Hilbert space out of the function space,
a quadradic form is required. A state on $\mathfrak{A}$ does exactly
the job. Let $w$ be a state on $\mathfrak{A}$. Then the sesquilinear
form\[
\iprod{\sum_{i}c_{i}\delta_{A_{i}}}{\sum_{i}d_{j}\delta_{B_{j}}}_{w}:=\sum_{i,j}\overline{c_{i}}d_{j}w(A_{i}^{*}B_{j})\]
is positive semi-definite, because\[
\iprod{\sum_{i}c_{i}\delta_{A_{i}}}{\sum_{i}c_{j}\delta_{A_{j}}}_{w}=\sum_{i,j}\overline{c_{i}}c_{j}w(A_{i}^{*}A_{j})=w\left((\sum_{i}c_{i}A_{i})^{*}(\sum_{i}c_{i}A_{i})\right)\geq0.\]
A Hilbert space $K_{w}$ is obtained by taking completion of the quotient
$\mathfrak{A}\otimes\mathbb{C}/Ker(w)$. A representation $\pi$,
$\pi(A)\delta_{B}:=\delta_{BA}$ is in fact a {}``shift'' in the
index variable, and extend linearly to $K_{w}$.

In Stinespring's construction, $\mathfrak{A}\otimes\mathbb{C}$ is
replaced by $\mathfrak{A}\otimes H$. i.e. in stead of working with
$\mathbb{C}$-valued functions on $\mathfrak{A}$, one looks at $H$-valued
functions on $\mathfrak{A}$. Hence we are looking at functions of
the form\[
\sum_{i}A_{i}\otimes\xi_{i}=\sum_{i}\xi_{i}\delta_{A_{i}}\]
with finite summation over $i$. $H$ is natually embedded into $\mathfrak{A}\otimes H$
as $1_{\mathfrak{A}}\otimes H$, i.e. $\xi\mapsto\xi\delta_{1_{\mathfrak{A}}}$.
$H\delta_{1_{\mathfrak{A}}}$ is infinite dimensional, or we say that
the function $\xi\delta_{1_{\mathfrak{A}}}$ at $\delta_{1_{\mathfrak{A}}}$
has infinite multiplicity. If $H$ is separable, we are actually attaching
an $l^{2}$ sequence at every point $A\in\mathfrak{A}$ on the horizontal
axis. How to build a Hilbert space out of these $H$-valued functions?
The question depends on the choice a quadratic form. If $\varphi:\mathfrak{A}\rightarrow B(H)$
is positive, then \[
\iprod{\xi\delta_{A}}{\eta\delta_{B}}_{\varphi}:=\iprod{\xi}{\varphi(A^{*}B)\eta}_{H}\]
is positive semi-definite. When extend linearly, one is in trouble.
Since \[
\iprod{\sum_{i}\xi_{i}\delta_{A_{i}}}{\sum_{j}\eta_{j}\delta_{B_{j}}}_{\varphi}=\sum_{i,j}\iprod{\xi_{i}}{\varphi(A_{i}^{*}B_{j})\eta_{j}}_{H}\]
which is equal to \[
\left[\begin{array}{cccc}
\xi_{1} & \xi_{2} & \cdots & \xi_{n}\end{array}\right]\left[\begin{array}{cccc}
\varphi(A_{1}^{*}B_{1}) & \varphi(A_{1}^{*}B_{2}) & \cdots & \varphi(A_{1}^{*}B_{n})\\
\varphi(A_{2}^{*}B_{1}) & \varphi(A_{2}^{*}B_{1}) & \cdots & \varphi(A_{2}^{*}B_{1})\\
\vdots & \vdots & \vdots & \vdots\\
\varphi(A_{n}^{*}B_{1}) & \varphi(A_{n}^{*}B_{2}) & \cdots & \varphi(A_{n}^{*}B_{n})\end{array}\right]\left[\begin{array}{c}
\xi_{1}\\
\xi_{2}\\
\vdots\\
\xi_{n}\end{array}\right]\]
it is not clear why the matrix $(\varphi(A_{i}^{*}B_{j}))$ should
be a positive operator acting on $H\otimes\mathbb{C}^{n}$. But we
could very well put this extra requirement into an axiom, and consider
$\varphi$ being a CP map!
\begin{itemize}
\item We only assume $\mathfrak{A}$ is a $*$-algebra, may not be a $C^{*}$-algebra.
$\varphi:\mathfrak{A}\rightarrow B(H)$ is positive does not necessarily
imply $\varphi$ is completely positve. A counterexample is taking
$\mathfrak{A}=M_{2}(\mathbb{C})$, and $\varphi:\mathfrak{A}\rightarrow B(\mathbb{C}^{2})\simeq M_{2}(\mathbb{C})$
given by taking transpose, i.e. $A\mapsto\varphi(A)=A^{tr}$. Then
$\varphi$ is positive, but $\varphi\otimes I_{M_{2}}$ is not.
\item The operator matrix $(A_{i}^{*}A_{j})$, which is also written as
$\sum_{i,j}A_{i}^{*}A_{j}\otimes e_{ij}$ is a positve element in
$\mathfrak{A}\otimes M_{n}$. All positive elements in $\mathfrak{A}\otimes M_{n}$
are in such form. This notation goes back again to Dirac, for the
rank-1 operators $\ketbra{v}{v}$ are positive and all positive operators
are sums of these rank-1 operators. 
\item Given a CP map $\varphi:\mathfrak{A}\rightarrow B(H)$, we get a Hilbert
space $K_{\varphi}$, a representation $\pi:\mathfrak{A}\rightarrow B(K_{\varphi})$
and an isometry $V:H\rightarrow K_{\varphi}$, such that \[
\varphi(A)=V^{*}\pi(A)V\]
for all $A\in\mathfrak{A}$. $P=VV^{*}$ is a self-adjoint projection
from $K_{\varphi}$ to the image of $H$ under the embedding. To check
$P$ is a projection, \[
P^{2}=VV^{*}VV^{*}=V(V^{*}V)V^{*}=VV^{*}.\]
\end{itemize}
\begin{xca}
Let $\mathfrak{A}=B(H)$, $\xi_{1},\ldots,\xi_{n}\in H$. The map
$(\xi_{1},\ldots,\xi_{n})\mapsto(A\xi_{1},\ldots,A\xi_{n})\in\oplus^{n}H$
is a representation of $\mathfrak{A}$ if and only if \[
\overset{n\mbox{ times}}{\overbrace{id_{H}\oplus\cdots\oplus id_{H}}}\in Rep(\mathfrak{A},\overset{n\mbox{ times}}{\overbrace{H\oplus\cdots\oplus H}})\]
where in matrix notation, we have\begin{eqnarray*}
 &  & \left[\begin{array}{cccc}
id_{\mathfrak{A}}(A) & 0 & \cdots & 0\\
0 & id_{\mathfrak{A}}(A) & \cdots & 0\\
\vdots & \vdots & \vdots & \vdots\\
0 & 0 & \cdots & id_{\mathfrak{A}}(A)\end{array}\right]\left[\begin{array}{c}
\xi_{1}\\
\xi_{2}\\
\vdots\\
\xi_{n}\end{array}\right]\\
 & = & \left[\begin{array}{cccc}
A & 0 & \cdots & 0\\
0 & A & \cdots & 0\\
\vdots & \vdots & \vdots & \vdots\\
0 & 0 & \cdots & A\end{array}\right]\left[\begin{array}{c}
\xi_{1}\\
\xi_{2}\\
\vdots\\
\xi_{n}\end{array}\right]=\left[\begin{array}{c}
A\xi_{1}\\
A\xi_{2}\\
\vdots\\
A\xi_{n}\end{array}\right].\end{eqnarray*}
\end{xca}
\begin{note*}
In this case, we say the identity representation $id_{\mathfrak{A}}:\mathfrak{A}\rightarrow H$
has multiplicity $n$.\end{note*}
\begin{xca}
Let $V_{i}:H\rightarrow H$, and \[
V:=\left[\begin{array}{c}
V_{1}\\
V_{2}\\
\vdots\\
V_{n}\end{array}\right]:H\rightarrow\oplus_{1}^{n}H.\]
Let $V^{*}:\oplus_{1}^{n}H\rightarrow H$ be the adjoint of $V$.
Prove that $V^{*}=\left[\begin{array}{cccc}
V_{1}^{*} & V_{2}^{*} & \cdots & V_{n}^{*}\end{array}\right]$. 
\begin{proof}
Let $\xi\in H$, then \[
V\xi=\left[\begin{array}{c}
V_{1}\xi\\
V_{2}\xi\\
\vdots\\
V_{n}\xi\end{array}\right]\]
and \begin{eqnarray*}
\iprod{\left[\begin{array}{c}
\eta_{1}\\
\eta_{2}\\
\vdots\\
\eta_{n}\end{array}\right]}{\left[\begin{array}{c}
V_{1}\xi\\
V_{2}\xi\\
\vdots\\
V_{n}\xi\end{array}\right]} & = & \sum_{i}\iprod{\eta_{i}}{V_{i}\xi}\\
 & = & \sum_{i}\iprod{V_{i}^{*}\eta_{i}}{\xi}\\
 & = & \iprod{\sum_{i}V_{i}^{*}\eta_{i}}{\xi}\\
 & = & \iprod{\left[\begin{array}{cccc}
V_{1}^{*} & V_{2}^{*} & \cdots & V_{n}\end{array}^{*}\right]\left[\begin{array}{c}
\eta_{1}\\
\eta_{2}\\
\vdots\\
\eta_{n}\end{array}\right]}{\xi}\end{eqnarray*}
this shows that $V^{*}=\left[\begin{array}{cccc}
V_{1}^{*} & V_{2}^{*} & \cdots & V_{n}\end{array}^{*}\right]$.
\end{proof}
\end{xca}
Let $V$ as defined above.
\begin{xca}
The following are equivalent:
\begin{enumerate}
\item $V$ is an isometry, i.e. $\norm{V\xi}^{2}=\norm{\xi}^{2}$, for all
$\xi\in H$;
\item $\sum V_{i}^{*}V_{i}=I_{H}$;
\item $V^{*}V=I_{H}$.\end{enumerate}
\begin{proof}
Notice that \[
\norm{V\xi}^{2}=\sum_{i}\norm{V_{i}\xi}^{2}=\sum_{i}\iprod{\xi}{V_{i}^{*}V_{i}\xi}=\iprod{\xi}{\sum_{i}V_{i}^{*}V_{i}\xi}.\]
Hence $\norm{V\xi}^{2}=\norm{\xi}^{2}$ if and only if \[
\iprod{\xi}{\sum_{i}V_{i}^{*}V_{i}\xi}=\iprod{\xi}{\xi}\]
for all $\xi\in H$. Equivalently, $\sum_{i}V_{i}^{*}V_{i}=I_{H}=V^{*}V$.
\end{proof}
\end{xca}
More examples of tensor products.
\begin{xca}
Prove the following.
\begin{itemize}
\item $\oplus_{1}^{n}H\simeq H\otimes\mathbb{C}^{n}$
\item $\sum_{1}^{\oplus\infty}H\simeq H\otimes l^{2}$
\item Given $L^{2}(X,\mathfrak{M},\mu)$, $L^{2}(X,H)\simeq H\otimes L^{2}(\mu)$
where $L^{2}(X,H)$ consists of all measurable functioins $f:X\rightarrow H$
such that\[
\int_{X}\norm{f(x)}_{H}^{2}d\mu(x)<\infty\]
and\[
\iprod{f}{g}=\int_{X}\iprod{f(x)}{g(x)}_{H}d\mu(x).\]

\item Show that all the spaces above are Hilbert spaces.
\end{itemize}
\end{xca}
\begin{cor}
(Krauss, physicist) Let $dimH=n$. Then all the CP maps are of the
form\[
\varphi(A)=\sum_{i}V_{i}^{*}AV_{i}.\]
\end{cor}
\begin{rem*}
This was discovered in the physics literature by Kraus. The original
proof was very intracate, but it is a corollary of Stinespring's theorem.
When $dimH=n$, let $e_{1},\ldots e_{n}$ be an ONB. $V_{i}:e_{i}\mapsto Ve_{i}\in K$
is an isometry, $i=1,2,\ldots,n$. So we get a system of isometries,
and \[
\varphi(A)=\left[\begin{array}{cccc}
V_{1}^{*} & V_{2}^{*} & \cdots & V_{n}^{*}\end{array}\right]\left[\begin{array}{cccc}
A\\
 & A\\
 &  & \ddots\\
 &  &  & A\end{array}\right]\left[\begin{array}{c}
V_{1}\\
V_{2}\\
\vdots\\
V_{n}\end{array}\right].\]
Notice that $\varphi(1)=1$ if and only if $\sum_{i}V_{i}^{*}V_{i}=1$.
\end{rem*}
Using tensor product in representations.
\begin{xca}
$(X_{i},\mathfrak{M}_{i},\mu_{i})$ $i=1,2$ are measure spaces. Let
$\pi_{i}:L^{\infty}(\mu_{i})\rightarrow L^{2}(\mu_{i})$ be the representation
such that $\pi_{i}(f)$ is the operator of multiplication by $f$
on $L^{2}(\mu_{i})$. Hence $\pi\in Rep(L^{\infty}(X_{i}),L^{2}(\mu_{i}))$,
and \[
\pi_{1}\otimes\pi_{2}\in Rep(L^{\infty}(X_{1}\times X_{2}),L^{2}(\mu_{1}\times\mu_{2}))\]
with \[
\pi_{1}\otimes\pi_{2}(\tilde{\varphi})\tilde{f}=\tilde{\varphi}\tilde{f}\]
where $\tilde{\varphi}\in L^{\infty}(X_{1}\times X_{2})$ and $\tilde{f}\in L^{2}(\mu_{1}\times\mu_{2})$.
\end{xca}
More about multiplicaity
\begin{xca*}
dsf
\end{xca*}

\section{More on the CP maps}

Positive maps have been a recursive theme in functional analysis.
An classical example is $\mathfrak{A}=C_{c}(X)$ with a positive linear
functional $\Lambda:\mathfrak{A}\rightarrow\mathbb{C}$, mapping $\mathfrak{A}$
into a 1-d Hilbert space $\mathbb{C}$.

In Stinespring's formulation, $\varphi:\mathfrak{A}\rightarrow H$
is a CP map, then we may write $\varphi(A)=V^{*}\pi(A)V$ where $\pi:\mathfrak{A}\rightarrow K$
is a representation on a bigger Hilbert space $K$containing $H$.
The containment is in the sense that $V:H\hookrightarrow K$ embeds
$H$ into $K$. Notice that \[
V\varphi(A)=\pi(A)V\Longrightarrow\varphi(A)=V^{*}\pi(A)V\]
but not the other way around. (???) In Nelson's notes, we use the
notation $\varphi\subset\pi$ for one representation being the subrepresentation
of another representation. To imitate the situation in linear algebra,
we may want to split an operator $T$ acting on $K$ into operators
action on $H$ and its complement in $K$. Let $P:K\rightarrow H$
be the orthogonal projection. In matrix language,\[
\left[\begin{array}{cc}
PTP & PTP^{\perp}\\
P^{\perp}TP & P^{\perp}TP^{\perp}\end{array}\right].\]
A better looking would be \[
\left[\begin{array}{cc}
PTP & 0\\
0 & P^{\perp}TP^{\perp}\end{array}\right]=\left[\begin{array}{cc}
\varphi_{1} & 0\\
0 & \varphi_{2}\end{array}\right]\]
hence \[
\pi=\varphi_{1}\oplus\varphi_{2}.\]
Stinespring's theorme is more general, where the off-diagonal entries
may not be zero.

\section{Krien-Milman revisited}

We study some examples of compact convex sets in locally convex topological
spaces. Some typical examples include the set of positive semi-definite
functions, taking values in $\mathbb{C}$ or $B(H)$.

The context for Krein-Milman is locally convex topological spaces.
Almost all spaces one works with are locally convex. The Krein-Milman
theorem is in all functional analysis books. Choqute's thoerem comes
later, hence it's not contained in most books. A good reference is
the book by R. Phelps. The proof of Choquet's theorem is not specially
illuminating. It uses standard integration theory.
\begin{thm}
(Krein-Milman) $K$ is a compact convex set in a locally convex topological
space. Then $K$ is equals to the closed convex hull of its extreme
points. $K=cl(conv(E(K)))$.
\end{thm}
A convex combination of points $(\xi_{i})$ in $K$ has the form $v=\sum c_{i}\xi_{i}$,
such that $\sum c_{i}=1$. Closure refers to taking limit, allow all
limits of such convex combinations. Such a $v$ is obviously in $K$,
since $K$ was assumed to be convex. The point of the Krein-Milman's
theorem is the converse. The Krein-Milman theorem is not as useful
as another version by Choquet. 
\begin{thm}
(Choquet) $K$ is a compact convex set in a locally convex topological
space. Let $E(K)$ be the set of extreme points on $K$. Then for
all $p\in K$, there exists a Borel probability measure $\mu_{p}$,
supported on a Borel set $bE(K)\supset E(K)$, such that\[
p=\int_{bE(X)}\xi d\mu(\xi).\]

\end{thm}
The expression in Choquet's theorem is a generalization of convex
combination. In stead of summation, it is an integral against a measure.
Since there are some bazarre cases where the extreme points $E(K)$
do not form a Borel set, the measure $\mu_{p}$ is actually supported
on $bE(K)$, such that $\mu_{p}(bE(K)-E(K))=0$. Examples of such
a decomposition include Fourier transform, Laplace transform, direct
integrals.
\begin{example}
Let $(X,\mathfrak{M},\mu)$ be a measure space, where $X$ is compact
and Hausdorff. The set of all probability measures $\mathcal{P}(X)$
is a convex set. To see this, let $\mu_{1},\mu_{2}\in\mathcal{P}(X)$
and $0\leq t\leq1$, then $t\mu_{1}+(1-t)\mu_{2}$ is a measure on
$X$, moreover $(t\mu_{1}+(1-t)\mu_{2})(X)=t+1-t=1$, hence $t\mu_{1}+(1-t)\mu_{2}\in\mathcal{P}(X)$.
Usually we don't want all probability measures, but a closed subset. 
\end{example}
We compute extreme points in the previous example.
\begin{example}
$K=\mathcal{P}(X)$ is compact convex in $C(X)^{*}$, which is identified
as the set of all measures due to Riesz. $C(X)^{*}$ is a Banach space
hence is always convex. The importance of being the dual of some Banach
space is that the unit ball is always weak $*$ compact. The weak
$*$ topology is just the cylindar topology. The unit ball $B_{1}^{*}$
sits inside the infinite product space (compact, Hausdorff) $\prod_{v\in B,\norm{v}=1}D_{1}$,
where $D_{1}=\{z\in\mathbb{C}:\abs{z}=1\}$. The weak $*$ topology
on $B_{1}^{*}$ is just the restriction of the product topology on
$\prod D_{1}$ onto $B_{1}^{*}$. 
\begin{example}
Claim: $E(K)=\{\delta_{x}:x\in X\}$, where $\delta_{x}$ is the Dirac
measure supported at $x\in X$. By Riesz, to know the measure is to
know the linear functional. $\int fd\delta_{x}=f(x)$. Hence we get
a family of measures indexed by $X$. If $X=[0,1]$, we get a continuous
family of measures. To see there really are extreme points, we do
the GNS contructioin on the algebra $\mathfrak{A}=C(X)$, with the
state $\mu\in\mathcal{P}(X)$. The Hilbert space so constructed is
simply $L^{2}(\mu)$. It's clear that $L^{2}(\delta_{x})$ is 1-dimensional,
hence the representation is irreducible. Therefore $\delta_{x}$ is
a pure state, for all $x\in X$.
\end{example}
\end{example}
\begin{note}
Extreme points: $\nu$ is an extreme point in $\mathcal{P}(X)$ if
and only if $\nu\in[\mu_{1},\mu_{2}]\Rightarrow\nu=\mu_{1}\mbox{ or }\nu=\mu_{2}$.\end{note}
\begin{example}
Let $\mathfrak{A}=B(H)$, and $K=states$. For each $\xi\in H$, the
map $A\mapsto w_{\xi}(A):=\iprod{\xi}{A\xi}$ is a state, called vector
state. Claim: $E(K)=\mbox{vector states}$. To show this, suppose
$W$ is a subspace of $H$ such that $0\varsubsetneq W\varsubsetneq H$,
and suppose $W$ is invariant under the action of $B(H)$. Then $\exists h\in H$,
$h\perp W$. Choose $\xi\in W$. The wonderful rank-1 operator (due
to Dirac) $T:\xi\mapsto h$ given by $T=\ketbra{h}{\xi}$, shows that
$h\in W$. Hence $h\perp h$ and $h=0$. Therefore $W=H$. We say
$B(H)$ acts transitively on $H$.\end{example}
\begin{note}
In general, any $C^{*}$-algebra is a closed subalgebra of $B(H)$
for some $H$. But if we choose $B(H)$, then all the pure states
are vector states.\end{note}
\begin{example}
Let $\mathfrak{A}$ be a $*$-algebra, $S(\mathfrak{A})$ be the set
of states on $\mathfrak{A}$. $w:\mathfrak{A}\rightarrow\mathbb{C}$
is a state if $w(1_{\mathfrak{A}})=1$ and $w(A)\geq0$, whenever
$A\geq0$. Let $\mathfrak{A}$ be a $*$-algebra, then the set of
completely postive maps is a compact convex set. CP maps are generalizations
of states. Back to the first example, $\mathfrak{A}=C(X)$ there is
a bijection between state $\varphi_{\mu}$ and Borel measure $\mu$.
$\varphi(a)=\int ad\mu$. We check that $\varphi(1_{\mathfrak{A}})=\varphi(1)=\int1d\mu=\mu(X)=1$;
and $\varphi(f)=\int g^{2}d\mu\geq0$ for $f\geq0$, $g^{2}=f$.
\end{example}
The next example if taken from AMS as a homework exercise.
\begin{example}
Take the two state sample space $\Omega=\prod_{1}^{\infty}\{0,1\}$
with product topology. Assign probability measure, so that we might
favor one outcome than the other. For example, let $s=x_{1}+\cdots x_{n}$,
$P_{\theta}(C_{x})=\theta^{s}(1-\theta)^{n-1}$, i.e. $s$ heads,
$(n-s)$ tails. Notice that $P_{\theta}$ is invariant under permutation
of coordinates. $x_{1},x_{2},\ldots,x_{n}\mapsto x_{\sigma(1)}x_{\sigma(2)}\ldots x_{\sigma(n)}$.
$P_{\theta}$ is a member of the set of all such invariant measures
(invariant under permutation) $P_{inv}(\Omega)$. Prove that \[
E(P_{inv}(\Omega))=[0,1]\]
i.e. $P_{\theta}$ are all the possible extreme points.\end{example}
\begin{rem}
$\sigma:X\rightarrow X$ is a measurable transformation. $\mu$ is
ergodic (probability measure) if\[
[E\in\mathfrak{M},\sigma E=E]\Rightarrow\mu(E)\in\{0,1\}\]
which intuitively says that the whole space $X$ can't be divided
into parts where $\mu$ is invariant. It has to be mixed up by the
transformation $\sigma$.
\end{rem}

\section{States and representation}

The GNS construction gives rise to a bijection between states and
representations. We consider decomposition of representations or equivalently
states.

The smallest representations are the irreducible ones. A representation
$\pi:\mathfrak{A}\rightarrow B(H)$ is irreducible, if whenever $H$
breaks up into two pieces $H=H_{1}\oplus H_{2}$, where $H_{i}$ is
invariant under $\pi(\mathfrak{A})$, one of them is zero (the other
is $H$). Equivalently, if $\pi=\pi_{1}\oplus\pi_{2}$, where $\pi_{i}=\pi\big|_{H_{i}}$,
then one of them is zero. This is similar to decomposition of natural
nubmers into product of primes. For example, $6=2\times3$, but $2$
and $3$ are primes and they do not dompose further. 

Hilbert spaces are defined up to unitary equivalence. A state $\varphi$
may have equivalent representations on different Hilbert spaces (but
unitarily equivalent), however $\varphi$ does not see the distinction,
and it can only detect equivalent classes of representations. 
\begin{example}
Let $\mathfrak{A}$ be a $*$-algebra. Given two states $s_{1}$ and
$s_{2}$, by the GNS construction, we get cyclic vector $\xi_{i}$,
and representation $\pi_{i}:\mathfrak{A}\rightarrow B(H_{i})$, so
that $s_{i}(A)=\iprod{\xi_{i}}{\pi_{i}(A)\xi_{i}}$, $i=1,2$. Suppose
there is a unitary operator $W:H_{1}\rightarrow H_{2}$, such that
for all $A\in\mathfrak{A}$, \[
\pi_{1}(A)=W^{*}\pi_{2}(A)W.\]
Then \begin{eqnarray*}
\iprod{\xi_{2}}{\pi_{2}(A)\xi_{2}}_{2} & = & \iprod{W\xi_{1}}{\pi_{2}(A)W\xi_{1}}_{1}\\
 & = & \iprod{\xi_{1}}{W^{*}\pi_{2}(A)W\xi_{1}}_{1}\\
 & = & \iprod{\xi_{1}}{\pi_{1}(A)\xi_{1}}_{1}\end{eqnarray*}
i.e. $s_{2}(A)=s_{1}(A)$. Therefore the same state $s=s_{1}=s_{2}$
has two distinct (unitarily equivalent) representations. \end{example}
\begin{rem}
A special case of states are measures. Two representations are mutually
singular $\pi_{1}\perp\pi_{2}$, if and only if two measures are mutually
singular, $\mu_{1}\perp\mu_{2}$. Later, we will follow Nelson's notes
to build Hilbert space out of equivalent classes measures. \end{rem}
\begin{lem}
(Schur) The following are equivalent.
\begin{enumerate}
\item \label{enu:schur-1-1}A representation $\pi:\mathfrak{A}\rightarrow B(H)$
is irreducible 
\item \label{enu:schur-1-2}$(\pi(\mathfrak{A}))'$ is one-dimensional,
i.e. $(\pi(\mathfrak{A}))'=cI$. \end{enumerate}
\begin{proof}
(\ref{enu:schur-1-2})$\Rightarrow$(\ref{enu:schur-1-1}). Suppose
$\pi$ is not irreducible, i.e. $\pi=\pi_{1}\oplus\pi_{2}$. Claim
that (using tensor) $M_{2}\subset(\pi(\mathfrak{A}))'$. Let \[
P_{H_{1}}=\left[\begin{array}{cc}
I_{H_{1}} & 0\\
0 & 0\end{array}\right],\quad P_{H_{2}}=1-P_{H_{1}}=\left[\begin{array}{cc}
0 & 0\\
0 & I_{H_{2}}\end{array}\right]\]
then for all $A\in\mathfrak{A}$, $P_{H_{i}}\pi(A)=\pi(A)P_{H_{i}}$,
$i=1,2$. Hence $(\pi(\mathfrak{A}))'$ has more than one dimension. 

(\ref{enu:schur-1-1})$\Rightarrow$(\ref{enu:schur-1-2}). Suppose
$(\pi(\mathfrak{A}))'$ has more than one dimension. Let $X\in(\pi(\mathfrak{A}))'$,
i.e. \[
X\pi(A)=\pi(A)X,\;\forall A\in\mathfrak{A}.\]
By taking adjoint, $X^{*}\in(\pi(\mathfrak{A}))'$. Hence $X+X^{*}$
is self-adjoint, and $X+X^{*}\neq cI$, since by hypothesis $(\pi(\mathfrak{A}))'$
has more than one dimension. Therefore $X+X^{*}$ has non trivial
spectral projection, by the spectral theorem, i.e. there is self-adjoint
projection $P(E)\notin\{0,I\}$. Let $H_{1}=P(E)H$, and $H_{2}=(I-P(E))H$.
$H_{1}$ and $H_{2}$ are both nonzero proper subspaces of $H$. Since
$P(E)$ commutes with $\pi(A)$, it follows that $H_{1}$ and $H_{2}$
are both invariant under $\pi$.\end{proof}
\begin{cor*}
To test invariant subspaces, one only needs to look at projections
in the commutant. $\pi$ is irreducible if and only if the only projections
in $(\pi(\mathfrak{A}))'$ are $0$ or $I$.\end{cor*}
\begin{rem}
In matrix notation, write\[
\pi(A)=\left[\begin{array}{cc}
\pi_{1}(A) & 0\\
0 & \pi_{2}(A)\end{array}\right].\]
If \[
\left[\begin{array}{cc}
X & Y\\
U & V\end{array}\right]\in(\pi(\mathfrak{A}))'\]
then \begin{eqnarray*}
\left[\begin{array}{cc}
X & Y\\
U & V\end{array}\right]\left[\begin{array}{cc}
\pi_{1}(A) & 0\\
0 & \pi_{2}(A)\end{array}\right] & = & \left[\begin{array}{cc}
X\pi_{1}(A) & Y\pi_{2}(A)\\
U\pi_{1}(A) & V\pi_{2}(A)\end{array}\right]\\
\left[\begin{array}{cc}
\pi_{1}(A) & 0\\
0 & \pi_{2}(A)\end{array}\right]\left[\begin{array}{cc}
X & Y\\
U & V\end{array}\right] & = & \left[\begin{array}{cc}
\pi_{1}(A)X & \pi_{1}(A)Y\\
\pi_{2}(A)U & \pi_{2}(A)V\end{array}\right].\end{eqnarray*}
Hence \begin{eqnarray*}
X\pi_{1}(A) & = & \pi_{1}(A)X\\
V\pi_{2}(A) & = & \pi_{2}(A)V\end{eqnarray*}
i.e. \[
X\in(\pi_{1}(\mathfrak{A}))',\; V\in(\pi_{2}(\mathfrak{A}))'\]
and \begin{eqnarray*}
U\pi_{1}(A) & = & \pi_{2}(A)U\\
Y\pi_{2}(A) & = & \pi_{1}(A)Y.\end{eqnarray*}
i.e. $U,Y\in int(\pi_{1},\pi_{2})$, the set of intertwing operators
of $\pi_{1}$ and $\pi_{2}$.\[
\xymatrix{
&H_{1}\ar[r]^{\pi_{1}(A)}\ar[d]^{U} & H_{1}\ar[d]_{U}&\\
&H_{2}\ar[r]^{\pi_{2}(A)}\ar@/^1pc/@{.>}[lur]^{Y} & H_{2}\ar@/_1pc/@{.>}[rul]_{Y} &\\
}
\]$\pi_{1}$ and $\pi_{2}$ are inequivalent if and only if $int(\pi_{1},\pi_{2})=0$.
If $\pi_{1}=\pi_{2}$, then we say $\pi$ has multiplicity (multiplicity
equals 2), which is equivalent to the commutant being non-abelian
(in the case where $\pi_{1}=\pi_{2}$, $(\pi(\mathfrak{A}))'\simeq M_{2}$.)
\end{rem}
\end{lem}
Schur's lemma addresses all representations. We characterise the relation
between state and its GNS representation, i.e. specilize to the GNS
representation. Given a $*$ algebra $\mathfrak{A}$, the states $S(\mathfrak{A})$
forms a compact convex subset in the unit ball of the dual $\mathfrak{A}^{*}$. 

Let $\mathfrak{A}_{+}$ be the set of positive elements in $\mathfrak{A}$.
Given $s\in S(\mathfrak{A})$, let $t$ be a positive linear functional.
By $t\leq s$, we means $t(A)\leq s(A)$ for all $A\in\mathfrak{A}_{+}$.
We look for relation between $t$ and the commutant $(\pi(\mathfrak{A}))'$. 
\begin{lem}
(Schur-Sakai-Nicodym) Let $t$ be a positive linear functional, and
let $s$ be a state. There is a bijection between $t$ such that $0\leq t\leq s$,
and self-adjoint operator $A$ in the commutant with $0\leq A\leq I$.
The relation is given by \[
t(\cdot)=\iprod{\Omega}{\pi(\cdot)A\Omega}\]
\end{lem}
\begin{rem}
This is an extention of the classical Radon-Nicodym derivative theorem
to the non-commutative setting. We may write $A=dt/ds$. The notation
$0\leq A\leq I$ refers to the partial order of self-adjoint operators.
It means that for all $\xi\in H$, $0\leq\iprod{\xi}{A\xi}\leq\norm{\xi}^{2}$.
See {[}\cite{MR0442701}{]}, {[}\cite{MR0512360}{]} and {[}\cite{MR1468230}{]}.\end{rem}
\begin{proof}
Easy direction, suppose $A\in(\pi(\mathfrak{A}))'$ and $0\leq A\leq I$.
As in many applications, the favorite functions one usually applies
to self-adjoint operators is the squre root function $\sqrt{\cdot}$.
So let's take $\sqrt{A}$. Since $A\in(\pi(\mathfrak{A}))'$, so is
$\sqrt{A}$. We need to show $t(a)=\iprod{\Omega}{\pi(a)A\Omega}\leq s(a)$,
for all $a\geq0$ in $\mathfrak{A}$. Let $a=b^{2}$, then\begin{eqnarray*}
t(a) & = & \iprod{\Omega}{\pi(a)A\Omega}\\
 & = & \iprod{\Omega}{\pi(b^{2})A\Omega}\\
 & = & \iprod{\Omega}{\pi(b)^{*}\pi(b)A\Omega}\\
 & = & \iprod{\pi(b)\Omega}{A\pi(b)\Omega}\\
 & \leq & \iprod{\pi(b)\Omega}{\pi(b)\Omega}\\
 & = & \iprod{\Omega}{\pi(a)\Omega}\\
 & = & s(a).\end{eqnarray*}
Conversely, suppose $t\leq s$. Then for all $a\geq0$, $t(a)\leq s(a)=\iprod{\Omega}{\pi(a)\Omega}$.
Again write $a=b^{2}$. It follows that \[
t(b^{2})\leq s(b^{2})=\iprod{\Omega}{\pi(a)\Omega}=\norm{\pi(b)\Omega}^{2}.\]
By Riesz's theorem, there is a unique $\eta$, so that \[
t(a)=\iprod{\pi(b)\Omega}{\eta}.\]

Conversely, Let $a=b^{2}$, then \[
t(b^{2})\leq s(b^{2})=\iprod{\Omega}{\pi(a)\Omega}=\norm{\pi(b)\Omega}^{2}.\]
i.e. $\pi(b)\Omega\mapsto t(b^{2})$ is a bounded quadratic form.
Therefore, there exists a unique $A\geq0$ such that\[
t(b^{2})=\iprod{\pi(b)\Omega}{A\pi(b)\Omega}.\]
It is easy to see that $0\leq A\leq I$. $A\in(\pi(\mathfrak{A}))'$????\end{proof}
\begin{cor}
Let $s$ be a state. $(\pi,\Omega,H)$ is the corresponding GNS construction.
The following are equivalent.
\begin{enumerate}
\item For all positive linear functional $t$, $t\leq s\Rightarrow t=\lambda s$
for some $\lambda\geq0$.
\item $\pi$ is irreducible.
\end{enumerate}
\end{cor}
\begin{proof}
By Sakai-Nicodym derivative, $t\leq s$ if and only if there is a
self-adjoint operator $A\in(\pi(\mathfrak{A}))'$ so that \[
t(\cdot)=\iprod{\Omega}{\pi(\cdot)A\Omega}\]
Therefore $t=\lambda s$ if and only if $A=\lambda I$.

Suppose $t\leq s\Rightarrow t=\lambda s$ for some $\lambda\geq0$.
Then $\pi$ must be irreducible, since otherwise there exists $A\in(\pi(\mathfrak{A}))'$
with $A\neq cI$, hence $\mathfrak{A}\ni a\mapsto t(a):=\iprod{\Omega}{\pi(a)A\Omega}$
defines a positive linear functional, and $t\leq s$, however $t\neq\lambda s$.
Thus a contradiction to the hypothesis.

Conversely, suppose $\pi$ is irreducible. Then by Schur's lemma,
$(\pi(\mathfrak{A}))'$ is 1-dimensional. i.e. for all $A\in(\pi(\mathfrak{A}))'$,
$A=\lambda I$ for some $\lambda$. Therefore if $t\leq s$, by Sakai's
theorem, $t(\cdot)=\iprod{\Omega}{\pi(\cdot)A\Omega}$. Thus $t=\lambda s$
for some $\lambda\geq0$.\end{proof}
\begin{defn}
A state $s$ is pure if it cannot be broken up into a convex combination
of two distinct states. i.e. for all states $s_{1}$ and $s_{2}$,
$s=\lambda s_{1}+(1-\lambda)s_{2}\Rightarrow s=s_{1}\mbox{ or }s=s_{2}$. 
\end{defn}
The main theorem in this section is a corollary to Sakai's theorem.
\begin{cor}
Let $s$ be a state. $(\pi,\Omega,H)$ is the corresponding GNS construction.
The following are equivalent.
\begin{enumerate}
\item $t\leq s\Rightarrow t=\lambda s$ for some $\lambda\geq0$.
\item $\pi$ is irreducible.
\item $s$ is a pure state.
\end{enumerate}
\end{cor}
\begin{proof}
By Sakai-Nicodym derivative, $t\leq s$ if and only if there is a
self-adjoint operator $A\in(\pi(\mathfrak{A}))'$ so that \[
t(a)=\iprod{\Omega}{\pi(a)A\Omega},\;\forall a\in\mathfrak{A}.\]
Therefore $t=\lambda s$ if and only if $A=\lambda I$.

We show that $(1)\Leftrightarrow(2)$ and $(1)\Rightarrow(3)\Rightarrow(2)$.

$(1)\Leftrightarrow(2)$ Suppose $t\leq s\Rightarrow t=\lambda s$,
then $\pi$ must be irreducible, since otherwise there exists $A\in(\pi(\mathfrak{A}))'$
with $A\neq cI$, hence $t(\cdot):=\iprod{\Omega}{\pi(\cdot)A\Omega}$
defines a positive linear functional with $t\leq s$, however $t\neq\lambda s$.
Conversely, suppose $\pi$ is irreducible. If $t\leq s$, then $t(\cdot)=\iprod{\Omega}{\pi(\cdot)A\Omega}$
with $A\in(\pi(\mathfrak{A}))'$. By Schur's lemma, $(\pi(\mathfrak{A}))'=\{0,\lambda I\}$.
Therefore, $A=\lambda I$ and $t=\lambda s$. 

$(1)\Rightarrow(3)$ Suppose $t\leq s\Rightarrow t=\lambda s$ for
some $\lambda\geq0$. If $s$ is not pure, then $s=cs_{1}+(1-c)s_{2}$
where $s_{1},s_{2}$ are states and $c\in(0,1)$. By hypothesis, $s_{1}\leq s$
implies that $s_{1}=\lambda s$. It follows that $s=s_{1}=s_{2}$. 

$(3)\Rightarrow(2)$ Suppose $\pi$ is not irreducible, i.e. there
is a non trivial projection $P\in(\pi(\mathfrak{A}))'$. Let $\Omega=\Omega_{1}\oplus\Omega_{2}$
where $\Omega_{1}=P\Omega$ and $\Omega_{2}=(I-P)\Omega$. Then\begin{eqnarray*}
s(a) & = & \iprod{\Omega}{\pi(a)\Omega}\\
 & = & \iprod{\Omega_{1}\oplus\Omega_{2}}{\pi(a)\Omega_{1}\oplus\Omega_{2}}\\
 & = & \iprod{\Omega_{1}}{\pi(a)\Omega_{1}}+\iprod{\Omega_{2}}{\pi(a)\Omega_{2}}\\
 & = & \norm{\Omega_{1}}^{2}\iprod{\frac{\Omega_{1}}{\norm{\Omega_{1}}}}{\pi(a)\frac{\Omega_{1}}{\norm{\Omega_{1}}}}+\norm{\Omega_{2}}^{2}\iprod{\frac{\Omega_{2}}{\norm{\Omega_{2}}}}{\pi(a)\frac{\Omega_{2}}{\norm{\Omega_{2}}}}\\
 & = & \norm{\Omega_{1}}^{2}\iprod{\frac{\Omega_{1}}{\norm{\Omega_{1}}}}{\pi(a)\frac{\Omega_{1}}{\norm{\Omega_{1}}}}+(1-\norm{\Omega_{1}}^{2})\iprod{\frac{\Omega_{2}}{\norm{\Omega_{2}}}}{\pi(a)\frac{\Omega_{2}}{\norm{\Omega_{2}}}}\\
 & = & \lambda s_{1}(a)+(1-\lambda)s_{2}(a).\end{eqnarray*}
Hence $s$ is not a pure state.
\end{proof}

\section{Normal states}

The thing that we want to do with representations comes down to the
smallest ones, i.e. the irreducible representations. Let $\mathfrak{A}$
be a $*$-algebra, a representation $\pi:\mathfrak{A}\rightarrow B(H)$
generates a $*$-subalgebra $\pi(\mathfrak{A})$ in $B(H)$. By taking
norm closure, one gets a $C^{*}$-algebra. 

An abstract $C^{*}$-algebra is a Banach $*$-algebra with the axiom
$\norm{a^{*}a}=\norm{a}^{2}$. By Gelfand and Naimark's theorem, all
abstract $C^{*}$-algebras are isometrically isomorphic to closed
subalgebras of $B(H)$, for some Hilbert space $H$. The construction
of $H$ comes down to states $S(\mathfrak{A})$ on $\mathfrak{A}$
and the GNS construction. Let $\mathfrak{A}_{+}$ be the positive
elements in $\mathfrak{A}$. $s\in S(\mathfrak{A})$, $s:\mathfrak{A}\rightarrow\mathbb{C}$
and $s(\mathfrak{A}_{+})\subset[0,\infty)$. For $C^{*}$-algebra,
positive elements can be written $f=(\sqrt{f})^{2}$ by the spectral
theorem. In general, positive elements have the form $a^{*}a$. There
is a bijection between states and GNS representations $Rep(\mathfrak{A},H)$,
where $s(A)=\iprod{\Omega}{\pi(A)\Omega}$.
\begin{example}
$\mathfrak{A}=C(X)$ where $X$ is a compact Hausdorff space. $s_{\mu}$
given by $s_{\mu}(a)=\int ad\mu$ is a state. The GNS constructuion
gives $H=L^{2}(\mu)$, $\pi(f)$ is the operator of multiplication
by $f$ on $L^{2}(\mu)$. $\{\varphi1:\varphi\in C(X)\}$ is dense
in $L^{2}$, where $1$ is the cyclic vector. $s_{\mu}(f)=\iprod{\Omega}{\pi(f)\Omega}=\int1f1d\mu=\int fd\mu$,
which is also seen as the expectation of $f$ in case $\mu$ is a
probability measure.
\end{example}
Breaking up representations corresponds to breaking up states. Irreducibe
representations correspond to pure states which are extreme points
in the states.

Schur's lemma addresses all representations. It says that a representation
$\pi:\mathfrak{A}\rightarrow B(H)$ is irreducible if and only if
$(\pi(\mathfrak{A}))'$ is 1-dimensional. When specialize to the GNS
representation of a given state $s$, this is also equivalent to saying
that for all positive linear functional $t$, $t\leq s\Rightarrow t=\lambda s$
for some $\lambda\geq0$. This latter equivalence is obtained by using
a more general result, which relates $t$ and self-adjoint operators
in the commutant $(\pi(\mathfrak{A}))'$. Specifically, there is a
bijection between $t\leq s$ and $0\leq A\leq I$ for $X_{t}\in\mathbf{(\pi(\mathfrak{A}))'}$,
so that \[
t(\cdot)=\iprod{\Omega}{\pi(\cdot)X_{t}\Omega}\]

If instead of taking the norm closure, but using the strong operator
topology, ones gets a Von Neumann algebra. Von Neumann showed that
the weak closure of $\mathfrak{A}$ is equal to $\mathfrak{A}''$
\begin{cor}
$\pi$ is irreducible $\Longleftrightarrow$ $(\pi(\mathfrak{A}))'$
is 1-dimensional $\Longleftrightarrow$ $(\pi(\mathfrak{A}))''=B(H)$.
\end{cor}
More general states in physics come from the mixture of particle states,
which correspond to composite system. There are called normal states
in mathematics.

$\rho:H\rightarrow H$ where $\rho\in T_{1}H$ (trace class operators)
with $\rho>0$ and $tr(\rho)=1$. Define state $s_{\rho}(a)=tr(a\rho)$.
Since $\rho$ is compact, by spectral theorem of compact operators,
\[
\rho=\sum_{k}\lambda_{k}P_{k}\]
such that $\lambda_{1}>\lambda_{2}>\rightarrow0$; $\sum\lambda_{k}=1$
and $P_{k}=\ketbra{\xi_{k}}{\xi_{k}}$ are the rank-1 projections. 
\begin{itemize}
\item $s_{\rho}(I)=tr(\rho)=1$
\item $s_{\rho}(a)=tr(a\rho)=\iprod{\xi_{k}}{a\sum_{k}\lambda_{k}P_{k}\xi_{k}}=\sum_{k}\bra{\xi_{k}}\lambda_{k}(\ketbra{\xi_{k}}{\xi_{k}})\ket{a\xi_{k}}=\sum_{k}\lambda_{k}\braket{\xi_{k}}{a\xi_{k}}$,
i.e. \[
s_{\rho}=\sum_{k}\lambda_{k}s_{\xi_{k}}\]
$s_{\rho}$ is a convex combination of pure states $s_{\xi_{k}}$.
\item oberserve that $tr(\ketbra{\xi}{\eta})=\iprod{\eta}{\xi}$. In fact,
take any onb $\{e_{n}\}$ then \begin{eqnarray*}
tr(\ketbra{\xi}{\eta}) & = & \sum_{n}\bra{e_{n}}(\ketbra{\xi}{\eta})\ket{e_{n}}\\
 & = & \sum_{n}\braket{e_{n}}{\xi}\braket{\eta}{e_{n}}\\
 & = & \iprod{\eta}{\xi}\end{eqnarray*}
where the last line comes from Parseval identity.
\item If we drop the condition $\rho\geq0$ then we get the duality $(T_{1}H)^{*}=B(H)$.
\end{itemize}

\section{Kadison-Singer conjecture}

\subsection{Dictionary of OP and QM}
\begin{itemize}
\item states - unit vectors $\xi\in H$. These are all the pure states on
$B(H)$.
\item observable - self-adjoint operators $A=A^{*}$
\item measurement - spectrum
\end{itemize}
The spectral theorem was developed by Von Neumann and later improved
by Dirac and others. A self-adjoint operator $A$ corresponds to a
quantum observable, and result of a quantum measurement can be represented
by the spectrum of $A$. 
\begin{itemize}
\item simple eigenvalue: $A=\lambda\ketbra{\xi_{\lambda}}{\xi_{\lambda}}$,
 \[
s_{\xi_{\lambda}}(A)=\iprod{\xi_{\lambda}}{A\xi_{\lambda}}=\lambda\in sp(A)\subset\mathbb{R}.\]

\item compact operator: $A=\sum_{\lambda}\lambda\ketbra{\xi_{\lambda}}{\xi_{\lambda}}$,
such that $(\xi_{\lambda})$ is an ONB of $H$. If $\xi=\sum c_{\lambda}\xi_{\lambda}$
is a unit vector, then\[
s_{\xi}(A)=\iprod{\xi}{A\xi}=\sum_{\lambda}\lambda\abs{c_{\lambda}}^{2}\]
where $({c_{\lambda}}^{2})_{\lambda}$ is a probability distribution
over the spectrum of $A$, and $s_{\xi}$ is the expectation value
of $A$.
\item more general, allow continuous spectrum: \begin{eqnarray*}
A & = & \int\lambda E(d\lambda)\\
A\xi & = & \int\lambda E(d\lambda)\xi\end{eqnarray*}
and\[
s_{\xi}(A)=\iprod{\xi}{A\xi}=\int\lambda\norm{E(d\lambda)\xi}^{2}.\]
We may write the unit vector $\xi$ as \[
\xi=\int\overset{\xi_{\lambda}}{\overbrace{E(d\lambda)\xi}}\]
so that\[
\norm{\xi}^{2}=\int\norm{E(d\lambda)\xi}^{2}=\int c(d\lambda)^{2}=1\]
where $c(d\lambda)^{2}=\norm{E(d\lambda)\xi}^{2}$. It is clear that
$(c(d\lambda))_{\lambda}$ is a probability distribution on spectrum
of $A$. $s_{\xi}(A)$ is again seen as the expectation value of $A$
with respect $(c(d\lambda))_{\lambda}$, since\[
s_{\xi}(A)=\iprod{\xi}{A\xi}=\int\lambda c(d\lambda)^{2}.\]

\end{itemize}

\subsection{Kadison-Singer conjecture (see Palle's private conversation)}

Dirac gave a lecture at Columbia university in the late 1950's, in
which he claimed without proof that pure states on $l^{\infty}$ extends
uniquely on $B(H)$. Two students Kadison and Singer sitting in the
audience were skeptical about whether Dirac knew what it meant to
be an extension. They later formuated the conjecture in a joint paper.
\begin{rem*}
Isadon Singer. Atiyah(abel price)-Singer(nobel price)
\end{rem*}
Let $H$ be a Hilbert space with an ONB $\{e_{n}\}_{n=1}^{\infty}$.
$P_{n}=\ketbra{e_{n}}{e_{n}}$ is a rank-one projection. Denote by
$D$ the set of all diagonal operators, i.e. $D$ is the span of \[
\sum\lambda_{n}P_{n},\;(\lambda_{n})\in l^{\infty}\]
and $D$ is a subalgebra of $B(H)$. 
\begin{conjecture}
Does every pure state on the subalgebra $D\approx l^{\infty}$ extend
uniquely to a pure state on $B(H)$?
\end{conjecture}
The difficulty lies in the fact that it's hard to find all states
on $l^{\infty}$. (the dual of $l^{\infty}$????) It is conceivable
that a pure state on $l^{\infty}$ may have two unit vectors on $B(H)$.
\begin{lem}
Pure states on $B(H)$ are unit vectors. Let $u\in H$ such that $\norm{u}=1$.
Then \[
w_{u}(A)=\iprod{u}{Au}\]
is a pure state. All pure states on $B(H)$ are of this form.\end{lem}
\begin{rem*}
It is in fact the equivalent class of unit vectors that are the pure
states on $B(H)$. Since \[
\iprod{e^{i\theta}u}{Ae^{i\theta}u}=\iprod{u}{Au}.\]
Equivalently, pure states sit inside the projective vector space.
In $\mathbb{C}^{n+1}$, this is $\mathbb{C}P^{n}$. 
\end{rem*}
Since $l^{\infty}$ is an abelian algebra, by Gelfand's theorem, $l^{\infty}\simeq C(X)$
for some compact Hausdorff space $X$. $X=\beta\mathbb{N}$, the Stone-Cech
compactification of $\mathbb{N}$. Points in $\beta\mathbb{N}$ are
called ultrafilters. Pure states on $l^{\infty}$ correspond to pure
states on $C(\beta\mathbb{N})$, i.e. Dirac measures on $\beta\mathbb{N}$.

Let $s$ be a pure state on $l^{\infty}$. Use Hahn-Banach theorem
to extend $s$, as a linear functional, from $l^{\infty}$ to $\tilde{s}$
on the Banach space $B(H)$. However, Hahn-Banach theorem doesn't
gurantee the extension is a \emph{pure state}. Let $E(s)$ be the
set of all states on $B(H)$ which extend $s$. Since $\tilde{s}\in E(s)$,
$E(s)$ is nonempty. $E(s)$ is compact convex in the weak $*$ topology.
By Krein-Milman's theorem, $E(s)=\mbox{closure(Extreme Points)}$.
Any extreme point will then be a pure state extension of $s$. But
which one to choose? It's the uniqueness part that is the famous conjecture.

\chapter{Appliations to Groups}

\section{More on representations}

\subsection{Motivations}

Every group $G$ is also a $*$ semigroup (Nagy), wherer \[
g^{*}:=g^{-1}.\]
It is obvious that $g**=g$. But $G$ is not a complex $*$ algebra
yet, in particular, multiplication by a complex numbert is not defined.
However, we may work with $\mathbb{C}$-valued functions on $G$,
which is clearly a $*$ algebra. 
\begin{itemize}
\item multiplication\[
(g\otimes c_{g})(h\otimes c_{h})=gh\otimes c_{g}c_{h}\]

\item scalar multiplication \[
t(g\otimes c_{g})=g\otimes tc_{g},\;\forall t\in\mathbb{C}\]

\item $*$ operation\[
(g\otimes c_{g})^{*}=g^{-1}\otimes\overline{c_{g^{-1}}}\]
\end{itemize}
\begin{note}
The $*$ operation so defined is the only choice in order to have
the properties $(ab)^{*}=b^{*}a^{*}$ and $(ta)^{*}=\bar{t}a^{*}$,
$t\in\mathbb{C}$.

What if $(g\otimes c_{g})^{*}:=g*\otimes c_{g^{*}}$, without the
complex conjugation? Then $(t(g\otimes c_{g}))^{*}\neq\bar{t}(g\otimes c_{g})^{*}$.

What is $(g\otimes c_{g})^{*}:=g\otimes\overline{c_{g}}$, without
taking $g^{*}$? This also satisfies $(g\otimes c_{g})^{**}=(g\otimes c_{g})^{*}$.
But\begin{alignat*}{1}
((g\otimes c_{g})(h\otimes c_{h}))^{*} & =(gh\otimes c_{g}c_{h})^{*}=gh\otimes\overline{c_{g}}\overline{c_{h}}\\
(h\otimes c_{h}){}^{*}(g\otimes c_{g})^{*} & =(h\otimes\overline{c_{h}})(g\otimes\overline{c_{g}})=hg\otimes\overline{c_{h}}\overline{c_{g}}.\end{alignat*}

\end{note}
So the idea is instead of working with $G$, look at the $*$-algebra
$\mathfrak{A}=span\{g\otimes c_{g}\simeq c_{g}:g\in G\}$. There is
a bijection between representation of $G$ and representation of $\mathfrak{A}$.
Suppose $U\in Rep(G,H)$, then it extends to a representation of $\mathfrak{A}$
as\[
U:g\otimes c_{g}\mapsto U_{g}\otimes c_{g}=c_{g}U_{g}.\]
Conversely, if $T\in Rep(\mathfrak{A},H)$, then restriction to $g\otimes1$
gives a representation of $G$.

\subsection{Group - algebra - representation}

Everything we say about algebras is also true for groups. In physics,
we are interested in representation of symmetry groups, which preserve
inner product or energy. We always want unitary representations. Irreducible
representation amounts to elementary particles which can not be broken
up further. In practice, quite a lot work goes into finding irreducible
representations of symmetry groups. The idea is to go from groups
to algebras and then to representations. \[
G\rightarrow\mathfrak{A}\rightarrow\pi\]

\begin{itemize}
\item $\pi_{G}\in Rep(G,H)$\[
\begin{cases}
\pi(g_{1}g_{2}) & =\pi(g_{1})\pi(g_{2})\\
\pi(e_{G}) & =I_{H}\\
\pi(g)^{*} & =\pi(g^{-1})\end{cases}\]

\item $\pi_{\mathfrak{A}}\in Rep(\mathfrak{A},H)$\[
\begin{cases}
\pi(A_{1}A_{2}) & =\pi(A_{1})\pi(A)_{2}\\
\pi(1_{\mathfrak{A}}) & =I_{H}\\
\pi(A)^{*} & =\pi(A^{*})\end{cases}\]
\end{itemize}
\begin{caseenv}
\item $G$ is discrete $\longrightarrow$ $\mathfrak{A}=G\otimes l^{1}$\begin{eqnarray*}
\left(\sum_{g}a(g)g\right)\left(\sum_{g}b(h)h\right) & = & \sum_{g,h}a(g)b(h)gh\\
 & = & \sum_{g'}\sum_{h}a(g'h^{-1})b(h)g'\end{eqnarray*}
\[
\left(\sum_{g}c(g)g\right)^{*}=\left(\sum_{g}\overline{c(g^{-1})}g\right)\]
where $c^{*}(g)=\overline{c(g^{-1})}$. The multiplication of functions
in $\mathfrak{A}$ is a generalization of convolutions.
\item $G$ is locally compact $\longrightarrow$ $\mathfrak{A}=G\otimes L^{1}(\mu)\simeq L^{1}(G)$.
\\
\\
Existance of Haar measure: easy proof for compact groups, and extend
to locally compact cases. For non compact groups, the left / right
Haar measures could be different. If they are always equal, the group
is called unimodular. Many non compact groups have no Haar measure.\\
\\
Let $\mathfrak{B}(G)$ be the Borel $\sigma$-aglebra of $G$,
$E\in\mathfrak{B}(G)$. Left Haar: $\lambda_{L}(gE)=\lambda_{L}(E)$;
right Haar: $\lambda_{R}(Eg)=\lambda_{R}(Eg)$. Theorem: the two measures
are equivalent\[
\lambda_{L}\ll\lambda_{R}\ll\lambda_{L}\]
\[
\triangle_{G}=\frac{d\lambda_{L}}{d\lambda_{R}}:G\rightarrow G\]
is called the modular function, $\triangle_{G}$ is a homomorphism,
$\triangle_{G}(gh)=\triangle_{G}(g)\triangle_{G}(h)$.\\
\\
$L^{1}(G)$\begin{eqnarray*}
(\varphi_{1}\star\varphi_{2})(g) & = & \int\varphi_{1}(gh^{-1})\varphi_{2}(h)d\lambda_{R}(h)\\
\varphi^{*}(g) & = & \overline{\varphi(g^{-1})}\triangle(g^{-1})\end{eqnarray*}
Take $C^{*}$ completion in both cases!

\begin{note*}
(1) for the case $l^{1}(G)$, the counting measure in unimodular,
hence $\triangle(g)$ does not appear; (2) In $\varphi_{1}(gh^{-1})$,
$h^{-1}$ appears since operation on points is dual to operation on
functions. A change of variable shows\[
\int\varphi(gh^{-1})d\lambda_{R}(g)=\int\varphi(g')d\lambda_{R}(g'h)=\int\varphi(g')d\lambda_{R}(g')\]
(3) $L^{1}(G)$ is a Banach algebra. Fubini's theorem shows that $f\star g\in L^{1}(G)$,
for all $f,g\in L^{1}(G)$.
\end{note*}
\end{caseenv}
There is a bijection between representations of groups and representations
of algebras. Given a unitary representation $\pi\in Rep(G,H)$, take
a Haar measre $dg$ in $L^{1}(G)$, then we get the group algebra
representation \begin{eqnarray*}
\pi_{L^{1}(G)}(\varphi) & = & \int_{G}\varphi(g)\pi(g)dg\\
\pi_{L^{1}(G)}(\varphi_{1}\star\varphi_{2}) & = & \pi_{L^{1}(G)}(\varphi_{1})\pi_{L^{1}(G)}(\varphi_{2})\\
\pi_{L^{1}(G)}(\varphi)^{*} & = & \pi_{L^{1}(G)}(\varphi^{*})\end{eqnarray*}
Conversely, given a representation of $L^{1}(G)$, let $(\varphi_{i})$
be a sequence in $L^{1}$ such that $\varphi_{i}\rightarrow\delta_{g}$
. Then \[
\int\varphi_{i}(h)\pi(h)gdh\rightarrow\pi(g)\]
i.e. the limit is a representation of $G$.

\section{Some examples}

\subsection{ax+b group}

\[
\left[\begin{array}{cc}
a & b\\
0 & 1\end{array}\right]\: a\in\mathbb{R}_{+},\: b\in\mathbb{R}\]
 \[
\left[\begin{array}{cc}
a' & b'\\
0 & 1\end{array}\right]\left[\begin{array}{cc}
a & b\\
0 & 1\end{array}\right]=\left[\begin{array}{cc}
a'a & a'b+b'\\
0 & 1\end{array}\right]\]
\[
\left[\begin{array}{cc}
a & b\\
0 & 1\end{array}\right]^{-1}=\left[\begin{array}{cc}
\frac{1}{a} & -\frac{b}{a}\\
0 & 1\end{array}\right]\]
The multiplication $aa'$ can be made into addition, by taking $a=e^{t}$,
$a'=e^{t'}$ so that $e^{t}e^{t'}=e^{t+t'}$. This is a transformation
group\[
x\mapsto ax+b\]
where composition gives\[
x\mapsto ax+b\mapsto a'(ax+b)+b'=aa'x+(a'b+b')\]
Left Haar measure: \[
\int f(h^{-1}g)g^{-1}dg=\int f(g')(hg')^{-1}d(hg')=\int f(g')g'^{-1}dg^{-1}\]
 \[
d\lambda_{L}=g^{-1}dg=\left[\begin{array}{cc}
\frac{1}{a} & -\frac{b}{a}\\
0 & 1\end{array}\right]\left[\begin{array}{cc}
da & db\\
0 & 0\end{array}\right]=\frac{dadb}{a^{2}}\]
check: \[
g=\left[\begin{array}{cc}
a & b\\
0 & 1\end{array}\right],\: h=\left[\begin{array}{cc}
a' & b'\\
0 & 1\end{array}\right],\: h^{-1}g=\left[\begin{array}{cc}
\frac{1}{a'} & -\frac{b'}{a'}\\
0 & 1\end{array}\right]\left[\begin{array}{cc}
a & b\\
0 & 1\end{array}\right]=\left[\begin{array}{cc}
\frac{a}{a'} & \frac{b-b'}{a'}\\
0 & 1\end{array}\right]\]
\begin{eqnarray*}
\int f(h^{-1}g)d\lambda_{L}(g) & = & \int f(\frac{a}{a'},\frac{b-b'}{a'})\frac{dadb}{a^{2}}\\
 & = & \int f(s,t)\frac{d(a's)d(a't+b')}{(a's)(a's)}\\
 & = & \int f(s,t)\frac{dsdt}{s^{2}}\end{eqnarray*}
where with a change of variable \[
s=\frac{a}{a'},da=a'ds\]
\[
t=\frac{b-b'}{a'},\: db=a'dt\]
\[
\frac{dadb}{a^{2}}=\frac{a'^{2}dsdt}{(sa')^{2}}=\frac{dsdt}{s^{2}}\]
Right Haar:\begin{eqnarray*}
\int f(gh^{-1})(dg)g^{-1} & = & \int f(g')d(g'h)(g'h)^{-1}\\
 & = & \int f(g')(dg')(hh^{-1})g'^{-1}\\
 & = & \int f(g')(dg')g'^{-1}\end{eqnarray*}
 \[
d\lambda_{R}=(dg)g^{-1}=\left[\begin{array}{cc}
da & db\\
0 & 0\end{array}\right]\left[\begin{array}{cc}
\frac{1}{a} & -\frac{b}{a}\\
0 & 1\end{array}\right]=\frac{dadb}{a}\]
check: \[
g=\left[\begin{array}{cc}
a & b\\
0 & 1\end{array}\right],\: h=\left[\begin{array}{cc}
a' & b'\\
0 & 1\end{array}\right],\: gh^{-1}=\left[\begin{array}{cc}
a & b\\
0 & 1\end{array}\right]\left[\begin{array}{cc}
\frac{1}{a'} & -\frac{b'}{a'}\\
0 & 1\end{array}\right]=\left[\begin{array}{cc}
\frac{a}{a'} & -\frac{ab'}{a'}+b\\
0 & 1\end{array}\right]\]
\begin{eqnarray*}
\int f(gh^{-1})d\lambda_{R}(g) & = & \int f(\frac{a}{a'},-\frac{ab'}{a'}+b)\frac{dadb}{a}\\
 & = & \int f(s,t)\frac{a'dsdt}{a's}\\
 & = & \int f(s,t)\frac{dsdt}{s}\end{eqnarray*}
where with a change of variable \[
s=\frac{a}{a'},da=a'ds\]
\[
t=-\frac{ab'}{a'}+b,\: db=dt\]
\[
\frac{dadb}{a}=\frac{a'dsdt}{a's}=\frac{dsdt}{s}\]

\section{Induced representation}

Two questions involved: (1) How to get a representation of a group
$G$ from a representation of the a subgroup $\Gamma\subset G$? (2)
Given a representation of a group $G$, how to test whether it is
induced from a representation of a subgroup $\Gamma$? The main examples
we have looked at so far are
\begin{itemize}
\item $ax+b$
\item Heisenberg
\item $SL_{2}(\mathbb{R})$
\item Lorents
\item Poincare
\end{itemize}
Among these, the $ax+b$, Heisenberg and Poincare groups are semi-direct
product groups. Their representations are induced from a smaller normal
subgroup. It is extremely easy to find representations of abelian
subgroups. Unitary representation of abelian subgroups are one-dimensional,
but the induced representation on an enlarged Hilbert space is infinite
dimensional. See the appendix for a quick review of semi-direct product.
\begin{example}
The $ax+b$ group ($a>0$). $G=\{(a,b)\}$ where $(a,b)=\left[\begin{array}{cc}
a & b\\
0 & 1\end{array}\right]$. The multiplication rule is given by \begin{eqnarray*}
(a,b)(a',b') & = & (aa',b+ab')\\
(a,b)^{-1} & = & (\frac{1}{a},-\frac{b}{a}).\end{eqnarray*}
The subgroup $\Gamma=\{(1,b)\}$ is one-dimensional, normal and abelian. 
\begin{itemize}
\item abelian: $(1,b)(1,c)=(1,c+b)$
\item normal: $(x,y)(1,b)(x,y)^{-1}=(1,xb)$, note that this is also $Ad_{g}$
acting on the normal subgrup $\Gamma$
\item The other subgroup $\{(a,0)\}$ is isomorphic to the multicative group
$(\mathbb{R}_{+},\times)$. Because we have \[
(a,0)(a',0)=(aa',0)\]
by the group multiplication rule above. 
\item Notice that $(\mathbb{R}_{+},\times)$ is not a normal subgroup, since
$(a,b)(x,0)(\frac{1}{a},-\frac{b}{a})=(ax,b)(\frac{1}{a},-\frac{b}{a})=(x-bx+b)$.
\end{itemize}
The multiplicative group $(\mathbb{R}_{+},\times)$ acts on the additive
group $(\mathbb{R},+)$ by\begin{eqnarray*}
\varphi:(\mathbb{R}_{+},\times) & \mapsto & Aut((\mathbb{R},+))\\
\varphi_{a}(b) & = & ab\end{eqnarray*}
check:\begin{eqnarray*}
(a,b)(a',b') & = & (aa',b+\varphi_{a}(b'))\\
 & = & (aa',b+ab')\\
(a,b)^{-1} & = & (a^{-1},\varphi_{a^{-1}}(b^{-1}))\\
 & = & (a^{-1},a^{-1}(-b))=(\frac{1}{a},-\frac{b}{a})\\
(a,b)(1,x)(a,b^{-1}) & = & (a,b+\varphi_{a}(x))(a,b^{-1})\\
 & = & (a,b+ax)(\frac{1}{a},-\frac{b}{a})\\
 & = & (1,b+ax-b)\\
 & = & (1,ax)\\
 & = & \varphi_{a}(x)\end{eqnarray*}

\begin{example}
The Lie algebra of $G$ is given by $X=\left[\begin{array}{cc}
1 & 0\\
0 & 0\end{array}\right]$, $=\left[\begin{array}{cc}
0 & 1\\
0 & 0\end{array}\right]$. We check that\[
e^{tX}=\left[\begin{array}{cc}
e^{t} & 0\\
0 & 1\end{array}\right]\]
which is subgroup $(\mathbb{R}_{+},\times)$; and \[
e^{sY}=I+sY+0+\cdots+0=\left[\begin{array}{cc}
1 & s\\
0 & 1\end{array}\right]\]
which is subgroup $(\mathbb{R},+)$. We also have $[X,Y]=Y$.

Form $L^{2}(\mu_{L})$ where $\mu_{L}$ is the left Haar measure.
Then $\pi:g\rightarrow\pi(g)f(x)=f(g^{-1}x)$ is a unitary representation.
Specifically, if $g=(a,b)$ then\[
f(g^{-1}x)=f(\frac{x}{a},\frac{y-b}{a}).\]
Differentiate along the $a$ direction we get\begin{eqnarray*}
\tilde{X}f & = & \frac{d}{da}\big|_{a=1,b=0}f(\frac{x}{a},\frac{y-b}{a})=(-x\frac{\partial}{\partial x}-y\frac{\partial}{\partial y})f(x,y)\\
\tilde{Y}f & = & \frac{d}{db}\big|_{a=1,b=0}f(\frac{x}{a},\frac{y-b}{a})=-\frac{\partial}{\partial y}f(x,y)\end{eqnarray*}
therefore we have the vector field \begin{eqnarray*}
\tilde{X} & = & -x\frac{\partial}{\partial x}-y\frac{\partial}{\partial y}\\
\tilde{Y} & = & -\frac{\partial}{\partial y}\end{eqnarray*}
or equivalently we get the Lie algebra representation $d\pi$ on $L^{2}(\mu_{L})$.
Notice that\begin{eqnarray*}
[\tilde{X},\tilde{Y}] & = & \tilde{X}\tilde{Y}-\tilde{Y}\tilde{X}\\
 & = & (-x\frac{\partial}{\partial x}-y\frac{\partial}{\partial y})(-\frac{\partial}{\partial y})-(-\frac{\partial}{\partial y})(-x\frac{\partial}{\partial x}-y\frac{\partial}{\partial y})\\
 & = & x\frac{\partial^{2}}{\partial x\partial y}+y\frac{\partial^{2}}{\partial y^{2}}-(x\frac{\partial^{2}}{\partial x\partial y}+\frac{\partial}{\partial y}+y\frac{\partial^{2}}{\partial y^{2}})\\
 & = & -\frac{\partial}{\partial y}\\
 & = & \tilde{Y}.\end{eqnarray*}
Notice that $\tilde{X}$ and $\tilde{Y}$ can be obtained by the exponential
map as well.\begin{eqnarray*}
\tilde{X}f & = & \frac{d}{dt}\big|_{t=0}f(e^{-tX}x)\\
 & = & \frac{d}{dt}\big|_{t=0}f((e^{-t},1)(x,y))\\
 & = & \frac{d}{dt}\big|_{t=0}f(e^{-t}x,e^{-t}y+1)\\
 & = & (-x\frac{\partial}{\partial x}-y\frac{\partial}{\partial y})f(x,y)\\
\tilde{Yf} & = & \frac{d}{dt}\big|_{t=0}f(e^{-tY}x)\\
 & = & \frac{d}{dt}\big|_{t=0}f((1,-t)(x,y))\\
 & = & \frac{d}{dt}\big|_{t=0}f(x,y-t)\\
 & = & -\frac{\partial}{\partial y}f(x,y)\end{eqnarray*}

\begin{example}
We may parametrize the Lie algebra of the ax+b group using $(x,y)$
variables. Build the Hilbert space $L^{2}(\mu_{L})$. The unitary
representation $\pi(g)f(\sigma)=f(g^{-1}\sigma)$ induces the follows
representations of the Lie algebra\begin{eqnarray*}
d\pi(s)f(\sigma) & = & \frac{d}{dx}\big|_{s=0}f(e^{-sX}\sigma)=\tilde{X}f(\sigma)\\
d\pi(t)f(\sigma) & = & \frac{d}{dy}\big|_{t=0}f(e^{-tY}\sigma)=\tilde{Y}f(\sigma).\end{eqnarray*}
Hence in the paramter space $(s,t)\in\mathbb{R}^{2}$ we have two
usual derivative operators $\partial/\partial s$ and $\partial/\partial t$,
where on the manifold we have\begin{eqnarray*}
\frac{\partial}{\partial s} & = & -x\frac{\partial}{\partial x}-y\frac{\partial}{\partial y}\\
\frac{\partial}{\partial t} & = & -y\frac{\partial}{\partial y}\end{eqnarray*}
The usual positive Laplacian on $\mathbb{R}^{2}$ translates to \begin{eqnarray*}
-\triangle & = & (\frac{\partial}{\partial s})^{2}+(\frac{\partial}{\partial t})^{2}\\
 & = & -[(\tilde{X})^{2}+(\tilde{Y})^{2}]\\
 & = & (-x\frac{\partial}{\partial x}-y\frac{\partial}{\partial y})(-x\frac{\partial}{\partial x}-y\frac{\partial}{\partial y})+(-y\frac{\partial}{\partial y})^{2}\\
 & = & -[(x^{2}\frac{\partial^{2}}{\partial x^{2}}+2xy\frac{\partial^{2}}{\partial x\partial y}+y^{2}\frac{\partial^{2}}{\partial y^{2}}+x\frac{\partial}{\partial x}+y\frac{\partial}{\partial y})+y^{2}\frac{\partial^{2}}{\partial y^{2}}]\\
 & = & -(x^{2}\frac{\partial^{2}}{\partial x^{2}}+2xy\frac{\partial^{2}}{\partial x\partial y}+(y^{2}+1)\frac{\partial^{2}}{\partial y^{2}}+x\frac{\partial}{\partial x}+y\frac{\partial}{\partial y}).\end{eqnarray*}
This is in fact an elliptic operator, since the matrix\[
\left[\begin{array}{cc}
x^{2} & xy\\
xy & y^{2}+1\end{array}\right]\]
has trace $trace=x^{2}+y^{2}+1\geq1$, and $\det=x^{2}\geq0$. If
instead we have {}``$y^{2}$'' then the determinant is the constant
zero.

Notice the term {}``$y^{2}+1$'' is essential for $\triangle$ being
elliptic. Also notice that all the coefficients are analytic functions
in the $(x,y)$ variables.
\end{example}
\end{example}
\end{example}
\begin{note}
Notice the $\Gamma$ is unimodular, hence it is just a copy of $\mathbb{R}$.
Its invariant measure is the Lebesgue measure on $\mathbb{R}$.\end{note}
\begin{example}
Heisenberg group $G=\{a,b,c\}$ where\[
(a,b,c)=\left[\begin{array}{ccc}
1 & a & c\\
0 & 1 & b\\
0 & 0 & 1\end{array}\right]\]
The multiplication rule is given by\begin{eqnarray*}
(a,b,c)(a',b',c') & = & (a+a',b+b',c+ab'+c')\\
(a,b,c)^{-1} & = & (-a,-b,-c+ab)\end{eqnarray*}
The subgroup $\Gamma=\{(0,b,c)\}$ where \[
(1,b,c)=\left[\begin{array}{ccc}
1 & 0 & c\\
0 & 1 & b\\
0 & 0 & 1\end{array}\right]\]
is two dimensional, abelian and normal.
\begin{itemize}
\item abelian: $(0,b,c)(0,b',c')=(0,b+b',c+c')$ 
\item normal: \begin{eqnarray*}
(a,b,c)(0,x,y)(a,b,c)^{-1} & = & (a,b,c)(0,x,y)(-a,-b,-c+ab)\\
 & = & (a,b+x,c+y+ax)(-a,-b,-c+ab)\\
 & = & (0,x,y+ax+ab-ab)\\
 & = & (0,x,ax+y)\end{eqnarray*}
Note that this is also $Ad_{g}$ acting on the Lie algebra of $\Gamma$.
\end{itemize}
The additive group $(\mathbb{R},+)$ acts on $\Gamma=\{(0,b,c)\}\simeq(\mathbb{R}^{2},+)$
by\begin{eqnarray*}
\varphi:(\mathbb{R},+) & \rightarrow & Aut(\Gamma)\\
\varphi(a)\left[\begin{array}{c}
c\\
b\end{array}\right] & = & \left[\begin{array}{cc}
1 & a\\
0 & 1\end{array}\right]\left[\begin{array}{c}
c\\
b\end{array}\right]\\
 & = & \left[\begin{array}{c}
c+ab\\
b\end{array}\right]\end{eqnarray*}
check: \begin{eqnarray*}
(a,(b,c))(a',(b',c')) & = & (a+a',(b,c)+\varphi(a)(b',c'))\\
 & = & (a+a',(b,c)+(b',c'+ab'))\\
 & = & (a+a',b+b',c+c'+ab)\\
(a,(b,c))^{-1} & = & (-a,\varphi_{a^{-1}}(-b,-c))\\
 & = & (-a,(-b,-c+ab))\\
 & = & (-a,-b,-c+ab)\\
(a,b,c)(0,b',c')(a,b,c)^{-1} & = & (a,b+b',c+c'+ab')(-a,-b,-c+ab)\\
 & = & (0,b',c'+ab')\\
 & = & \varphi_{a}\left[\begin{array}{c}
c'\\
b'\end{array}\right]\end{eqnarray*}

\end{example}

\subsection{Induced representation}

This also goes under the name of {}``Mackey machine''. Its modern
formulation is in the context of completely positive map.

Let $G$ be a locally compact group, and $\Gamma\subset G$ is closed
subgroup. 

\begin{longtable}{c|c|c}
\hline 
group & right Haar measure & modular function\tabularnewline
\hline 
$G$ & $dg$ & $\triangle$\tabularnewline
\hline 
$\Gamma$ & $d\xi$ & $\delta$\tabularnewline
\hline
\end{longtable}Recall the modular functions come in when the translation was put
on the wrong side, i.e.\[
\int_{G}f(gx)dx=\triangle(g^{-1})\int_{G}f(x)dx\]
or equivalently,\[
\triangle(g)\int_{G}f(gx)dx=\int_{G}f(x)dx\]
similar for $d\xi$ on the subgroup $\Gamma$.

Form the quotient $M=\Gamma\backslash G$. Let $\pi:G\rightarrow\Gamma\backslash G$
be the quotient map or the covering map. $M$ carries a transitive
$G$ action. 
\begin{note}
$M$ is called \emph{fundamental domain} or homogeneous space. $ $$M$
is a group if and only if $\Gamma$ is a normal subgroup in $G$.
In general, $M$ may not be a group, but it is still a very important
manifold.
\begin{note}
$\mu$ is called an invariant measure on $M$, if $\mu(Eg)=\mu(E)$,
$\forall g\in G$. $\mu$ is said to be quasi-invariant, if $\mu(E)=0$
$\Leftrightarrow$ $\mu(Eg)=0$, $\forall g$. In general there is
no invariant measure on $M$, but only quasi-invariant measures. $M$
has an invariant measure if and only if $M$ is unimodular (Heisenberg
group). Not all groups are unimodular, a typical example is the $ax+b$
group.
\end{note}
\end{note}
Define $\tau:C_{c}(G)\rightarrow C_{c}(M)$ by\[
(\tau\varphi)(\pi(x))=\int_{\Gamma}\varphi(\xi x)d\xi.\]

\begin{lem}
$\tau$ is surjective.\end{lem}
\begin{note}
Since $\varphi$ has compact support, the integral is well-defined.
$\tau$ is called conditional expectation. It is simply the summation
of $\varphi$ over the orbit $\Gamma x$. This is because if $\xi$
runs over $\Gamma$, $\xi x$ runs over $\Gamma x$. $\tau\varphi$
may also be interpreted as taking average, only it does not divide
out the total mass, but that only differs by a constant. 
\begin{note}
We may also say $\tau\varphi$ is a $\Gamma$-periodic extention,
by looking at it as a function defined on $G$. Then we check that\[
\tau\varphi(\xi_{1}x)=\int_{\Gamma}\varphi(\xi\xi_{1}x)d\xi=\tau\varphi(x)\]
because $d\xi$ is a right Haar measure. Thus $\tau\varphi$ is $\Gamma$-periodic
in the sense that\[
\tau\varphi(\xi x)=\tau\varphi(x),\:\forall\xi\in\Gamma.\]

\end{note}
\end{note}
\begin{example}
$G=\mathbb{R}$, $\Gamma=\mathbb{Z}$ with $d\xi$ being the counting
measure on $\mathbb{Z}$. \[
(\tau\varphi)(\pi(x))=\int_{\Gamma}\varphi(\xi x)d\xi=\sum_{z\in\mathbb{Z}}\varphi(z+x)\]
As a consequence, $\tau\varphi$ is left translation invariant by
integers, i.e. $\tau\varphi$ is $\mathbb{Z}$-periodic, \begin{eqnarray*}
(\tau\varphi)(\pi(z_{0}+x)) & = & \sum_{z\in\mathbb{Z}}\varphi(z_{0}+z+x)\\
 & = & \sum_{z\in\mathbb{Z}}\varphi(z+x)\end{eqnarray*}
Since $\varphi$ has compact support, $\varphi(z+x)$ vanishes for
all but a finite number of $z$. Hence it is a finite summation, and
it is well-defined.
\end{example}
Let $L$ be a unitary representation of $\Gamma$ on a Hilbert space
$V$. The task now is to get a unitary representaton $U^{ind}$ of
$G$ on an enlarged Hilbert space $H$. i.e. given $L\in Rep(\Gamma,V)$,
get $U^{ind}\in Rep(G,H)$ with $H\supset V$. 

Let $F_{*}$ be the set of function $f:G\rightarrow V$ so that\[
f(\xi g)=\rho(\xi)^{1/2}L_{\xi}(f(g))\]
where $\rho=\delta/\triangle$.
\begin{note}
If $f\in F_{*}$, then $f(\cdot g')\in F_{*}$, as \[
f(\xi gg')=\rho(\xi)^{1/2}L_{\xi}(f(gg')).\]
It follows that $F_{*}$ is invariant under right translation by $ $$g\in G$,
i.e. $(R_{g}f)(\cdot)=f(\cdot g)\in F_{*}$, $\forall f\in F_{*}$.
Eventually, we want to define $(U_{g}^{ind}f)(\cdot):=f(\cdot g)$,
not on $F_{*}$ but pass to a subspace.
\begin{note}
The factor $\rho(\xi)^{1/2}$ comes in, since later we will defined
an inner product $\iprod{\cdot}{\cdot}_{new}$ on $F_{*}$ so that
$\norm{f(\xi g)}_{new}=\norm{f(g)}_{new}$. Let's ignore $\rho(\xi)^{1/2}$
for a moment. 

$L_{\xi}$ is unitary implies that $\norm{f(\xi g)}_{V}=\norm{L_{\xi}f(g)}_{V}=\norm{f(g)}_{V}$.
Notice that Hilbert spaces exist up to unitary equivalence, $L_{\xi}f(g)$
and $f(g)$ really are the same function. As $\xi$ running throught
$\Gamma$, $\norm{f(\xi g)}$ is a constant on the orbit $\Gamma g$.
It follows that $f(\xi g)$ is in fact a $V$-valued function defined
on the quotient $M=\Gamma\backslash G$. 

We will later use these functions as multiplication operators.
\end{note}
\end{note}
\begin{example}
The Heisenberg group is unimodular, $\rho=1$. 
\begin{example}
For the $ax+b$ group, \begin{eqnarray*}
d\lambda_{R} & = & \frac{dadb}{a}\\
d\lambda_{L} & = & \frac{dadb}{a^{2}}\\
\triangle & = & \frac{d\lambda_{L}}{d\lambda_{R}}=\frac{1}{a}\end{eqnarray*}
On the abelian normal subgroup $\Gamma=\{(1,b)\}$, where $a=1$,
$\triangle(\xi)=1$. $\Gamma$ is unimodular, hence $\delta(\xi)=1$.
Therefore, $\rho(\xi)=\delta(\xi)/\triangle(\xi)=1$, $\forall\xi\in\Gamma$.
\end{example}
\end{example}
For all $f\in F_{*}$, the map $\mu_{f,f}:C_{c}(M)\rightarrow\mathbb{C}$
given by \[
\mu_{f,f}:\tau\varphi\mapsto\int_{G}\norm{f(g)}_{V}^{2}\varphi(g)dg\]
is a positive linear functional. By Riesz, there exists a unique Radon
measure $\mu_{f,f}$ on $M$, such that \[
\int_{G}\norm{f(g)}_{V}^{2}\varphi(g)dg=\int_{M}(\tau\varphi)d\mu_{f,f}.\]

\begin{note}
Recall that given a measure space $(X,\mathfrak{M},\mu)$, let $f:X\rightarrow Y$.
Define a linear functional $\Lambda:C_{c}(Y)\rightarrow\mathbb{C}$
by \[
\Lambda\varphi:=\int\varphi(f(x))d\mu(x)\]
$\Lambda$ is positive, hence by Riesz's theorem, there exists a unique
regular Borel measure $\mu_{f}$ on $Y$ so that\[
\Lambda\varphi=\int_{Y}\varphi d\mu_{f}=\int_{X}\varphi(f(x))d\mu(x).\]
It follows that $\mu_{f}=\mu\circ f^{-1}$. 
\begin{note}
Under current setting, we have a covering map $\pi:G\rightarrow\Gamma\backslash G=:M$,
and the right Haar measure $\mu$ on $G$. Thus we may define a measure
$\mu\circ\pi^{-1}$. However, given $\varphi\in C_{c}(M)$, $\varphi(\pi(x))$
may not have compact support, or equivalently, $\pi^{-1}(E)$ is $\Gamma$
periodic. For example, take $G=\mathbb{R}$, $\Gamma=\mathbb{Z}$,
$M=\mathbb{Z}\backslash\mathbb{R}$. Then $\pi^{-1}([0,1/2))$ is
$\mathbb{Z}$-periodic, which has infinite Lebesgue measure. What
we really need is some map so that the inverse of a subset of $M$
is restricted to a single $\Gamma$ period. This is essentially what
$\tau$ does. Taking $\tau\varphi\in C_{c}(\Gamma)$, get the inverse
image $\varphi\in C_{c}(G)$. Even if $\varphi$ is not restricted
in a single $\Gamma$ period, $\varphi$ always has compact support. 
\end{note}
\end{note}
Hence we get a family of measures indexed by elements in $F_{*}$.
If choosing $f,g\in F_{*}$ then we get complex measures $\mu_{f,g}$.
(using polarization identity)
\begin{itemize}
\item Define $\norm{f}^{2}:=\mu_{f,f}(M)$, $\iprod{f}{g}:=\mu_{f,g}(M)$
\item Complete $F_{*}$ with respect to this norm to get an enlarged Hilbert
space $H$.
\item Define induced representation $U^{ind}$ on $H$ as\[
U_{g}^{ind}f(x)=f(xg)\]
$U_{g}^{ind}$ is unitary. \[
\norm{U^{ind}f}_{H}=???\]
\end{itemize}
\begin{note}
$\mu_{f,g}(M)=\int_{M}\tau\varphi d\xi$, where $\tau\varphi\equiv1$.
What is $\varphi$ then? It turns out that $\varphi$ could be constant
1 over a single $\Gamma$-period, or $\varphi$ could spread out to
a finite number of $\Gamma$-periods. Draw a picture here!! Therefore,
in this case \begin{eqnarray*}
\norm{f}^{2} & = & \int_{G}\norm{f(g)}_{V}^{2}\varphi(g)dg\\
 & = & \int_{1-period}\norm{f(g)}_{V}^{2}\varphi(g)dg\\
 & = & \int_{1-period}\norm{f(g)}_{V}^{2}dg\\
 & = & \int_{M}\norm{f(g)}_{V}^{2}dg\end{eqnarray*}

\end{note}
Define $P(\psi)f(x):=\psi(\pi(x))f(x)$, for $\psi\in C_{c}(M)$,
$f\in H$, $x\in G$. $P(\psi)$ is the abelian algebra of multiplication
operators. Observe that\[
U_{g}^{ind}P(\psi)U_{g^{-1}}^{ind}=P(\psi(\cdot g))\]
check: \begin{eqnarray*}
U_{g}^{ind}P(\psi)f(x) & = & U_{g}^{ind}\psi(\pi(x))f(x)\\
 & = & \psi(\pi(xg))f(xg)\\
P(\psi(\cdot g))U_{g}^{ind}f(x) & = & P(\psi(\cdot g))f(xg)\\
 & = & \psi(\pi(xg))f(xg)\end{eqnarray*}

Conversely, how to recognize induced representation?
\begin{thm}
(imprimitivity) Let $G$ be a locally compact group with a closed
subgroup $\Gamma$. Let $M=\Gamma\backslash G$. Suppose the system
$(U,P)$ satisfies the covariance relation,\[
U_{g}P(\psi)U_{g^{-1}}=P(\psi(\cdot g)).\]
Then, there exists a unitary representation $L\in Rep(\Gamma,V)$
such that $U\simeq ind_{\Gamma}^{G}(L)$.
\end{thm}

\section{Example - Heisenberg group}

Heisenberg group $G=\{(a,b,c)\}$ where\[
(a,b,c)=\left[\begin{array}{ccc}
1 & a & c\\
0 & 1 & b\\
0 & 0 & 1\end{array}\right]\]
The multiplication rule is given by\begin{eqnarray*}
(a,b,c)(a',b',c') & = & (a+a',b+b',c+c'+ab')\\
(a,b,c)^{-1} & = & (-a,-b,-c+ab)\end{eqnarray*}
The subgroup $\Gamma=\{(0,b,c)\}$ where \[
(1,b,c)=\left[\begin{array}{ccc}
1 & 0 & c\\
0 & 1 & b\\
0 & 0 & 1\end{array}\right]\]
is two dimensional, abelian and normal.
\begin{itemize}
\item abelian: $(0,b,c)(0,b',c')=(0,b+b',c+c')$ 
\item normal: \begin{eqnarray*}
(a,b,c)(0,x,y)(a,b,c)^{-1} & = & (a,b,c)(0,x,y)(-a,-b,-c+ab)\\
 & = & (a,b+x,c+y+ax)(-a,-b,-c+ab)\\
 & = & (0,x,y+ax+ab-ab)\\
 & = & (0,x,ax+y)\end{eqnarray*}
i.e. $Ad:G\rightarrow GL(\mathfrak{n})$, as \begin{eqnarray*}
Ad(g)(n): & = & gng^{-1}\\
(x,y) & \mapsto & (ax+y)\end{eqnarray*}
the orbit is a 2-d transformation.
\end{itemize}
Fix $h\in\mathbb{R}\backslash\{0\}$. Recall the Schrodinger representation
of $G$ on $L^{2}(\mathbb{R})$ \[
U_{g}f(x)=e^{ih(c+bx)}f(x+a)\]
We show that the Schrodinger representation is induced from a unitary
representation $L$ on the subgroup $\Gamma$. 
\begin{enumerate}
\item Let $L\in Rep(\Gamma,V)$ where $\Gamma=\{(0,b,c)\}$, $V=\mathbb{C}$,
\[
L_{\xi(b,c)}=e^{ihc}.\]
The complex exponential comes in since we want a unitary representation.
The subgroup $\{(0,0,c)\}$ is the center of $G$. What is the induced
representation? Is it unitarily equivalent to the Schrodinger representation?
\item Look for the family $F_{*}$ of functions $f:G\rightarrow\mathbb{C}$
($V$ is the 1-d Hilbert space $\mathbb{C}$), such that\[
f(\xi(b,c)g)=L_{\xi}f(g).\]
Since \begin{eqnarray*}
f(\xi(b,c)g) & = & f((0,b,c)(x,y,z))=f(x,b+y,c+z)\\
L_{\xi(b,c)}f(g) & = & e^{ihc}f(x,y,z)\end{eqnarray*}
$f(x,y,z)$ satisfies\[
f(x,b+y,c+z)=e^{ihc}f(x,y,z).\]
i.e. we may translate the $y,z$ variables by arbitrary amount, and
the only price to pay is multiplicative factor $e^{ihc}$. Therefore
$f$ is really a function defined on the quotient \[
M=\Gamma\backslash G\simeq\mathbb{R}.\]
$M=\{(x,0,0)\}$ is identified with $\mathbb{R}$, and the invariant
measure on the homogeneous space $M$ is simply the Lebesgue measure.
It is almost clear at this point why the induced representation is
unitarily equivalent to the Schrodinger representation on $L^{2}(\mathbb{R})$. 
\item $\tau\varphi\mapsto\int_{G}\norm{f(g)}_{V}^{2}\varphi(g)dg$ induces
a measure $\mu_{f,f}$ on $M$. This can be seen as follows.\begin{eqnarray*}
\int_{G}\norm{f(g)}_{V}^{2}\varphi(g)dg & = & \int_{G\simeq\mathbb{R}^{3}}\abs{f(x,y,z)}^{2}\varphi(x,y,z)dxdydz\\
 & = & \int_{M\simeq\mathbb{R}}\left(\int_{\Gamma\simeq\mathbb{R}^{2}}\abs{f(x,y,z)}^{2}\varphi(x,y,z)dydz\right)dx\\
 & = & \int_{\mathbb{R}}\abs{f(x,y,z)}^{2}\left(\int_{\mathbb{R}^{2}}\varphi(x,y,z)dydz\right)dx\\
 & = & \int_{\mathbb{R}}\abs{f(x,y,z)}^{2}(\tau\varphi)(\pi(g))dx\\
 & = & \int_{\mathbb{R}}\abs{f(x,y,z)}^{2}(\tau\varphi)(x)dx\\
 & = & \int_{\mathbb{R}}\abs{f(x,0,0)}^{2}(\tau\varphi)(x)dx\end{eqnarray*}
where\begin{eqnarray*}
(\tau\varphi)(\pi(g)) & = & \int_{\Gamma}\varphi(\xi g)d\xi\\
 & = & \int_{\mathbb{R}^{2}}\varphi((0,b,c)(x,y,z))dbdc\\
 & = & \int_{\mathbb{R}^{2}}\varphi(x,b+y,c+z)dbdc\\
 & = & \int_{\mathbb{R}^{2}}\varphi(x,b,c)dbdc\\
 & = & \int_{\mathbb{R}^{2}}\varphi(x,y,z)dydz\\
 & = & (\tau\varphi)(x).\end{eqnarray*}
Hence $\Lambda:C_{c}(M)\rightarrow\mathbb{C}$ given by\[
\Lambda:\tau\varphi\mapsto\int_{G}\norm{f(g)}_{V}^{2}\varphi(g)dg\]
is a positive linear functional, therefore\[
\Lambda=\mu_{f,f}\]
i.e.\[
\int_{\mathbb{R}^{3}}\abs{f(x,y,z)}^{2}\varphi(x,y,z)dxdydz=\int_{\mathbb{R}}(\tau\varphi)(x)d\mu_{f,f}(x).\]

\item Define\[
\begin{alignedat}{1}\norm{f}_{ind}^{2} & :=\mu_{f,f}(M)=\int_{M}\abs{f}^{2}d\xi=\int_{\mathbb{R}}\abs{f(x)}^{2}dx=\int_{\mathbb{R}}\abs{f(x,0,0)}^{2}dx\\
U_{g}^{ind}f(g') & :=f(g'g)\end{alignedat}
\]
By definition, if $g=g(a,b,c)$, $g'=g'(x,y,z)$ then\begin{eqnarray*}
U_{g}^{ind}f(g') & = & f(g'g)\\
 & = & f((x,y,z)(a,b,c))\\
 & = & f(x+a,y+b,z+c+xb)\end{eqnarray*}
and $U^{ind}$ is a unitary representation by the definition of $\norm{f}_{ind}$.
\item To see $U^{ind}$ is unitarily equivalent to the Schrodinger representation
on $L^{2}(\mathbb{R})$, define\begin{eqnarray*}
W:H^{ind} & \rightarrow & L^{2}(\mathbb{R})\\
(Wf)(x) & = & f(x,0,0)\end{eqnarray*}
If put other numbers into $f$, as $f(x,y,z)$, the result is the
same, since $f\in H^{ind}$ is really defined on the quotient $M=\Gamma\backslash G\simeq\mathbb{R}$.
\\
\\
$W$ is unitary: \[
\norm{Wf}_{L^{2}}^{2}=\int_{\mathbb{R}}\abs{Wf}^{2}dx=\int_{\mathbb{R}}\abs{f(x,0,0)}^{2}dx=\int_{\Gamma\backslash G}\abs{f}^{2}d\xi=\norm{f}_{ind}^{2}\]
Intertwining: let $U_{g}$ be the Schrodinger representation. \begin{eqnarray*}
U_{g}(Wf) & = & e^{ih(c+bx)}f(x+a,0,0)\\
WU_{g}^{ind}f & = & W\left(f((x,y,z)(a,b,c))\right)\\
 & = & W\left(f(x+a,y+b,z+c+xb)\right)\\
 & = & W\left(e^{ih(c+bx)}f(x+a,y,z)\right)\\
 & = & e^{ih(c+bx)}f(x+a,0,0)\end{eqnarray*}

\item Since $\{U,L\}'\subset\{L\}'$, then the system $\{U,L\}$ is reducible
implies $L$ is reducible. Equivalent, $\{L\}$ is irreducible implies
$\{U,L\}$ is irreducible. Consequently, $U_{g}$ is irrducible. Since
$U_{g}$ is the induced representation, if it is reducible, $L$ would
also be reducible, but $L$ is 1-dimensional.\end{enumerate}
\begin{note}
The Heisenberg group is a non abelian unimodular Lie group, so the
Haar measure on $G$ is just the product measure $dxdydz$ on $\mathbb{R}^{3}$.
Conditional expectation becomes integrating out the variables correspond
to subgroup. For example, given $f(x,y,z)$ conditioning with respect
to the subgroup $(0,b,c)$ amounts to integrating out the $y,z$ variable
and get a function $\tilde{f}(x)$, where\[
\tilde{f}(x)=\iint f(x,y,z)dydz.\]

\end{note}

\subsection{ax+b group}

$a\in\mathbb{R}_{+}$, $b\in\mathbb{R}$, $g=(a,b)=\left[\begin{array}{cc}
a & b\\
0 & 1\end{array}\right]$. \[
U_{g}f(x)=e^{iax}f(x+b)\]
could also write $a=e^{t}$, then \[
U_{g}f(x)=e^{ie^{t}x}f(x+b)\]
\[
U_{g(a,b)}f(x)=e^{iae^{x}f(x+b)}\]
\begin{eqnarray*}
[\frac{d}{dx},ie^{x}] & = & ie^{x}\\
{}[A,B] & = & B\end{eqnarray*}
or\[
U_{(e^{t},b)}f=e^{ite^{x}}f(x+b)\]
\[
\left[\begin{array}{cc}
0 & b\\
0 & 1\end{array}\right]\]
1-d representation. $L=e^{ib}$. Induce $ind_{L}^{G}\simeq\mbox{Schrodinger}$.

\subsection{$ax+b$ gruop}

\section{Coadjoint orbits}

It turns out that only a small family of representations are induced.
The question is how to detect whether a representation is induced.
The whole theory is also under the name of {}``Mackey machine''.
The notion of {}``machine'' refers to something that one can actually
compute in practice. Two main examples are the Heisenberg group and
the $ax+b$ group.

What is the mysteries parameter $h$ that comes into the Schrodinger
representation? It is a physical constant, but how to explain it in
mathematical theory?

\subsection{review of some Lie theory}
\begin{thm}
Every Lie group is diffeomorphic to a matrix group.
\end{thm}
The exponential function maps a neighborhood of $0$ into a connected
component of $G$ containing the identity element. For example, the
Lie algebra of the Heisenberg group is\[
\left[\begin{array}{ccc}
0 & * & *\\
0 & 0 & *\\
0 & 0 & 0\end{array}\right]\]
All the Lie groups the we will ever encounter come from a quadratic
form. Given a quadratic form \[
\varphi:V\times V\rightarrow\mathbb{C}\]
there is an associated group that fixes $\varphi$, i.e. we consider
elements $g$ such that \[
\varphi(gx,gy)=\varphi(x,y)\]
and define $G(\varphi)$ as the collection of these elements. $G(\varphi)$
is clearly a group. Apply the exponential map and the product rule,\[
\frac{d}{dt}\big|_{t=0}\varphi(e^{tX}x,e^{tX}y)=0\Longleftrightarrow\varphi(Xx,y)+\varphi(x,Xy)=0\]
hence \[
X+X^{tr}=0\]
The determinant and trace are related so that\[
\det(e^{tX})=e^{t\cdot trace(X)}\]
thus $\det=1$ if and only if $trace=0$. It is often stated in differential
geometry that the derivative of the determinant is equal to the trace.
\begin{example}
$\mathbb{R}^{n}$, $\varphi(x,y)=\sum x_{i}y_{i}$. The associated
group is the orthogonal group $O_{n}$.
\end{example}
There is a famous cute little trick to make $O_{n-1}$ into a subgroup
of $O_{n}$. $O_{n-1}$ is not normal in $O_{n}$. We may split the
quadratic form into\[
\sum_{i=1}^{n}x_{i}^{2}+1\]
where $1$ corresponds to the last coordinate in $O_{n}$. Then we
may identity $O_{n-1}$ as a subgroup of $O_{n}$ \[
g\mapsto\left[\begin{array}{cc}
g & 0\\
0 & I\end{array}\right]\]
where $I$ is the identity operator. 

Claim: $O_{n}/O_{n-1}\simeq S^{n-1}$. How to see this? Let $u$ be
the unit vector corresponding to the last dimension, look for $g$
that fixes $u$ i.e. $gu=u$. Such $g$ forms a subgroup of $O_{n}$,
and it is called isotropy group. \[
I_{n}=\{g:gu=u\}\simeq O_{n-1}\]
For any $g\in O_{n}$, $gu=?$. Notice that for all $v\in S^{n-1}$,
there exists $g\in O_{n}$ such that $gu=v$. Hence \[
g\mapsto gu\]
in onto $S^{n-1}$. The kernel of this map is $I_{n}\simeq O_{n-1}$,
thus \[
O_{n}/O_{n-1}\simeq S_{n}\]
Such spaces are called homogeneous spaces.
\begin{example}
visualize this with $O_{3}$ and $O_{2}$.
\end{example}
Other examples of homogeneous spaces show up in number theory all
the time. For example, the Poincare group $G/\mbox{discrete subgroup}$.

$G$, $N\subset G$ nornal subgroup. The map $g\cdot g^{-1}:G\rightarrow G$
is an automorphism sending identity to identity, hence if we differentiate
it, we get a transformation in $GL(\mathfrak{g})$. i.e. we get a
family of maps $Ad_{g}\in GL(\mathfrak{g})$ indexed by elements in
$G$. $g\mapsto Ad_{g}\in GL(\mathfrak{g})$ is a representation of
$G$, hence if it is differentiated, we get a representation of $\mathfrak{g}$,
$ad_{g}:\mathfrak{g}\mapsto End(\mathfrak{g})$ acting on the vector
space $\mathfrak{g}$.

$gng^{-1}\in N$. $\forall g$, $g\cdot g^{-1}$ is a transformation
from $N$ to $N$, define $Ad_{g}(n)=gng^{-1}$. Differentiate to
get $ad:\mathfrak{n}\rightarrow\mathfrak{n}$. $\mathfrak{n}$ is
a vector space, has a dual. Linear transformation on vector space
passes to the dual space. \begin{eqnarray*}
\varphi^{*}(v^{*})(u) & = & v^{*}(\varphi(u))\\
 & \Updownarrow\\
\iprod{\Lambda^{*}v^{*}}{u} & = & \iprod{v^{*}}{\Lambda u}.\end{eqnarray*}
In order to get the transformation rules work out, have to pass to
the adjoint or the dual space.\[
Ad_{g}^{*}:\mathfrak{n}^{*}\rightarrow\mathfrak{n}^{*}\]
the coadjoint representation of $\mathfrak{n}$.

Orbits of co-adjoint representation acounts precisely to equivalence
classes of irrducible representations.
\begin{example}
Heisenberg group $G=\{(a,b,c)\}$ with \[
(a,b,c)=\left[\begin{array}{ccc}
1 & a & c\\
0 & 1 & b\\
0 & 0 & 1\end{array}\right]\]
normal subgroup $N=\{(0,b,c)\}$\[
(0,b,c)=\left[\begin{array}{ccc}
1 & 0 & c\\
0 & 1 & b\\
0 & 0 & 1\end{array}\right]\]
with Lie algebra $\mathfrak{n}=\{(b,c)\}$\[
(0,\xi,\eta)=\left[\begin{array}{ccc}
1 & 0 & c\\
0 & 1 & b\ \\
0 & 0 & 1\end{array}\right]\]
$Ad_{g}:\mathfrak{n}\rightarrow\mathfrak{n}$ given by\begin{eqnarray*}
gng^{-1} & = & (a,b,c)(0,y,x)(-a,-b,-c+ab)\\
 & = & (a,b+y,c+x+ay)(-a,-b,-c+ab)\\
 & = & (0,y,x+ay)\end{eqnarray*}
hence $Ad_{g}:\mathbb{R}^{2}\rightarrow\mathbb{R}^{2}$\[
Ad_{g}:\left[\begin{array}{c}
x\\
y\end{array}\right]\mapsto\left[\begin{array}{c}
x+ay\\
y\end{array}\right].\]
The matrix of $Ad_{g}$ is (before taking adjoint) is\[
Ad_{g}=\left[\begin{array}{cc}
1 & a\\
0 & 1\end{array}\right].\]
The matrix for $Ad_{g}^{*}$ is\[
Ad_{g}^{*}=\left[\begin{array}{cc}
1 & 0\\
a & 1\end{array}\right].\]
We use $[\xi,\eta]^{T}$ for the dual $\mathfrak{n}^{*}$; and use
$[x,y]^{T}$ for $\mathfrak{n}$. Then\[
Ad_{g}^{*}:\left[\begin{array}{c}
\xi\\
\eta\end{array}\right]\mapsto\left[\begin{array}{c}
\xi\\
a\xi+\eta\end{array}\right]\]
What about the orbit? In the example of $O_{n}/O_{n-1}$, the orbit
is $S^{n-1}$. 

For $\xi\in\mathbb{R}\backslash\{0\}$, the orbit of $Ad_{g}^{*}$
is \[
\left[\begin{array}{c}
\xi\\
0\end{array}\right]\mapsto\left[\begin{array}{c}
\xi\\
\mathbb{R}\end{array}\right]\]
i.e. vertical lines with $x$-coordinate $\xi$. $\xi=0$ amounts
to fixed point, i.e. the orbit is a fixed point.

The simplest orbit is when the orbit is a fixed point. i.e.\[
Ad_{g}^{*}:\left[\begin{array}{c}
\xi\\
\eta\end{array}\right]\mapsto\left[\begin{array}{c}
\xi\\
\eta\end{array}\right]\in V^{*}\]
where if we choose \[
\left[\begin{array}{c}
\xi\\
\eta\end{array}\right]=\left[\begin{array}{c}
0\\
1\end{array}\right]\]
it is a fixed point. 

The other extreme is to take any $\xi\neq0$, then \[
Ad_{g}^{*}:\left[\begin{array}{c}
\xi\\
0\end{array}\right]\mapsto\left[\begin{array}{c}
\xi\\
\mathbb{R}\end{array}\right]\]
i.e. get vertical lines indexed by the $x$-coordinate $\xi$. In
this example, a cross section is a subset of $\mathbb{R}^{2}$ that
intersects each orbit at precisely one point. Every cross section
in this example is a Borel set in $\mathbb{R}^{2}$.

We don't always get measurable cross sections. An example is the construction
of non-measurable set as was given in Rudin's book. Cross section
is a Borel set that intersects each coset at precisely one point.

Why does it give all the equivalent classes of irreducible representations?
Since we have a unitary representation $L_{n}\in Rep(N,V)$, $L_{n}:V\rightarrow V$
and by construction of the induced representation $U_{g}\in Rep(G,H)$,
$N\subset G$ normal such that\[
U_{g}L_{n}U_{g^{-1}}=L_{gng^{-1}}\]
i.e. \[
L_{g}\simeq L_{gng^{-1}}\]
now pass to the Lie algebra and its dual\[
L_{n}\rightarrow LA\rightarrow LA^{*}.\]

\end{example}

\section{Gaarding space}

We talked about how to detect whether a representation is induced.
Given a group $G$ with a subgroup $\Gamma$ let $M:=\Gamma\backslash G$.
The map $\pi:G\rightarrow M$ is called a covering map, which sends
$g$ to its equivalent class or the coset $\Gamma g$. $M$ is given
its projective topology, so $\pi$ is coninuous. When $G$ is compact,
many things simplify. For example, if $G$ is compact, any irreducible
representation is fnite dimensional. But many groups are not compact,
only locally compact. For exmaple, ax+b, $H_{3}$, $SL_{n}$.

Specialize to Lie groups. $G$ and subgroup $H$ have Lie algebras
$\mathfrak{g}$ and $\mathfrak{h}$ respectively.\[
\mathfrak{g}=\{X:e^{tx}\in G,\forall t\in\mathbb{R}\}\]
Almost all Lie algebras we will encounter come from specifying a quadratic
form $\varphi:G\times G\rightarrow\mathbb{C}$. $\varphi$ is then
uniquely determined by a Hermitian matrix $A$ so that\[
\varphi(x,y)=x^{tr}\cdot Ay\]
Let $G=G(\varphi)=\{g:\varphi(gx,gy)=\varphi(x,y)\}$, then \[
\frac{d}{dt}\big|_{t=1}\varphi(e^{tX}x,e^{tX}y)=0\]
and with an application of the product rule,\begin{eqnarray*}
\varphi(Xx,y)+\varphi(x,Xy) & = & 0\\
(Xx)^{tr}\cdot Ay+x^{tr}\cdot AXy & = & 0\end{eqnarray*}
 \[
X^{tr}A+AX=0\]
hence \[
\mathfrak{g}=\{X:X^{tr}A+AX=0\}.\]

Let $U\in Rep(G,H)$, for $X\in\mathfrak{g}$, $U(e^{tX})$ is a one
parameter continuous group of unitary operator, hence by Stone's theorem,
\[
U(e^{tX})=e^{itH_{X}}\]
for some self-adjoint operator $H_{X}$ (possibly unbounded). We often
write \[
dU(X):=iH_{X}\]
to indicate that $dU(X)$ is the directional derivative along the
direction $X$. Notice that $H_{X}^{*}=H_{X}$ but \[
(iH_{X})^{*}=-(iH_{X})\]
i.e. $dU(X)$ is skew adjoint.
\begin{example}
$G=\{(a,b,c)\}$ Heisenberg group. $\mathfrak{g}=\{X_{1}\sim a,X_{2}\sim b,X_{3}\sim c\}$.
Take the Schrodinger representation $U_{g}f(x)=e^{ih(c+bx)}f(x+a)$,
$f\in L^{2}(\mathbb{R})$. 
\begin{itemize}
\item $U(e^{tX_{1}})f(x)=f(x+t)$\begin{eqnarray*}
\frac{d}{dt}\big|_{t=0}U(e^{tX_{1}})f(x) & = & \frac{d}{dx}f(x)\\
dU(X_{1}) & = & \frac{d}{dx}\end{eqnarray*}

\item $U(e^{tX_{2}})f(x)=e^{ih(tx)}f(x)$\begin{eqnarray*}
\frac{d}{dt}\big|_{t=0}U(e^{tX_{2}})f(x) & = & ihxf(x)\\
dU(X_{2}) & = & ihx\end{eqnarray*}

\item $U(e^{tX_{3}})f(x)=e^{iht}f(x)$\begin{eqnarray*}
\frac{d}{dt}\big|_{t=0}U(e^{tX_{3}})f(x) & = & ihf(x)\\
dU(X_{2}) & = & ihI\end{eqnarray*}
Notice that $dU(X_{i})$ are all skew adjoint.\begin{eqnarray*}
[dU(X_{1}),dU(X_{2})] & = & [\frac{d}{dx},ihx]\\
 & = & ih[\frac{d}{dx},x]\\
 & = & ih\end{eqnarray*}
In case we want self-adjoint operators, replace $dU(X_{i})$ by$-idU(X_{i})$
and get\begin{eqnarray*}
-idU(X_{1}) & = & \frac{1}{i}\frac{d}{dx}\\
-idU(X_{2}) & = & hx\\
-idU(X_{1}) & = & hI\end{eqnarray*}
\[
[\frac{1}{i}\frac{d}{dx},hx]=\frac{h}{i}.\]

\end{itemize}
\end{example}
What is the space of functions that $U_{g}$ acts on? L. Gaarding
/gor-ding/ (Sweedish mathematician) looked for one space that always
works. It's now called the Gaarding space.

Start with $C_{c}(G)$, every $\varphi\in C_{c}(G)$ can be approximated
by the so called Gaarding functions, using the convolution argument.
Define convolution as\begin{eqnarray*}
\varphi\star\psi(g) & = & \int\varphi(gh)\psi(h)d_{R}h\\
\varphi\star\psi(g) & = & \int\varphi(h)\psi(g^{-1}h)d_{L}h\end{eqnarray*}
Take an approximation of identity $\zeta_{j}$, so that\[
\varphi\star\zeta_{j}\rightarrow\varphi,\; j\rightarrow0.\]
Define Gaarding space as functions given by\[
U(\varphi)v=\int\varphi(h)U(h)vd_{L}h\]
where $\varphi\in C_{c}(G)$, $v\in H$, or we say\[
U(\varphi):=\int\varphi(h)U(h)vd_{L}h.\]
Since $\varphi$ vanishes outside a compact set, and since $U(h)v$
is continuous and bounded in $\norm{\cdot}$, it follows that $U(\varphi)$
is well-defined.
\begin{lem}
$U(\varphi_{1}\star\varphi_{2})=U(\varphi_{1})U(\varphi_{2})$ ($U$
is a representation of the group algebra)\end{lem}
\begin{proof}
Use Fubini,\begin{eqnarray*}
\int\varphi_{1}\star\varphi_{2}(g)U(g)dg & = & \iint\varphi_{1}(h)\varphi(h^{-1}g)U(g)dhdg\\
 & = & \iint\varphi_{1}(h)\varphi(g)U(hg)dhdg\:(dg\mbox{ is r-Haar}g\mapsto hg)\\
 & = & \iint\varphi_{1}(h)\varphi(g)U(h)U(g)dhdg\\
 & = & \int\varphi_{1}(h)U(h)dh\int\varphi_{2}(g)U(g)dg\end{eqnarray*}
Choose $\varphi$ to be an approximation of identity, then \[
\int\varphi(g)U(g)vdg\rightarrow U(e)v=v\]
i.e. any vector $v\in H$ can be approximated by functions in the
Gaarding space. It follows that\[
\{U(\varphi)v\}\]
is dense in $H$.\end{proof}
\begin{lem}
$U(\varphi)$ can be differented, in the sense that\[
dU(X)U(\varphi)v=U(\tilde{X}\varphi)v\]
where we use $\tilde{X}$ to denote the vector field.\end{lem}
\begin{proof}
need to prove \[
\lim_{t\rightarrow0}\frac{1}{t}\left[(U(e^{tX})-I)U(\varphi)v\right]=U(\tilde{X}\varphi)v.\]
Let $v_{\varphi}:=U(\varphi)v$, need to look at in general $U(g)v_{\varphi}$.\begin{eqnarray*}
U(g)v_{\varphi} & = & U(g)\int\varphi(h)U(h)vdh\\
 & = & \int\varphi(h)U(gh)dh\\
 & = & \int\triangle(g)\varphi(g^{-1}h)U(h)dh\end{eqnarray*}
set $g=e^{tX}$.\end{proof}
\begin{note}
If assuming unimodular, $\triangle$ does not show up. Otherwise,
$\triangle$ is some correction term which is also differnetiable.
$\tilde{X}$ acts on $\varphi$ as $\tilde{X}\varphi$. $\tilde{X}$
is called the derivative of the translation operator $e^{tX}$.
\begin{note}
Schwartz space is the Gaarding space for the Schrodinger representation.
\end{note}
\end{note}

\section{Decomposition of representation}

We study some examples of duality. 
\begin{itemize}
\item $G=T$, $\hat{G}=\mathbb{Z}$\begin{eqnarray*}
\chi_{n}(z) & = & z^{n}\\
\chi_{n}(zw) & = & z^{n}w^{n}=\chi_{n}(z)\chi_{n}(w)\end{eqnarray*}

\item $G=\mathbb{R}$, $\hat{G}=\mathbb{R}$ \begin{eqnarray*}
\chi_{t}(x) & = & e^{itx}\end{eqnarray*}

\item $G=\mathbb{Z}/n\mathbb{Z}\simeq\{0,1,\cdots,n-1\}$. $\hat{G}=G$.
\\
This is another example where $\hat{G}=G$. \\
Let $\zeta=e^{i2\pi/n}$ be the primitive $n^{th}$-root of unity.
$k\in\mathbb{Z}_{n}$, $l=\{0,1,\ldots,n-1\}$\[
\chi_{l}(k)=e^{i\frac{2\pi kl}{n}}\]

\end{itemize}
If $G$ is a locally compact abelian group, $\hat{G}$ is the set
of 1-dimensional representations. \[
\hat{G}=\{\chi:g\mapsto\chi(g)\in T,\chi(gh)=\chi(g)\chi(h)\}.\]
$\hat{G}$ is also a group, with group operation defined by $(\chi_{1}\chi_{2})(g):=\chi_{1}(g)\chi_{2}(g)$.
$\hat{G}$ is called the group characters.
\begin{thm}
(Pontryagin) If $G$ is a locally compact abelian group, then $G=\hat{\hat{G}}$.\end{thm}
\begin{note}
This result first appears in 1930s in the annals of math, when John
Von Neumann was the editor of the journal at the time. The original
paper was hand written. Von Neumann rewrote it, since then the theorem
became very popular.
\end{note}
There are many groups that are not locally compact abelian. We want
to study the duality question in general. Examples are 
\begin{itemize}
\item compact group
\item fintie group (abelian, or not)
\item $H_{3}$ locally compact, nonabelian, unimodular
\item ax+b locally compact, nonabelian, non-unimodular
\end{itemize}
If $G$ is not abelian, $\hat{G}$ is not a group. We would like to
decompose $\hat{G}$ into irreducible representations. The big names
under this development are Krein (died 7 years ago), Peter-Weyl, Weil,
Segal. 

Let $G$ be a group (may not be abelian). The right regular representation
is defined as\[
R_{g}f(\cdot)=f(\cdot g)\]
$R_{g}$ is a unitary operator acting on $L^{2}(\mu_{R})$, where
$\mu_{R}$ is the right invariant Haar measure.
\begin{thm}
(Krein, Weil, Segal) Let $G$ be locally compact unimodular (abelian
or not). Then the right regular representation decomposes into a direct
integral of irreducible representations \[
R_{g}=\int_{\hat{G}}^{\oplus}irrep\: d\mu\]

\end{thm}
where $\mu$ is called the Plancherel measure. 
\begin{example}
$G=T$, $\hat{G}=\mathbb{Z}$. Irreducible representations $\{e^{in(\cdot)}\}_{n}\sim\mathbb{Z}$
\begin{eqnarray*}
(U_{y}f)(x) & = & f(x+y)\\
 & = & \sum_{n}\hat{f}(n)\chi_{n}(x+y)\\
 & = & \sum_{n}\hat{f}(n)e^{i2\pi n(x+y)}\end{eqnarray*}
\[
(U_{y}f)(0)=f(y)=\sum_{n}\hat{f}(n)e^{i2\pi ny}\]
The Plancherel measure in this case is the counting measure.
\begin{example}
$G=\mathbb{R}$, $\hat{G}=\mathbb{R}$. Irreducible representations
$\{e^{it(\cdot)}\}_{t\in\mathbb{R}}\sim\mathbb{R}$.\begin{eqnarray*}
(U_{y}f)(x) & = & f(x+y)\\
 & = & \int_{\mathbb{R}}\hat{f}(t)\chi_{t}(x+y)dt\\
 & = & \int_{\mathbb{R}}\hat{f}(t)e^{it(x+y)}dt\end{eqnarray*}
\[
(U_{y}f)(0)=f(y)=\int_{\mathbb{R}}\hat{f}(t)e^{ity}dt\]
where the Plancherel measure is the Lebesgue measure on $\mathbb{R}$.
\end{example}
\end{example}
As can be seen that Fourier series and Fourier integrals are special
cases of the decomposition of the right regular representation $R_{g}$
of a unimodular locally compact group. $\int^{\oplus}$ $\Longrightarrow$
$\norm{f}=\norm{\hat{f}}$. This is a result that was done 30 years
earlier before the non abelian case. Classical function theory studies
other types of convergence, pointwise, uniform, etc. 
\begin{example}
$G=H_{3}$. $G$ is unimodular, non abelian. $\hat{G}$ is not a group. 

Irreducible representations: $\mathbb{R}\backslash\{0\}$ Schrodinger
representation, $\{0\}$ 1-d trivial representation

Decomposition: \[
R_{g}=\int_{\mathbb{R}\backslash\{0\}}^{\oplus}U_{irrep}^{h}hdh\]
For all $F\in L^{2}(G)$,\[
(U_{g}f)(e)=\int^{\oplus}U^{h}f\: hdh,\quad U^{h}\mbox{ irrep}\]
 \begin{eqnarray*}
F(g) & = & (R_{g}F)(e)\\
 & = & \int_{\mathbb{R}\backslash\{0\}}^{\oplus}e^{ih(c+bx)}f(x+a)(U^{h}F)\: hdh\\
\hat{F}(h) & = & \int_{G}(U_{g}^{h}F)dg\end{eqnarray*}
Plancherel measure: $hdh$ and the point measure $\delta_{0}$ at
zero.
\begin{example}
$G$ $ax+b$ group, non abelian. $\hat{G}$ not a group. 3 irreducible
representations: $+,-,0$ but $G$ is not unimodular.
\end{example}
\end{example}
The duality question may also be asked for discrete subgroups. This
leads to remarkable applications in automorphic functions, automorphic
forms, p-adic nubmers, compact Riemann surface, hypobolic geometry,
etc. 
\begin{example}
Cyclic group of order $n$. $G=\mathbb{Z}/n\mathbb{Z}\simeq\{0,1,\cdots,n-1\}$.
$\hat{G}=G$. This is another example where the dual group is identical
to the group itself. Let $\zeta=e^{i2\pi/n}$ be the primitive $n^{th}$-root
of unity. $k\in\mathbb{Z}_{n}$, $l=\{0,1,\ldots,n-1\}$\[
\chi_{l}(k)=e^{i\frac{2\pi kl}{n}}\]
In this case, Segal's theorem gives finite Fourier transform. $U:l^{2}(\mathbb{Z})\rightarrow l^{2}(\hat{\mathbb{Z}})$
where\[
Uf(l)=\frac{1}{\sqrt{N}}\sum_{k}\zeta^{kl}f(k)\]

\end{example}

\section{Summary of induced reprep, $d/dx$ example}

We study decomposition of group representations. Two cases: abelian
and non abelian. The non abelian case may be induced from the abelian
ones.

non abelian
\begin{itemize}
\item semi product $G=HN$ oftne $H,N$ are normal.
\item $G$ simple. $G$ does not have normal subgroups. Lie algebra does
not have any ideals. \end{itemize}
\begin{example}
$SL_{2}(\mathbb{R})$ (non compact)\[
\left(\begin{array}{cc}
a & b\\
c & d\end{array}\right),\; ad-bc=1\]
with Lie algebra \[
sl_{2}(\mathbb{R})=\{X:tr(X)=0\}\]
$sl_{2}$ is generated by \[
\left(\begin{array}{cc}
0 & 1\\
1 & 0\end{array}\right)\;\left(\begin{array}{cc}
0 & -1\\
1 & 0\end{array}\right)\;\left(\begin{array}{cc}
1 & 0\\
0 & -1\end{array}\right).\]
$\left(\begin{array}{cc}
0 & -1\\
1 & 0\end{array}\right)$ generates the one-parameter group $\left(\begin{array}{cc}
\cos t & -\sin t\\
\sin t & \cos t\end{array}\right)\simeq T$ whose dual group is $\mathbb{Z}$, where \[
\chi_{n}(g(t))=g(t)^{n}=e^{itn}.\]
May use this to induce a representation of $G$. This is called principle
series. Need to do something else to get all irreducible representations.

A theorem by Iwasawa states that simple matrix group (Lie group) can
be decomposed into \[
G=KAN\]
where $K$ is compact, $A$ is abelian and $N$ is nilpotent. For
example, in the $SL_{2}$ case, \[
SL_{2}(\mathbb{R})=\left(\begin{array}{cc}
\cos t & -\sin t\\
\sin t & \cos t\end{array}\right)\left(\begin{array}{cc}
e^{s} & 0\\
0 & e^{-s}\end{array}\right)\left(\begin{array}{cc}
1 & u\\
0 & 1\end{array}\right).\]
The simple groups do not have normal subgroups. The representations
are much more difficult.
\end{example}

\subsection{Induced representation}

Suppose from now on that $G$ has a normal abelian subgrup $N\vartriangleleft G$,
and $G=H\ltimes N$ The $N\simeq\mathbb{R}^{d}$ and $N^{*}\simeq(\mathbb{R}^{d})^{*}=\mathbb{R}^{d}$.
In this case\[
\chi_{t}(\nu)=e^{it\nu}\]
for $\nu\in N$ and $t\in\hat{N}=N^{*}$. Notice that $\chi_{t}$
is a 1-d irreducible representation on $\mathbb{C}$. 

Let $\mathcal{H}_{t}$ be the space of functions $f:G\rightarrow\mathbb{C}$
so that\[
f(\nu g)=\chi_{t}(\nu)f(g).\]
On $H_{t}$, define inner product so that\[
\norm{f}_{H_{t}}^{2}:=\int_{G}\abs{f(g)}^{2}=\int_{G/N}\norm{f(g)}^{2}dm\]
where $dm$ is the invariant measure on $N\backslash G\simeq H$.

Define $U_{t}=ind_{N}^{G}(\chi_{t})\in Rep(G,H_{t})$. Define $U_{t}(g)f(x)=f(xg)$,
for $f\in H_{t}$. Notice that the representation space of $\chi_{t}$
is $\mathbb{C}$, 1-d Hilbert space, however, the representation space
of $U_{t}$ is $\mathcal{H}_{t}$ which is infinite dimensional. $U_{t}$
is a family of irreducible representations indexed by $t\in N\simeq\hat{N}\simeq\mathbb{R}^{d}$.
\begin{note}
Another way to recognize induced representations is to see these functions
are defined on $H$, not really on $G$.

Define the unitary transformation $W_{t}:\mathcal{H}_{t}\rightarrow L^{2}(H)$.
Notice that $H\simeq N\backslash G$ is a group, and it has an invariant
Haar measure. By uniqueness on the Haar measure, this has to be $dm$.
It would be nice to cook up the same space $L^{2}(H)$ so that all
induced representations indexed by $t$ act on it. In other words,
this Hilbert space $L^{2}(H)$ does not depend on $t$. $W_{t}$ is
defined as\[
W_{t}F_{t}(h)=F_{t}(h).\]

So what does the induced representation look like in $L^{2}(H)$ then?
Recall by definition that \[
U_{t}(g):=W_{t}ind_{\chi_{t}}^{G}(g)W_{t}^{*}\]
and the following diagram commutes.

\[\xymatrix{
\mathcal{H}_{t} \ar[r]^{ind_{\chi_{t}}^{G}} \ar[d]_{W_{t}} & \mathcal{H}_{t} \ar[d]^{W_{t}} \\
L^{2}(H) \ar[r]^{U_{t}} & L^{2}(H) 
}\]

Let $f\in L^{2}(H)$. \begin{eqnarray*}
U_{t}(g)f(h) & = & W_{t}ind_{\chi_{t}}^{G}(g)W_{t}^{*}f(h)\\
 & = & (ind_{\chi_{t}}^{G}(g)W_{t}^{*}f)(h)\\
 & = & (W_{t}^{*}f)(hg).\end{eqnarray*}
Since $G=H\ltimes N$, $g$ is uniquely decomposed into $g=g_{N}g_{H}$.
Hence $hg=hg_{N}g_{H}=g_{N}g_{N}^{-1}hg_{N}g_{H}=g_{N}\tilde{h}g_{H}$
and \begin{eqnarray*}
U_{t}(g)f(h) & = & (W_{t}^{*}f)(hg)\\
 & = & (W_{t}^{*}f)(g_{N}\tilde{h}g_{H})\\
 & = & \chi_{t}(g_{N})(W_{t}^{*}f)(\tilde{h}g_{H})\\
 & = & \chi_{t}(g_{N})(W_{t}^{*}f)(g_{N}^{-1}hg_{N}g_{H})\end{eqnarray*}
This last formula is called the Mackey machine. 

The Mackey machine does not cover many important symmetry groups in
physics. Actually most of these are simple groups. However it can
still be applied. For example, in special relativity theory, we have
the Poincare group $\mathcal{L}\ltimes\mathbb{R}^{4}$ where $\mathbb{R}^{4}$
is the normal subgroup. The baby version of this is when $\mathcal{L}=SL_{2}(\mathbb{R})$.
V. Bargman fomulated this baby version. Wigner poineered the Mackey
machine, long before Mackey was around. 

Once we get unitary representations, differentiate it and get self-adjoint
algebra of operators (possibly unbounded). These are the observables
in quantum mechanics.\end{note}
\begin{example}
$\mathbb{Z}\subset\mathbb{R}$, $\hat{\mathbb{Z}}=T$. $\chi_{t}\in T$,
$\chi_{t}(n)=e^{itn}$. Let $\mathcal{H}_{t}$ be the space of functions
$f:\mathbb{R}\rightarrow\mathbb{C}$ so that \[
f(n+x)=\chi_{t}(n)f(x)=e^{int}f(x).\]
Define inner product on $\mathcal{H}_{t}$ so that \[
\norm{f}_{\mathcal{H}_{t}}^{2}:=\int_{0}^{1}\abs{f(x)}^{2}dx.\]
Define $ind_{\chi_{t}}^{\mathbb{R}}(y)f(x)=f(x+y)$. Claim that $\mathcal{H}_{t}\simeq L^{2}[0,1]$.
The unitary transformation is given by $W_{t}:\mathcal{H}_{t}\rightarrow L^{2}[0,1]$
\[
(W_{t}F_{t})(x)=F_{t}(x).\]
Let's see what $ind_{\chi_{t}}^{\mathbb{R}}(y)$ looks like on $L^{2}[0,1]$.
For any $f\in L^{2}[0,1]$, \begin{eqnarray*}
(W_{t}ind_{\chi_{t}}^{G}(y)W_{t}^{*}f)(x) & = & (ind_{\chi_{t}}^{G}(y)W_{t}^{*}f)(x)\\
 & = & (W_{t}^{*}f)(x+y)\end{eqnarray*}
Since $y\in\mathbb{R}$ is uniquely decomposed as $y=n+x'$ for some
$x'\in[0,1)$, therefore\begin{eqnarray*}
(W_{t}ind_{\chi_{t}}^{G}(y)W_{t}^{*}f)(x) & = & (W_{t}^{*}f)(x+y)\\
 & = & (W_{t}^{*}f)(x+n+x')\\
 & = & (W_{t}^{*}f)(n+(-n+x+n)+x')\\
 & = & \chi_{t}(n)(W_{t}^{*}f)((-n+x+n)+x')\\
 & = & \chi_{t}(n)(W_{t}^{*}f)(x+x')\\
 & = & e^{itn}(W_{t}^{*}f)(x+x')\end{eqnarray*}
\end{example}
\begin{note}
Are there any functions in $\mathcal{H}_{t}$? Yes, for example, $f(x)=e^{itx}$.
If $f\in\mathcal{H}_{t}$, $\abs{f}$ is 1-periodic. Therefore $f$
is really a function defined on $\mathbb{Z}\backslash\mathbb{R}\simeq[0,1]$.
Such a function has the form \[
f(x)=(\sum c_{n}e^{i2\pi nx})e^{itx}=\sum c_{n}e^{i(2\pi n+t)x}.\]
Any 1-periodic function $g$ satisfies the boundary condition $g(0)=g(1)$.
$f\in\mathcal{H}_{t}$ has a modified boundary condition where $f(1)=e^{it}f(0)$. 
\end{note}

\section{Connection to Nelson's spectral theory}

In Nelson's notes, a normal representation has the form (counting
multiplicity)\[
\rho=\sum^{\oplus}n\pi\big|_{H_{n}},\; H_{n}\perp H_{m}\]
where \[
n\pi=\pi\oplus\cdots\oplus\pi\:(\mbox{n times})\]
is a representation acting on the Hilbert space \[
\sum^{\oplus}K=l_{\mathbb{Z}_{n}}^{2}\otimes K.\]
In matrix form, this is a diagonal matrix with $\pi$ repeated on
the diagonal $n$ times. $n$ could be $1,2,\ldots,\infty$. We apply
this to group representations.

Locally compact group can be divided into the following types.
\begin{itemize}
\item abelian
\item non-abelian: unimodular, non-unimodular
\item non-abelian: Mackey machine, semidirect product e.g. $H_{3}$, $ax+b$;
simple group $SL_{2}(\mathbb{R})$. Even it's called simple, ironically
its representation is much more difficult than the semidirect product
case.
\end{itemize}
We want to apply these to group representations.

Spectral theorem says that given a normal operator $A$, we may define
$f(A)$ for quite a large class of functions, actually all measurable
functions. One way to define $f(A)$ is to use the multiplication
version of the spectral theorem, and let \[
f(A)=\mathcal{F}f(\hat{A})\mathcal{F}^{-1}.\]
The other way is to use the projection-valued meaure version of the
spectral theorem, write \begin{eqnarray*}
A & = & \int\lambda P(d\lambda)\\
f(A) & = & \int f(\lambda)P(d\lambda).\end{eqnarray*}
The effect is $\rho$ is a representation of the abelian algebra of
measurable functions onto operators action on some Hilbert space.
\begin{eqnarray*}
\rho:f\mapsto\rho(f) & = & f(A)\\
\rho(fg) & = & \rho(f)\rho(g)\end{eqnarray*}
To imitate Fourier transform, let's call $\hat{f}:=\rho(f)$. Notice
that $\hat{f}$ is the multiplication operator.
\begin{example}
$G=(\mathbb{R},+)$, group algebra $L^{1}(\mathbb{R})$. Define Fourier
transform \[
\hat{f}(t)=\int f(x)e^{-itx}dx.\]
$\{e^{itx}\}_{t}$ is a family of 1-dimensional irreducible representation
of $(\mathbb{R},+)$. 
\begin{example}
Fix $t$, $H=\mathbb{C}$, $\rho(\cdot)=e^{it(\cdot)}\in Rep(G,H)$.
From the group representation $\rho$, we get a group algebra representation
$\tilde{\rho}\in Rep(L^{1}(\mathbb{R}),H)$ defined by\[
\tilde{\rho}(f)=\int f(x)\rho(x)dx=\int f(x)e^{itx}dx\]
It follows that \begin{eqnarray*}
\hat{f}(\rho) & := & \tilde{\rho}(f)\\
\widehat{f\star g} & = & \widehat{f\star g}=\hat{f}\hat{g}\end{eqnarray*}
i.e. Fourier transform of $f\in L^{1}(\mathbb{R})$ is a representation
of the group algebra $L^{1}(\mathbb{R})$ on to the 1-dimensional
Hilbert space $\mathbb{C}$. The range of Fourier transform in this
case is 1-d abelian algebra of multiplication operators, multiplication
by complex numbers. 
\begin{example}
$H=L^{2}(\mathbb{R})$, $\rho\in Rep(G,H)$ so that \[
\rho(y)f(x):=f(x+y)\]
i.e. $\rho$ is the right regular representation. The representation
space $H$ in this case is infinite dimensional. From $\rho$, we
get a group algebra representation $\tilde{\rho}\in Rep(L^{1}(\mathbb{R}),H)$
where \[
\tilde{\rho}(f)=\int f(y)\rho(y)dy.\]
Define \[
\hat{f}(\rho):=\hat{\rho}(f)\]
then $\hat{f}(\rho)$ is an operator acting on $H$.\begin{alignat*}{1}
\hat{f}(\rho)g=\tilde{\rho}(f)g & =\int f(y)\rho(y)g(\cdot)dy\\
 & =\int f(y)(R_{y}g)(\cdot)dy\\
 & =\int f(y)g(\cdot+y)dy.\end{alignat*}
If we have used the left regular representation, instead of the right,
then \begin{alignat*}{1}
\hat{f}(\rho)g=\tilde{\rho}(f)g & =\int f(y)\rho(y)g(\cdot)dy\\
 & =\int f(y)(L_{y}g)(\cdot)dy\\
 & =\int f(y)g(\cdot-y)dy.\end{alignat*}
Hence $\hat{f}(\rho)$ is the left or right convolution operator.
\end{example}
\end{example}
\end{example}
Back to the general case. Given a locally compact group $G$, form
the group algebra $L^{1}(G)$, and define the left and right convolutions
as

\begin{alignat*}{1}
(\varphi\star\psi)(x) & =\int\varphi(g)\psi(g^{-1}x)d_{L}g=\int\varphi(g)(L_{g}\psi)d_{L}g\\
(\varphi\star\psi)(x) & =\int\varphi(xg)\psi(g)d_{R}g=\int(R_{g}\varphi)\psi(g)d_{R}g\end{alignat*}
Let $\rho(g)\in Rep(G,H)$, define $\tilde{\rho}\in Rep(L^{1}(G),H)$
given by \[
\tilde{\rho}(\psi):=\int_{G}\psi(g)\rho(g)dg\]
and write \[
\hat{\psi}(\rho):=\tilde{\rho}(\psi).\]
$\hat{\psi}$ is an analog of Fourier transform. If $\rho$ is irreducible,
the operators $\hat{\psi}$ forms an abelian algebra. In general,
the range of this generalized Fourier transform gives rise to a non
abelian algebra of operators.

For example, if $\rho(g)=R_{g}$ and $H=L^{2}(G,d_{R})$, then \[
\tilde{\rho}(\psi)=\int_{G}\psi(g)\rho(g)dg=\int_{G}\psi(g)R_{g}dg\]
and\begin{alignat*}{1}
\tilde{\rho}(\psi)\varphi & =\int_{G}\psi(g)\rho(g)\varphi dg=\int_{G}\psi(g)(R_{g}\varphi)dg\\
 & =\int_{G}\psi(g)\varphi(xg)dg\\
 & =(\varphi\star\psi)(x)\end{alignat*}

\begin{example}
$G=H_{3}\sim\mathbb{R}^{3}$. $\hat{G}=\{\mathbb{R}\backslash\{0\}\}\cup\{0\}$.
$0\in\hat{G}$ corresponds to the trivial representation, i.e. $g\mapsto Id$
for all $g\in G$.\[
\rho_{h}:G\rightarrow L^{2}(\mathbb{R})\]
\[
\rho_{h}(g)f(x)=e^{ih(c+bx)}f(x+a)\simeq ind_{H}^{G}(\chi_{h})\]
where $H$ is the normal subgroup $\{b,c\}$. It is not so nice to
work with $ind_{H}^{G}(\chi_{h})$ directly, so instead, we work with
the equivalent representations, i.e. Schrodinger representation. See
Folland's book on abstract harmonic analysis.\[
\hat{\psi}(h)=\int_{G}\psi(g)\rho_{h}(g)dg\]
Notice that $\hat{\psi}(h)$ is an operator acting on $L^{2}(\mathbb{R})$.
Specifically,\begin{eqnarray*}
\hat{\psi}(h) & = & \int_{G}\psi(g)\rho_{h}(g)dg\\
 & = & \iiint\psi(a,b,c)e^{ih(c+bx)}f(x+a)dadbdc\\
 & = & \iint\left(\int\psi(a,b,c)e^{ihc}dc\right)f(x+a)e^{ihbx}dadb\\
 & = & \iint\hat{\psi}(a,b,h)f(x+a)e^{ihbx}dadb\\
 & = & \int\left(\int\hat{\psi}(a,b,h)e^{ihbx}db\right)f(x+a)da\\
 & = & \int\hat{\psi}(a,hx,h)f(x+a)da\\
 & = & \left(\hat{\psi}(\cdot,h\cdot,h)\star f\right)(x)\end{eqnarray*}
Here the $\hat{\psi}$ on the right hand side in the Fourier transform
of $\psi$ in the usual sense. Therefore the operator $\hat{\psi}(h)$
is the one so that \[
L^{2}(\mathbb{R})\ni f\mapsto\left(\hat{\psi}(\cdot,h\cdot,h)\star f\right)(x).\]
If $\psi\in L^{1}(G)$, $\hat{\psi}$ is not of trace class. But if
$\psi\in L^{1}\cap L^{2}$, then $\hat{\psi}$ is of trace class.\[
\int_{\mathbb{R}\backslash\{0\}}^{\oplus}tr\left(\hat{\psi}^{*}(h)\hat{\psi}(h)\right)d\mu=\int\abs{\psi}^{2}dg=\int\bar{\psi}\psi dg\]
where $\mu$ is the Plancherel measure. 

If the group $G$ is non unimoduler, the direct integral is lost (not
orthogonal). These are related to coherent states from physics, which
is about decomposing Hilbert into non orthogonal pieces.
\end{example}
Important observables in QM come in pairs (dual pairs). For example,
position - momentum; energy - time etc. The Schwartz space $S(\mathbb{R})$
has the property that $\widehat{S(\mathbb{R})}=S(\mathbb{R})$. We
look at the analog of the Schwartz space. $h\mapsto\hat{\psi}(h)$
should decrease faster than any polynomials. 

Take $\psi\in L^{1}(G)$, $X_{i}$ in the Lie algebra, form $\triangle=\sum X_{i}^{2}$.
Require that\[
\triangle^{n}\psi\in L^{1}(G),\;\psi\in C^{\infty}(G).\]
For $\triangle^{n}$, see what happens in the transformed domain.
Notice that\[
\frac{d}{dt}\big|_{t=0}\left(R_{e^{tX}}\psi\right)=\tilde{X}\psi\]
where $X\mapsto\tilde{X}$ represents the direction vector $X$ as
a vector field.

Let $G$ be any Lie group. $\varphi\in C_{c}^{\infty}(G)$, $\rho\in Rep(G,H)$.
\[
d\rho(X)v=\int(\tilde{X}\varphi)(g)\rho(g)vdg\]
where \[
v=\int\varphi(g)\rho(g)wdg=\rho(\varphi)w.\mbox{ generalized convolution}\]
If $\rho=R$, the n\[
v=\int\varphi(g)R(g)\]
\[
\tilde{X}(\varphi\star w)=(X\varphi)\star w.\]

\begin{example}
$H_{3}$\begin{eqnarray*}
a & \rightarrow & \frac{\partial}{\partial a}\\
b & \mapsto & \frac{\partial}{\partial b}\\
c & \mapsto & \frac{\partial}{\partial c}\end{eqnarray*}
get standard Laplace operator. $\{\rho_{h}(\varphi)w\}\subset L^{2}(\mathbb{R})$
. $"="$ due to Dixmier. $\{\rho_{h}(\varphi)w\}$ is the Schwartz
space. \[
\left(\frac{d}{dx}\right)^{2}+(ihx)^{2}+(ih)^{2}=\left(\frac{d}{dx}\right)^{2}-(hx)^{2}-h^{2}\]
Notice that\[
-\left(\frac{d}{dx}\right)^{2}+(hx)^{2}+h^{2}\]
is the Harmonic oscilator. Spectrum = $h\mathbb{Z}_{+}$.
\end{example}

\chapter{Unbounded Operators}

\section{Unbounded operators, definitions}

We study unitary representation of Lie groups, one there is a representation,
it can be differentiated and get a representation of the Lie algebra.
These operators are the observables in quantum mechanics. They are
possibly unbounded.

\subsection{Domain}
\begin{example}
$d/dx$ and $M_{x}$ in QM, acting on $L^{2}$ with dense domain the
Schwartz space.
\end{example}
An alternative way to get a dense domain, a way that works for all
representations, is to use Garding space, or $C^{\infty}$ vectors.
Let $u\in H$ and define\[
u_{\varphi}:=\int\varphi(g)U_{g}udg\]
where $\varphi\in C_{c}^{\infty}$, and $U\in Rep(G,H)$. Let $\varphi_{\epsilon}$
be an approximation of identity. Then for functions on $\mathbb{R}^{d}$,
$\varphi_{\epsilon}\star\psi\rightarrow\psi$ as $\epsilon\rightarrow0$;
and for $C^{\infty}$ vectors, $u_{\varphi_{\epsilon}}\rightarrow u$,
as $\epsilon\rightarrow0$ in $H$ i.e. in the $\norm{\cdot}_{H}$-
norm. The set $\{u_{\varphi}\}$ is dense in $H$. It is called the
Garding space, or $C^{\infty}$ vectors, or Schwartz space. Differentiate
$U_{g}$ and get a Lie algebra representation\[
\rho(X):=\frac{d}{dt}\big|_{t=0}U(e^{tX})=dU(X).\]

\begin{lem}
$\norm{u_{\varphi_{\epsilon}}-u}\rightarrow0$, as $\epsilon\rightarrow0$.\end{lem}
\begin{proof}
Since $u_{\varphi_{\epsilon}}-u=\int\varphi_{\epsilon}(g)(u-U_{g}u)dg$,
we have\begin{eqnarray*}
\norm{u_{\varphi_{\epsilon}}-u} & = & \norm{\int\varphi_{\epsilon}(g)(u-U_{g}u)dg}\\
 & \leq & \int\varphi_{\epsilon}(g)\norm{u-U_{g}u}dg\end{eqnarray*}
where we used the fact that $\int\varphi_{\epsilon}=1$. Notice that
we always assume the representations are norm continuous in the $g$
variable, otherwise it is almost impossible to get anything interesting.
i.e. asuume $U$ being strongly continuous. So for all $\delta>0$,
there is a neigborhood $\mathcal{O}$ of $e\in G$ so that ${u-U_{g}u}<\delta$
for all $g\in\mathcal{O}$. Choose $\epsilon_{\delta}$ so that $\varphi_{\epsilon}$
is upported in $\mathcal{O}$ for all $\epsilon<\epsilon_{\delta}$.
Then the statement is proved.\end{proof}
\begin{note}
There are about 7 popular kernels in probability theory. Look at any
book on probability theory. One of them is the Cauchy kernel $\frac{x}{\pi}\frac{1}{1+x^{2}}$. 
\begin{note}
Notice that not only $u_{\varphi}$ is dense in $H$, their derivatives
are also dense in $H$.
\end{note}
\end{note}

\section{Self-adjoint extensions}

In order to apply spectral theorem, one must work with self adjoint
operators including the unbounded ones. Some examples first.

In quantum mechanics, to understand energy levels of atoms and radiation,
the engegy level comes from discrete packages. The interactions are
givn by Colum Law where\[
H=-\triangle_{\vec{r}}+\frac{c_{jk}}{\norm{r_{j}-r_{k}}}\]
and Laplacian has dimension $3\times\#(\mbox{electrons})$.

In Schrodinger's wave mechanics, one needs to solve for $\psi(r,t)$
from the equation \[
H\psi=\frac{1}{i}\frac{\partial}{\partial t}\psi.\]
If we apply spectral theorem, then $\psi(t)=e^{itH}\psi(r,t=0)$.
This shows that motion in qm is governed by unitary operators. The
two parts in Shrodinger equation are separately self-adjoint, but
justification of the sum being self-adjoint wasn't made rigorous until
1957. Kato wrote a book called perturbation theory. It is a summary
of the sum of self-adjoint operators.

In Heisenberg's matrix mechanics, he suggested that one should look
at two states and the transition probability between them. \[
\iprod{\psi_{1}}{A\psi_{2}}=\iprod{\psi_{1}(t)}{A\psi_{2}(t)},\;\forall t.\]
If $\psi(t)=e^{itH}\psi$, then it works. In Heisenberg's picture,
one looks at evolution of the observables $e^{-itH}Ae^{itH}$. In
Shrodinger's picture, one looks at evolution of states. The two point
of views are equivalent. 

Everything so far is based on application of the spectral theorm,
which requires the operators being self-adjoint in the first place.

\subsection{Self-adjoint extensions}
\begin{lem}
$R(A)^{\perp}=N(A^{*})$. $N(A^{*})^{\perp}=cl(R(A))$. 
\begin{proof}
$y\in R(A)^{\perp}$ $\Longleftrightarrow$ $\iprod{Ax}{y}=0$, $\forall x\in D(A)$
$\Longleftrightarrow$ $x\mapsto\iprod{Ax}{y}$ being the zero linear
functional. Therefore by Riesz's theorem, there exists unique $y^{*}=A^{*}y$
so that $\iprod{Ax}{y}=\iprod{x}{A^{*}y}=0$, i.e. $y\in N(A^{*})$.
\end{proof}
\end{lem}
\begin{note}
In general, $(\mbox{set})^{\perp\perp}=cl(span(\mbox{set}))$.\end{note}
\begin{defn}
Define the defeciency spaces and defeciency indecies as\[
D_{+}:=R(A+iI)^{\perp}=N(A^{*}-iI)=\{v\in dom(A^{*}):A^{*}v=+iv\}\]
\[
D_{-}:=R(A-iI)^{\perp}=N(A^{*}+iI)=\{v\in dom(A^{*}):A^{*}v=-iv\}\]
\begin{eqnarray*}
d_{+} & := & dim(D_{+})\\
d_{-} & := & dim(D_{-})\end{eqnarray*}

Von Neumann's notion of defeciency space expresses the extent to which
$A^{*}$ is bigger than $A$. One wouldn't expect to get a complex
eigenvalue for a self adjoint operator. We understand $A$ not being
self adjoint by looking at its {}``wrong'' eigenvalues. This reveals
that $A^{*}$ is defined on a bigger domain. The extend that $A^{*}$
is defined on a bigger domain is refleced on the {}``wrong'' eigenvalues.
\end{defn}
\begin{tabular}{|c|c|}
\hline 
$D_{+}$ & $D_{-}$\tabularnewline
\hline
\hline 
$R(A+i)$ & $R(A-i)$\tabularnewline
\hline
\end{tabular}
\begin{defn}
Define the Caley transform \begin{eqnarray*}
C_{A}:R(A+i) & \rightarrow & R(A-i)\\
(A+i)x & \mapsto & (A-i)x\end{eqnarray*}
i.e.\[
C_{A}=(A-i)(A+i)^{-1}="\frac{A-i}{A+i}"\]
The inverse map is given by \[
A=i(1+C_{A})(1-C_{A})^{-1}="i\frac{1+C_{A}}{1-C_{A}}".\]
\end{defn}
\begin{lem}
$C_{A}$ is a partial isometry.
\begin{proof}
We check that \begin{eqnarray*}
\norm{(A+i)x}^{2} & = & \iprod{Ax+ix}{Ax+ix}\\
 & = & \norm{Ax}^{2}+\norm{x}^{2}+\iprod{Ax}{ix}+\iprod{ix}{Ax}\\
 & = & \norm{Ax}^{2}+\norm{x}^{2}+i\iprod{Ax}{x}-i\iprod{x}{Ax}\\
 & = & \norm{Ax}^{2}+\norm{x}^{2}+i\iprod{x}{Ax}-i\iprod{x}{Ax}\\
 & = & \norm{Ax}^{2}+\norm{x}^{2}\end{eqnarray*}
where the last two equations follows from $A$ being Hermitian (symmetric).
This shows that $C_{A}$ is a partial isometry.
\end{proof}
\end{lem}
$C_{A}$ is only well-defined on $R(A+i)$. Since $C_{A}$ is an isometry
from $R(A+i)$ onto $R(A-i)$, it follows that if $R(A+i)$ is dense
in $H$, then $C_{A}$ extends uniquely to a unitary operator $\tilde{C}_{A}$
on $H$, and thus $D_{+}=0$. The only way that $D_{+}$ is non zero
is that $R(A+i)$ is not dense in $H$. 

Failure of $A$ being self adjoint $\Longleftrightarrow$ failure
of $C_{A}$ being everywhere defined; $A$ Hermitian $\Longleftrightarrow$
$C_{A}$ is a partial isometry (by lemma above); Von Neumann's method:
look at extensions of $C_{A}$, and transform the result back to $A$.
\begin{thm}
(VN) $C_{A}$ extends to $H$ if and only if $d_{+}=d_{-}$.\end{thm}
\begin{example}
$d_{+}=d_{-}=1$. Let $e_{\pm}$ be corresponding eigenvalues. $e_{+}\mapsto ze_{-}$
is the unitary operator sending one to the other eigenvalue. It is
clear that $\abs{z}=1$. Hence the self adjoint extension is indexed
by $U_{1}(\mathbb{C})$.
\begin{example}
$d_{+}=d_{-}=2$, get a family of extensions indexed by $U_{2}(\mathbb{C})$.
\end{example}
\end{example}
\begin{rem}
M. Stone and Von Neumann are the two pioneers who worked at the same
period. They were born at about the same time. Stone died at 1970's
and Von Neumann died in the 1950's.\end{rem}
\begin{defn}
A conjugation is an operator $J$ so that $J^{2}=1$ and it is conjugate
linear i.e. $J(cx)=\bar{c}Jx$.\end{defn}
\begin{thm}
(VN) If there exists a conjugation $J$ such that $AJ=JA$, then $d_{+}=d_{-}$.\end{thm}
\begin{proof}
Claim that if $A$ commutes with $J$, so does $A^{*}$. Assuming
this is true, then we claim that $J:D_{+}\rightarrow D_{-}$ is a
bijection. Suppose $A^{*}v_{+}=iv_{+}$. Then\[
A^{*}(Jv_{+})=JA^{*}v_{+}=J(iv_{+})=-iJv_{+}.\]
\end{proof}
\begin{thm}
(VN) $A\subset A^{*}$, $A$ closed. Then $D(A^{*})=D(A)\oplus D_{+}\oplus D_{-}$.
\begin{proof}
It is clear that $D(A)$, $D_{+}$ and $D_{-}$ are subspaces of $D(A^{*})$.
$D_{+}\cap D_{-}=0$, since if $Ax=ix$ and $Ax=-ix$ implies $x=0$.
$D(A)\cap D_{+}=0$ as well, since if $x$ is in the intersection
then\[
\iprod{x}{Ax}=\iprod{x}{A^{*}x}=i\norm{x}^{2}\]
but $\iprod{x}{Ax}$ being a real number implies that $x=0$. Similarly,
$D(A)\cap D_{-}=0$. Therefore, $D(A)\oplus D_{+}\oplus D_{-}\subset D(A^{*})$.
To show the two sides are equal, use the graph norm of $A^{*}$ $\norm{\cdot}_{G}$,
show that $x\perp D(A)\oplus D_{+}\oplus D_{-}=0$ implies that $\norm{x}_{G}=0$.

Another proof: let $f\in D(A^{*})$ and we will decompose $f$ into
the direct sum of three parts. Since $H=R(A+i)\oplus D_{+}$, therefore
$(A^{*}+i)f$ decomposes into\[
(A^{*}+i)f=(A+i)f_{0}+s\]
where $f_{0}\in D(A)$ and $s\in D_{+}$. Write $s=2if_{+}$, for
some $f_{+}\in D_{+}$. (For example, take $f_{+}=s/2i$.) Thus we
have\begin{eqnarray*}
(A^{*}+i)f & = & (A+i)f_{0}+2if_{+}\\
 & \Updownarrow\\
A^{*}(f-f_{0}-f_{+}) & = & -i(f-f_{0}-f_{+})\end{eqnarray*}
Define $f_{-}=f-f_{0}-f_{+}$. Then $f_{-}\in D_{-}$, and \[
f=f_{0}+f_{+}+f_{-}.\]
It remains to show the decomposition is unique. Suppose $f_{0}+f_{+}+f_{-}=0$.
Then\begin{eqnarray*}
A^{*}(f_{0}+f_{+}+f_{-}) & = & Af_{0}+if_{+}-if_{-}\\
i(f_{0}+f_{+}+f_{-}) & = & if_{0}+if_{+}+if_{-}\end{eqnarray*}
Notice that the left hand side of the above equations are zero. Add
the two equations together, we get $(A+i)f_{0}+2if_{+}=0$. Since
$H=R(A+i)\oplus D_{+}$, $f_{+}=0$. Similarly, $f_{-}=0$ and $f_{0}=0$
as well. 
\end{proof}
\end{thm}
\begin{example}
$A=d/dx$ on $L^{2}[0,1]$. Integration by parts shows that $A\subset A^{*}$.
\end{example}

\chapter*{Appendix}

\section*{semi-direct product}
\begin{thm}
$\begin{cases}
G=HK\\
H,K\mbox{ commute}\\
H\cap K=1\end{cases}$ $\Longleftrightarrow$ $G\simeq H\times K$.
\begin{proof}
$\Longrightarrow$ Define $\varphi:H\times K\rightarrow G$ by $(h,k)\mapsto hk$.
Then $\varphi$ is a homomorphism, because $H,K$ commute; $\varphi$is
1-1, since if $hk=1$, then $k=h^{-1}\in H\cap K$, therefore $k=1$.
This implies that $h=1$. Moreover, $\varphi$ is onto, since $G=HK$.
$\Longleftarrow$ is trivial.
\end{proof}
\end{thm}
We modify the above theorem. 

Suppose$\begin{cases}
G=HK\\
H\vartriangleleft G\\
H\cap K=1\end{cases}$. Define $H\rtimes K=\{(h,k):h\in H,k\in K\}$ with multiplication
given by\[
(h,k_{1})(h_{2},k_{2}):=(h_{1}k_{1}h_{2}k_{1}^{-1},k_{1}k_{2}).\]
This turns $H\rtimes K$ into a group, so that $G\simeq H\rtimes K$.
If $H,K$ commute, then $k_{1}h_{2}k_{1}^{-1}=h_{2}$ and we are back
to the product group. The direct product is always abelian. Here $K$
acts on $H$ by conjugation in this case. 

Turn it around, start with two groups $H,K$ with a map $\varphi:K\rightarrow Aut(H)$,
build a bigger group $G=\tilde{H}\tilde{K}$.
\begin{thm}
The following are equivalent.
\begin{enumerate}
\item $\begin{cases}
H,K\mbox{ groups}\\
K\mbox{ acts on }H\mbox{ by }\varphi\end{cases}$
\item $\begin{cases}
G=\tilde{H}\tilde{K}\\
\tilde{H}\vartriangleleft G\\
\tilde{H}\simeq H,\tilde{K}\simeq K\\
\tilde{H}\cap\tilde{K}=1\end{cases}$ Morever, $K$ acts on $H$ via $\varphi$ translates to $\tilde{K}$
acts on $\tilde{H}$ by conjugation.
\end{enumerate}
\end{thm}
Once we get $G=\tilde{H}\tilde{K},$from previous discussion, we have
$G\simeq\tilde{H}\rtimes\tilde{K}$. $\tilde{H}\rtimes\tilde{K}$
is unique up to isomorphism. (consider the ax+b group embeded into
a matrix group with dimension greater than $2\times2$.)

How does it work? 

Define $G=\{(h,k):h\in H,k\in K\}$ with multiplication and inverse
given by\begin{eqnarray*}
(h_{1},k_{1})(h_{2},k_{2}) & = & (h_{1}\varphi_{k_{1}}(h_{2}),k_{1}k_{2})\\
(h,k)^{-1} & = & (\varphi_{k^{-1}}(h^{-1}),k^{-1})\end{eqnarray*}
Then $G$ is a group.

check:
\begin{itemize}
\item the inverse is correct\begin{eqnarray*}
(a,b)(\varphi_{b^{-1}}(a^{-1}),b^{-1}) & = & (a\varphi_{b}(\varphi_{b^{-1}}(a^{-1})),bb^{-1})\\
 & = & (a\varphi_{b}\circ\varphi_{b^{-1}}(a^{-1}),bb^{-1})\\
 & = & (a\varphi_{1}(a^{-1}),1)\\
 & = & (aa^{-1},1)\\
 & = & (1,1)\end{eqnarray*}

\item $H\simeq\tilde{H}=\{(h,1):h\in H\}$ is a subgroup in $G$\begin{eqnarray*}
(a,1)(b,1) & = & (ab,1)\\
(a,1)^{-1} & = & (\varphi_{1^{-1}}(a^{-1}),1^{-1})\\
 & = & (a^{-1},1)\end{eqnarray*}

\item $K\simeq\tilde{K}=\{(1,k):k\in K\}$ is a subgroup in $G$\begin{eqnarray*}
(1,a)(1,b) & = & (1,ab)\\
(1,b)^{-1} & = & (1,b^{-1})\end{eqnarray*}

\item $\tilde{H}\cap\tilde{K}=\{1\}$ obvious
\item $\tilde{H}$ is normal in $G$ \begin{eqnarray*}
(x,y)(h,1)(x,y)^{-1} & = & (x\varphi_{y}(h),y)(\varphi_{y^{-1}}(x^{-1}),y^{-1})\\
 & = & (x\varphi_{y}(h)\varphi_{y}(\varphi_{y^{-1}}(x^{-1})),yy^{-1})\\
 & = & (x\varphi_{y}(h)x^{-1},1)\end{eqnarray*}

\item $\tilde{K}$ acts on $\tilde{H}$ by conjugation\begin{eqnarray*}
(1,y)(h,1)(1,y)^{-1} & = & (\varphi_{y}(h),y)(1,y^{-1})\\
 & = & (\varphi_{y}(h),1)\end{eqnarray*}

\end{itemize}
\newpage{}

\bibliographystyle{plain}
\bibliography{FAnotes}

\end{document}